\documentclass{amsart}

\usepackage{fullpage}
\usepackage[utf8]{inputenc}
\usepackage[T1]{fontenc}
\usepackage{graphicx}
\usepackage{amsmath,amsfonts,amssymb,mathtools}
\usepackage{stmaryrd}
\usepackage{graphicx}
\usepackage{mathrsfs,dsfont}
\usepackage{amsmath,amsthm,amsthm}
\usepackage{soul}

\usepackage{hyperref}

\usepackage{enumerate}

\usepackage{color}

\newcommand{\E}{\mathbb{E}}
\newcommand{\R}{\mathbb{R}}
\newcommand{\N}{\mathbb{N}}
\newcommand{\D}{\mathcal{D}} 

\newcommand{\calA}{\mathcal{A}}
\newcommand{\calB}{\mathcal{B}}
\newcommand{\calC}{\mathcal{C}}
\newcommand{\calD}{\mathcal{D}}

\newtheorem{theo}{Theorem}[section]

\newtheorem{rem}[theo]{Remark}
\newtheorem{prop}[theo]{Property}
\newtheorem{propo}[theo]{Proposition}
\newtheorem{lemma}[theo]{Lemma}

\newtheorem{hyp}[theo]{Assumption}

\newcommand{\blue}{\textcolor{blue}}
\begin{document}

\title{
Kolmogorov equations and weak order analysis for SPDEs with nonlinear diffusion coefficient}

\author{Charles-Edouard Br\'ehier}
\address{Univ Lyon, CNRS, Université Claude Bernard Lyon 1, UMR5208, Institut Camille Jordan, F-69622 Villeurbanne, France}
\email{brehier@math.univ-lyon1.fr}

\author{Arnaud Debussche}
\address{Univ Rennes, CNRS, IRMAR - UMR 6625, F-
35000 Rennes, France}
\email{arnaud.debussche@ens-rennes.fr}

\date{}


\thanks{This work was started during the Fall semester 2015, while A. Debussche was in residence at the Mathematical Sciences Research Institute in Berkeley, California, supported by the National Science Foundation under Grant No. DMS-1440140. A. Debussche benefits from the support of the French government “Investissements d’Avenir” program ANR-11-LABX- 0020-01. The authors warmly thank the anonymous referee for the very useful comments which helped to improve our manuscript.
}

\begin{abstract}
We provide new regularity results for the solutions of the Kolmogorov equation associated to a SPDE with nonlinear diffusion coefficient and a Burgers type nonlinearity.
This generalizes previous results in the simpler cases of additive or affine noise. The basic tool is a discrete version of a 
two sided stochastic integral which allows a new formulation for the derivatives of these solutions. 
We show that this can be used to generalize the weak order analysis performed in \cite{Debussche:11}. 
The tools we develop are very general and can be used to study many other examples of applications.
\end{abstract}

\maketitle

\section{Introduction}

The Kolmogorov equation associated to a stochastic equation is a fundamental object. It is important to have a good understanding of this equation since many properties of the stochastic 
equation can be derived. For instance, it may be used to obtain uniqueness results - in the weak or strong sense - using ideas initially 
developped by Stroock and Varadhan \cite{stroock-varadhan} or the so-called "It\^o Tanaka" trick widely used by F. Flandoli and co-
authors, see for instance \cite{flandoli-gubinelli-priola}. Also, it is the basic tool in the weak order analysis of stochastic equations, see 
\cite{Talay:86}.

For Stochastic Partial Differential Equations (SPDEs), the associated 
Kolmogorov equation is not a standard object since it is a partial differential equations for an unknown depending on time and on an infinite dimensional variable. In the case of an additive noise, it has been the object of several studies, see \cite{Cerrai:01}, \cite{DPZ3}, \cite{DaPrato:04}, \cite{Krylov_Rockner_Zabczyck:99}, \cite{Rockner_Sobol:06} and the references therein.
But for general diffusion coefficients, very little is known. In \cite{DaPrato:07}, strict solutions are constructed but the assumptions are 
extremely strong and the result is of little interest in the applications. 

In this work, we consider a parabolic semilinear Stochastic Partial Differential Equation (SPDE) of the following form:
\begin{equation}\label{eq:SPDE_intro}
dX_t=AX_tdt+G(X_t)dt+\sigma(X_t)dW_t,
\end{equation}
where $W$ is a cylindrical Wiener process on a separable infinite dimensional Hilbert space $H$. Typically, $H$ is the space of square integrable functions on an open, bounded, interval in $\R$ so that the SPDE is driven by a space time white noise.

We wish to study regularity properties of the solutions of the associated Kolmogorov equation. The main application we have in 
mind is the weak order analysis of a Euler scheme applied to \eqref{eq:SPDE_intro}.
This has been the subject of many articles in the last decade, see \cite{Andersson_Larsson:16}, 
\cite{Brehier:14}, \cite{BrehierKopec:13}, \cite{Debussche:11}, \cite{Debussche_Printems:09}, \cite{Geissert_Kovacs_Larsson:09}, \cite{Hausenblas:03}, \cite{Kovacs_Larsson_Lindgren:12}, \cite{Kovacs_Larsson_Lindgren:13}, \cite{Tambue_Ngnotchouye:16}, \cite{Wang:16}, \cite{Wang_Gan:13}. In all these articles, the method is a  generalization of the finite dimensional proof initially used in \cite{Talay:86} (see also the monographs~\cite{Kloeden_Platen:92} and~\cite{Milstein_Tretyakov:04} for further references) and based on the Kolmogorov equation associated to \eqref{eq:SPDE_intro}. These results are restricted to the case of a $\sigma$ satisfying very strong assumptions.

\medskip

Thus our first aim is to obtain new regularity estimates 
on the transition semigroup $(P_t)_{t\ge}$. When \eqref{eq:SPDE_intro} has a unique solution (which is the case in the present article), denoted by $(X(t,x))_{t\ge 0}$, it is defined by 
\begin{equation}\label{eq:u_intro}
u(t,x)=P_t\varphi(x)=\E(\varphi(X(t,x))),
\end{equation}
where $\varphi$ is a bounded borelian function on $H$. The function $u$ formally satisfies the Kolmogorov equation:
\begin{equation}\label{eq:Kolmogorov_intro}
\frac{du}{dt}(t,x)=\frac12{\rm Tr} \left(\sigma(x)\sigma^*(x)D^2u(t,x)\right) +\langle Ax+G(x),Du(t,x)\rangle,\quad u(0,x)=\varphi(x).
\end{equation}
As usual, we have identified the first order derivative of $u$ with respect to $x$ and its gradient in $H$ and the second order derivative with the Hessian. The inner product in 
$H$ is denoted by $\langle\cdot,\cdot\rangle$.

\bigskip

Our arguments are general and can be applied in various situations. However, in order to concentrate on the new arguments, we consider a prototype example. Namely, 
we take three functions $\tilde{F}_1,\tilde{F}_2,\sigma:\R\to\R$, and consider the following stochastic partial differential equation on the interval $(0,1)$ with 
Dirichlet boundary conditions and driven by a space time white noise:
$$
\left\{
\begin{array}{l}
dX= (\partial_{\xi\xi} X + +\tilde{F}_1(X)+ \partial_\xi \tilde{F}_2(X))dt+\tilde\sigma(X)dW,\; t>0, \; x\in (0,1),\\
X(0,t)=X(1,t)=0,\\
X(\xi,0)=x(\xi).
\end{array}
\right.
$$
The initial data $x$ is given in $L^2(0,1)$ and $W$ is a cylindrical Wiener process (see \cite{DPZ}). This equation can be rewritten in the abstract form \eqref{eq:SPDE_intro} classically. Indeed, we define $H=L^2(0,1)$ with norm $|\cdot|$, $A=\partial_{\xi\xi}$ on the domain $D(A)=H^2(0,1)\cap H^1_0(0,1)$, and the Nemytskii operators:
\begin{equation*}
F_i(x)=\tilde{F}_i\bigl(x(\cdot)\bigr),\;\sigma(x)h=\tilde\sigma\bigl(x(\cdot)\bigr)h(\cdot),\; x\in H,\; h\in H.
\end{equation*}

Assuming that $\tilde{F}_1$, $\tilde{F}_2$, and $\tilde \sigma$ are bounded, this defines $F_1,F_2:H\to H$ and $\sigma:H\to {\mathcal L}(H)$, where ${\mathcal L}(H)$ is the space of bounded linear operators on $H$. Below, we assume that  $\tilde{F}_1$, $\tilde{F}_2$ and $\tilde \sigma$ are functions of class $C^3$, which are bounded and have bounded derivatives. However, it is well-known that $F_1$, $F_2$ and $\sigma$ do not inherit these regularity properties on $H$. The control of their derivatives requires the use of $L^p$ norms.

Finally, setting $B=\partial_\xi$ on $H^1(0,1)$ and $G=F_1+BF_2$, we obtain an equation in the abstract form \eqref{eq:SPDE_intro} above.

Global existence and uniqueness of a solution $X\in L^2(\Omega;C([0,T];H))$ follow from standard arguments (see \cite{DPZ} for instance). Indeed, we have boundedness and Lispchitz continuity properties on the coefficients $G$ and $\sigma$. Thus the transition semigroup can be defined by the formula \eqref{eq:u_intro}.

\bigskip

The regularity results which are required for the numerical analysis and which we obtain in this article have roughly the following form, under appropriate assumptions on $\varphi$: for $t\in(0,T)$
\begin{equation}\label{eq:DuD2u_intro}
\begin{gathered}
|Du(t,x)\cdot h|\le C(T,\varphi) t^{-\alpha} |(-A)^{-\alpha}h|,\\
|D^2u(t,x)\cdot (h,k)| \le C(T,\varphi) t^{-(\beta+\gamma)}  |(-A)^{-\beta}h||(-A)^{-\gamma}k|,
\end{gathered}
\end{equation}
where $(-A)^{-\alpha}$ denotes a negative power (for $\alpha>0$) of the linear operator $-A$. We do not make precise which $L^p$ norms appear on the right-hand side in~\eqref{eq:DuD2u_intro}. Precise and rigorous statements are given in Section~\ref{sec:results_Kolmogorov}.

Note that these regularity results are natural. They hold for instance in the case $G=0$, $\sigma=0$ for any $\alpha,\beta,\gamma\ge 0$ thanks to the regularization properties of the heat semi-group. Using elementary arguments (differentiation inside the expectation, control of the derivative processes using It\^o formula and Gronwall inequalities), see for instance~\cite{Andersson_Hefter_Jentzen_Kurniawan:16}, \cite{Debussche:11}, one can consider the case 
when the diffusion coefficient $\sigma$ is constant - additive noise case. Then the estimate above holds for $\alpha\in[0,1)$, and $\beta,\gamma\in[0,1)$ such that $\beta+\gamma<1$. The case of an affine $\sigma$ 
is also treated in the above references but then we impose $\alpha,\beta,\gamma\in[0,\frac12)$. When the diffusion coefficient $\sigma$ is nonlinear (the so-called multiplicative noise case), the results obtained so far in the literature are not satisfactory:  the extra restriction $\beta+\gamma<\frac12$ is imposed. This is not sufficient for the applications. For the weak order analysis, we need to take
$\beta+\gamma$ arbitrarily close to $1$.

Also, the right hand side of~\eqref{eq:Kolmogorov_intro} is well defined only if one is able to get \eqref{eq:DuD2u_intro} for $\alpha\in [0,1)$, $\beta,\gamma\in [0,\frac12)$ with 
$\beta+\gamma>\frac12$. This is important to prove existence of strict solutions to this Kolmogorov equation and thus to generalize results available in the case of additive noise. 

In this article, we introduce a new approach to obtain such results. Our first main contribution in this article is to prove that in~\eqref{eq:DuD2u_intro} one may take $\alpha\in[0,1)$ and $\beta,\gamma\in[0,\frac12)$, in the multiplicative noise case, for SPDEs of the type of~\eqref{eq:SPDE_intro}.


{At a formal level}, our strategy is based on new expressions for the first and the second order derivatives of $u$ which are obtained thanks to {the Malliavin duality formula}.  {These formulas are}
written in terms of some two-sided stochastic integrals, with anticipating integrands. {Several notions of anticipating integral exist}: see for instance~\cite{Alos_Nualart_Viens:00}, \cite{Leon_Nualart:98}, \cite{Nualart_Viens:00} where the definition of such integrals is motivated by similar reasons to ours. 
The two-sided integrals which we would need are similar to those developed in~\cite{Nualart_Pardoux:88}, \cite{Pardoux:87}, \cite{Pardoux_Protter:87}, but we need to consider more general types of integrands.

We have not found the construction of the two-sided integrals we need in the literature. Although interesting in itself, their rigorous and general construction would considerably lengthen the article; this is left for future works. 

We have chosen a different approach: we consider time discretized versions of the problem, and at the end  pass to the limit in estimates. {The advantage is that we do not need to provide the construction of the two-sided integral since at the
discrete level it  is straightforward. The drawback is that all our estimates are made on the discretized processes and computations 
are sometimes technical.}

Nonetheless, we give a formal derivation of the formulas for the first and second derivatives of $u$ in section \ref{sec:formal}.  We hope that this helps the reader to 
understand  our ideas. Also, this allows to describe the type of integrals which would be required to have a direct proof in continuous time. Again such a proof, which would simplify some technical estimates such as those in section \ref{propo:aux}, requires the rigorous construction of a two sided integral and we chose to avoid this.  Thus, results presented in section \ref{sec:formal} remain at a formal level.

\bigskip

Once new regularity estimates on the solutions of Kolmogorov equations are obtained, our second contribution is to address the weak order analysis of the following Euler scheme applied to \eqref{eq:SPDE_intro}:
$$
X_{n+1}-X_n=\Delta t\bigl(AX_{n+1}+G(X_n)\bigr)dt+\sigma(X_n)\bigl(W\bigl((n+1)\Delta t\bigr)-W(n\Delta t)\bigr),\; X_0=x,
$$
where $\Delta t$ is the time step. We prove that the weak rate of convergence is equal to $\frac12$: for arbitrarily small $\kappa\in(0,\frac12)$,
\begin{equation*}
\big|\E\varphi\bigl(X(N\Delta t)\bigr)-\varphi\bigl(X_N\bigr)\big|\le C_\kappa(T,\varphi,x)\Delta t^{\frac12-\kappa},
\end{equation*}
where the integer $N$ is such that $N\Delta t=T$, for arbitrary but fixed $T\in(0,\infty)$.

The value $\frac12$ for the weak order convergence is natural: indeed, it is possible to show that (for an appropriate norm $\|\cdot\|$) one has the strong convergence rate $\frac14$: $\E\|X(N\Delta t)-X_N\|\le C_\kappa(T,x)\Delta t^{\frac14-\frac\kappa2}$.

Like in~\cite{Brehier:14},~\cite{BrehierKopec:13}, in the case of ergodic SPDEs, the analysis can be extended on arbitrarily large time intervals, with a uniform control of the error. This yields error estimates concerning the approximation of invariant distributions. In fact, under appropriate conditions on the Lipschitz constants of the nonlinear coefficients, one can include factors of the type $\exp(-ct)$, with $c>0$, on the right-hand sides of the equations in~\eqref{eq:DuD2u_intro}; alternatively, these regularity estimates are transfered to the solutions of associated Poisson equations. We do not consider this question further in this article.

We generalize the proof of~\cite{Debussche:11}, and of subsequent articles, which was done under the artificial assumptions that $F:H\to H$ and $\sigma:H\to {\mathcal L}(H)$ are of class $C^2$, with bounded derivatives, and that the second order derivative of $\sigma$ satisfies a very restrictive assumption. As already explained above, the new regularity estimates on the solutions of Kolmogorov equation obtained in the first part of the article are fundamental. Here we treat diffusion coefficients of Nemytskii type, and drift coefficients which are sums of Nemytskii and Burgers type nonlinearities. Treating Burgers type nonlinearities is one of the novelties, and one of the main source of technical difficulties, of this work. Even if the decomposition of the error and ideas in the control of the terms are similar to~\cite{Debussche:11}, we need to consider all the terms again since the functional setting is different.

\medskip
Another approach, using the concept of mild It\^o processes, see~\cite{Cox_Jentzen_Kurniawan_Pusnik:16},~\cite{DaPrato_Jentzen_Rockner:10}, has been recently studied to provide weak convergence rates for SPDEs~\eqref{eq:SPDE_intro} with multiplicative noise, for several examples of numerical schemes: see~\cite{ConusJentzenKurniawan:14}, \cite{Hefter_Jentzen_Kurniawan:16}, \cite{JacobeJentzenWelti:15}, \cite{JentzenKurniawan:15}. 
In particular,  in~\cite{Hefter_Jentzen_Kurniawan:16}, a similar result as ours is obtained when the Burgers type nonlinearity is absent ($F_2=0$). This requires also to work in a Banach spaces setting, with an appropriate type of mild It\^o formula~\cite{Cox_Jentzen_Kurniawan_Pusnik:16}. It is not clear that this can be extended to the case $F_2\ne 0$. 
Moreover, we believe that our way of treating the discretization error is more natural and somewhat simpler. We also mention that the regularity requirements are weaker in our work.

 Also, in \cite{Andersson_Kruse_Larsson:16}, a completely different approach is used; but up to now, this covers only additive noise, {\it i.e.} the case when $\sigma$ is constant.

\medskip
{Using Malliavin calculus techniques to get weak convergence rates for numerical approximations is standard in the literature of finite dimensional Stochastic Differential Equations: see for instance~\cite{BallyTalay:96} and~\cite{ClementKohatsu-HigaLamberton:06}. As already emphasized in~\cite{Debussche:11}, Malliavin calculus techniques are used in a completely different manner in this article. Note that the approach of~\cite{ClementKohatsu-HigaLamberton:06} has been extended in the infinite dimensional setting in~\cite{Andersson_Kruse_Larsson:16}. However, the approach of~\cite{BallyTalay:96} can not be applied for SPDEs, as proved by~\cite{Brehier:17}: weak convergence rates for SPDEs heavily depend on the regularity of the test function.}

\bigskip

In future works, we plan to analyze the weak error associated to spatial discretization, using Finite Elements, like in~\cite{Andersson_Larsson:16}. Note that the analysis of the weak error may also be generalized to other examples of time discretization schemes, such as exponential Euler schemes, like in~\cite{Tambue_Ngnotchouye:16},~\cite{Wang:16} for instance.

We have chosen to consider SPDEs~\eqref{eq:SPDE_intro} of one type, namely with Nemytskii diffusion coefficients, and Nemytskii and Burgers type nonlinear drift coefficients, driven by space-time white noise, in dimension $1$. We believe that natural generalizations hold true, for instance for equations in dimension $2$ or $3$, with appropriate noise. Moreover, considering coefficients with unbounded derivatives, with polynomial growth assumptions, is also an important subject, which we have not chosen to treat; indeed it would have required to deal with additional technical difficulties, resulting in hiding the fundamental ideas of our approach.

On a more theoretical point of view, we leave for future work the important question of the construction in continuous time of the two-side stochastic integrals used in the proof of the new regularity results for the solutions of Kolmogorov equations. It may also be interesting to generalize these estimates to higher order derivatives. Finally, we believe that these results and the strategy of proof will have other applications, beyond analysis of weak convergence rates.

\bigskip

This article is organized as follows. The functional setting is made precise in Sections~\ref{sec:setting} and~\ref{sec:setting_coeffs}. Section~\ref{sec:results} contains the statements of our main results, on the regularity of the solution of the Kolmogorov equation (Section~\ref{sec:results_Kolmogorov}), then on the weak rate of convergence of the Euler approximation (Section~\ref{sec:results_num}). Detailed proofs are given in Section~\ref{sec:proof_reg} and in Section~\ref{sec:proof_num} respectively.

\section{Setting}\label{sec:setting}

We use the notation $\N^\star=\left\{1,2,\ldots\right\}$ for the set of (positive) integers.

Throughout the article, $c$ or $C$ denote generic positive constants, which may change from line to line. We do not always precise the various parameters they depend on. When necessary, we write $C=C_{\ldots}(\ldots)$ to emphasize the dependence on some parameters, by convention it is locally bounded on the domains where the parameters live.

\subsection{Functional spaces and stochastic integration}

In all the article, given two Banach spaces $E_1$ and $E_2$, $C^k_b(E_1;E_2)$, or $C^k_b(E_1)$ when $E_1=E_2$, is the space of bounded $C^k$ functions from $E_1$ to $E_2$ with bounded derivatives up to order 
$k$. Also $\mathcal{L}(E_1;E_2)$ denotes the space of bounded linear operators from $E_1$ to $E_2$. If $E_1=E_2$, we set $\mathcal{L}(E_1)=\mathcal{L}(E_1;E_1)$.

The SPDE~\eqref{eq:SPDE_intro} is considered as taking values in the separable Hilbert space $H=L^2(0,1)$, with norm (resp. inner product) denoted by $|\cdot|$ (resp. $\langle \cdot,\cdot \rangle$). We will also extensively use the Banach spaces $L^p(0,1)$, for $p\in[1,\infty]$; the $L^p$ norm is denoted by $|\cdot|_{L^p}$.

\bigskip

When $K$ is a separable Hilbert space, the trace operator is denoted by ${\rm Tr}(\cdot)$; recall that ${\rm Tr \Psi}$ is well defined when $\Psi\in\mathcal{L}(K)$ is nuclear (\cite{Gohberg_Krein}).

We recall that if $\Psi\in\mathcal{L}(K)$ is a nuclear operator and $L\in\mathcal{L}(K)$ is a bounded linear mapping, then $L\Psi$ and $\Psi L$ are nuclear operators, and ${\rm Tr} L\Psi={\rm Tr}\Psi L$.

Let $H_1,\, H_2$ be two separable Hilbert spaces. For $L\in\mathcal{L}(H_1,H_2)$, we denote by $L^\star$ its adjoint. We now introduce the space $\mathcal{L}_2(H_1;H_2)$ of Hilbert-Schmidt operators from $H_1$ to $H_2$: a linear mapping $\Phi\in\mathcal{L}(H_1;H_2)$ is an Hilbert-Schmidt operator if $\Phi^\star\Phi\in\mathcal{L}(H_1,H_1)$ is nuclear, and the associated norm $\|\cdot\|_{\mathcal{L}_2(H_1,H_2)}$ satisfies
$\|\Phi\|_{\mathcal{L}_2(H_1;H_2)}=\|\Phi^*\|_{\mathcal{L}_2(H_2;H_1)}=({\rm Tr}\, \Phi\Phi^*)^{\frac12}$. We use the notation $\mathcal{L}_2(H_1)=\mathcal{L}_2(H_1;H_1)$.

For a function $\psi\in C^1(H;\R)$, we often identify the first order derivative and the gradient: $\langle D\psi(x),h\rangle= D\psi(x)\cdot h$, for $x,h\in H$. Similarly,   if $\psi\in C^2(H;\R)$, we
often identify the second order derivative and the Hessian:  $\langle D^2\psi(x)h,k\rangle= D^2\psi(x)\cdot (h,k)$, for $x,h,k\in H$
\bigskip

We are now in position to present basic elements about stochastic It\^o integrals on Hilbert spaces, see~\cite{DPZ} for further properties. The cylindrical Wiener process on $H$ is defined by
\begin{equation}\label{eq:Wiener_def}
W(t)=\sum_{i\in\N^*}\beta_i(t)f_i,
\end{equation}
where $\bigl(\beta_i\bigr)_{i\in \N^*}$ is a sequence of independent standard scalar Wiener processes on 	a
filtered probability space satisfying the usual conditions $\bigl(\Omega,\mathcal{F},(\mathcal{F}_t)_{t\ge 0},\mathbb{P}\bigr)$ and $\bigl(f_i\bigr)_{i\in \N^*}$ is a complete orthonormal system of $H$. 

It is standard that this representation does not depend on the choice of the complete orthonormal system of $H$. Moreover, it is well-known that $W(t)$ as defined by~\eqref{eq:Wiener_def} does not take values in $H$; however, the series is convergent in any larger Hilbert space $K$, such that the embedding from $H$ into $K$ is an Hilbert-Schmidt operator. 

Given a predictable process $\Phi\in L^2(\Omega\times (0,T);{\mathcal L}_2(H;K))$, the integral $\int_0^T \Phi(s)dW(s)$ is a well defined It\^o integral with values in 
the Hilbert space $K$. Moreover, It\^o isometry reads: 
$$
\E\left(\|\int_0^T \Phi(s)dW(s)\|_K^2\right)=\E\left(\int_0^T\|\Phi(s)\|_{\mathcal L_2(H;K)}^2 ds \right).
$$

\bigskip

In the sequel, we will need to control $L^p$ norms of stochastic integrals, for $p\in[2,\infty)$, for processes $\Phi$ with values in $\mathcal L(H;E)$, where $E=L^p(0,1)$ is a separable Banach space. The space $\mathcal{L}_2(H,K)$ of Hilbert-Schmidt operators is then replaced by the space $R(H,E)$ of $\gamma$- radonifying operators: a linear operator $\Psi\in\mathcal{L}(H,E)$ is a $\gamma$-radonifying operator, if the image by $\Phi$ of the canonical gaussian distribution on $H$ extends to a Borel probability measure on $E$. The space $R(H;E)$ is equipped with the norm $\|\cdot\|_{R(H,E)}$ defined by
$$
\|\Phi\|_{R(H;E)}^2=\tilde{\E}\big| \sum_{i\in\N^*} \gamma_i \Phi f_i\big|^2,
$$
where $(\gamma_i)_{i\in \N^*}$ is a sequence of independent standard (mean $0$ and variance $1$) Gaussian random variables, defined on a probability space $(\tilde\Omega,\tilde{\mathcal F}, \tilde{ \mathbb P})$, with expectation operator denoted by $\tilde{\E}$, and $(f_i)_{i\in \N^*}$ is a complete orthonormal system. The expression of $\|\Phi\|_{R(H;E)}$ does not depend on the choice of these elements. We refer for instance to~\cite{Brzezniak:97,VanNeerven_Veraar_Weis:07,VanNeerven_Veraar_Weis:08} for further properties.

An important tool which is used frequently in the sequel is the left and right ideal property for $\gamma$-radonifying operators: for every separable Hilbert spaces $K,\mathcal{K}$ and for every Banach spaces $E=L^p(0,1)$, $\mathcal{E}=L^q(0,1)$, with $p,q\in[2,\infty)$, for every $L_1\in \mathcal{L}(E,\mathcal{E})$, $\Psi\in R(K,E)$ and $L_2\in\mathcal{L}(\mathcal{K},K)$, one has $L_1\Psi L_2\in R(\mathcal{K},\mathcal{E})$,
\begin{equation}\label{eq:ideal}
\|L_1\Psi L_2\|_{R(\mathcal{K},\mathcal{E})}\le \|L_1\|_{\mathcal{L}(E,\mathcal{E})} \|\Psi\|_{R(K,E)} \|L_2\|_{\mathcal{L}(\mathcal{K},K)}.
\end{equation}

For $E=L^p(0,1)$ with $p\in[2,\infty)$, the following generalization of It\^o isometry holds true, in terms of an inequality only: for predictable processes $\Phi\in L^2(\Omega\times (0,T);R(H;E))$, the It\^o integral $\int_0^T \Phi(s)dW(s)$ can be defined, with values in $E$, and there exists $c_E\in(0,\infty)$, depending only on the space $E$, such that 
\begin{equation}\label{eq:Ito_gamma}
\E\left(\|\int_0^T \Phi(s)dW(s)\|_E^2\right)\le c_E \E\left(\int_0^T\|\Phi(s)\|_{ R(H,E)}^2 ds \right).
\end{equation}

Finally, generalizations of Burkholder-Davies-Gundy inequalities are also available and will be used throughout the article.

To simplify the notation, we often write $L^p$ instead of $L^p(0,1)$.

\subsection{Elements of Malliavin calculus}

We recall basic definitions regarding Malliavin calculus, which is a key tool for the analysis provided below; especially, we define the Malliavin derivative, and state the duality formula which will be used. We simply aim at giving the main notation; for a comprehensive treatment of Malliavin calculus, we refer to the classical monograph~\cite{Nualart:06}.

Malliavin calculus techniques will be required for both contributions of this article: first the proof of new regularity estimates for the solution of Kolmogorov equations associated to SPDEs with nonlinear diffusion coefficient, and second the analysis of weak convergence rates for the numerical discretization of the SPDE. For the first part, we will only use discrete time versions of all objects, which are based on standard integration by parts in the weighted $L_\rho^2$ spaces, where $\rho$ is the Gaussian density. The full generality of Malliavin calculus, in continuous time, is mainly needed in the second part.

Given a smooth real-valued function $G$ on 
$H^n$ and $\psi_1,\dots, \psi_n\in L^2(0, T; H)$, the Malliavin derivative of the smooth random variable $G\bigl(\int_0^T\langle\psi_1(r),dW(r)\rangle,\dots, \int_0^T\langle\psi_n(r),dW(r)\rangle\bigr)$, at time $s$, in the direction $h\in H$, is defined as
$$\begin{array}{l}
\displaystyle {\mathcal D}_s^hG\bigl(\int_0^T\langle\psi_1(r),dW(r)\rangle,\dots, \int_0^T\langle\psi_n(r),dW(r)\rangle\bigr)\\
\displaystyle= \sum_{i=1}^n \partial_i G\bigl(\int_0^T\langle\psi_1(r),dW(r)\rangle,\dots, \int_0^T\langle\psi_n(r),dW(r)\rangle\bigr) \langle\psi_i(s),h\rangle.
\end{array}
$$
We also define the process ${\mathcal D}G$ by $\langle{\mathcal D}G (s), h\rangle={\mathcal D}_s^hG$. It can be shown that ${\mathcal D}$ defines a closable operator 
with values in $L^2(\Omega\times (0,T);H)$, and we denote by $\mathbb D^{1,2}$ the closure of the set of smooth random variables for  the norm
$$
\|G\|_{\mathbb D^{1,2}}=\Bigl(\E(|G|^2+\int_0^T |{\mathcal D}_sG|^2ds)\Bigr)^{\frac12}.
$$   
We define similarly the Malliavin derivative of random variables taking values in H. If $G=  \sum_i G_ie_i \in L^2(\Omega,H)$ with $G_i \in \mathbb D^{1,2}$  for all
$i\in \N^\star$ and   $\sum_i \int_0^T |\mathcal D^s G_i|^2ds<\infty$,
we set 
$$
{\mathcal D}_s^hG = \sum_i {\mathcal D}_s^hG_ie_i, \; {\mathcal D}_sG =\sum_i {\mathcal D}_sG_ie_i. 
$$
The chain rule is valid: if $u\in C_b^1(\R)$ and $G\in \mathbb D^{1,2}$, then $u(G) \in \mathbb D^{1,2}$ and
${\mathcal D}(u(G )) = u'(G ){\mathcal D}G$.

For $G\in \mathbb D^{1,2}$ and $\psi\in L^2(\Omega\times(0, T); H)$, such that $\psi(t)\in \mathbb D^{1,2}$ for all $t \in  [0,T]$, and such that $\int_0^T  \int_0^T |{\mathcal D}_s\psi(t)|^2dsdt < \infty$, we have the {Malliavin calculus duality formula} 
$$
 \E \left(G\int_0^T (\psi(s),dW(s))\right)= \E\left(\int_0^T ({\mathcal D}_sG,\psi(s))ds \right)= \sum_i \E\left(\int_0^T {\mathcal D}_s^{e_i} G\,(\psi(s),e_i)ds\right),
 $$
where the stochastic integral is in general a Skohorod integral. However, in this article, it corresponds with the It\^o integral since we only need to consider the Skohorod integral of adapted processes. Moreover, the duality formula above holds for $G\in {\mathbb D}^{1,2}$ and $\psi\in L^2(\Omega\times (0,T);H)$ when $\psi$ is an adapted process.

Recall that if $G$ is $\mathcal F_t$ measurable, then ${\mathcal D}_sG= 0$ for $s \ge t$.

Finally, we use the following formula, {as a consequence of the duality formula above,} see Lemma 2.1 in~\cite{Debussche:11}: let $G \in {\mathbb D}^{1,2}$, $u\in C_b^2(H)$ and $\psi\in L^2(\Omega\times(0,T),{\mathcal L}_2(H))$ be an adapted process, then
$$
 \begin{array}{ll}
\displaystyle \E\left( Du(G)\cdot \int_0^T\psi(s)dW(s)\right) &\displaystyle=\sum_i \E\left( \int_0^T D^2u(G)\cdot \left(D_s^{e_i}G,\psi(s)e_i\right)ds\right)\\
& \displaystyle = \E\left(\int_0^T{\rm Tr}\left(\psi^*(s)D^2u(G){\mathcal D}_sG\right)ds\right).\\
\end{array}
 $$

\section{{Assumptions and properties of coefficients}}\label{sec:setting_coeffs}

In this section, we give definitions and properties of the coefficients $A$, $G=F_1+BF_2$, and $\sigma$, which appear in~\eqref{eq:SPDE_intro}. {In addition, Section~\ref{sec:Sobolev} presents results concerning Sobolev norms.}

\subsection{{The linear operator $A$}}

The operator $A$ is an unbounded linear operator on $H=L^2(0,1)$: it is defined as the Laplace operator on $(0,1)$, with homogeneous Dirichlet boundary conditions, on the domain $D(A)=H^2(0,1)\cap H^1_0(0,1)$. It satisfies Property~\ref{ass:A} below.

\begin{prop}\label{ass:A}
For $i\in \N^\star$, define $e_i=\sqrt{2}\sin\bigl(i\pi \cdot\bigr)$ and $\lambda_i=(i\pi)^2$. Then
\begin{itemize}
\item $\bigl(e_i\bigr)_{i\in \N^*}$ is a complete orthornormal system of $H$, and, for all $i\in\N^*$,
\begin{equation*}
Ae_i=-\lambda_ie_i.
\end{equation*}
\item For any $\alpha\in \R$, $\sum_{i=1}^{\infty}\lambda_{i}^{-\alpha}<\infty$ if and only if $\alpha>\frac12$.
\item the family of eigenvectors is equibounded in $L^\infty$: $\sup_{i\in \N^\star}|e_i|_{L^\infty}<\infty$.
\end{itemize}
\end{prop}

In particular, for every $p\in[2,\infty]$, $\sup_{i\in \N^\star}|e_i|_{L^p}<\infty$. This equiboundedness property is crucial for many estimates which will be proved in this article.

For every $p\in(2,\infty)$, $A$ can also be seen as an unbounded  linear operator on $L^p(0,1)$, with domain $D_p(A)=\left\{x\in L^p(0,1) ; Ax\in L^p(0,1)\right\}$. Note the inclusion $D_p(A)\subset D_q(A)\subset D(A)$ for $p\ge q\ge 2$.

\bigskip

The operator $A$ generates an analytic semigroup $\bigl(e^{tA}\bigr)_{t\ge 0}$ on $L^p(0,1)$, for every $p\in[2,\infty)$, see for instance~\cite{Pazy}. In the case $p=2$, we have the following formula: $e^{tA}=\sum_{i=1}^{\infty}e^{-t\lambda_i}\langle \cdot, e_i\rangle e_i$ for every $x\in H$ and $t\ge 0$.

\bigskip

We use the standard construction of fractional powers $(-A)^{-\alpha}$ and $(-A)^{\alpha}$ of $A$, for $\alpha\in (0,1)$, see for instance~\cite{Pazy}:
\begin{gather*}
(-A)^{-\alpha}=\frac{\sin(\pi \alpha)}{\pi}\int_{0}^{\infty}t^{-\alpha}(tI-A)^{-1}dt,\\
(-A)^{\alpha}=\frac{\sin(\pi \alpha)}{\pi}\int_{0}^{\infty}t^{\alpha-1}(-A)(tI-A)^{-1}dt,
\end{gather*}
where $(-A)^{\alpha}$ is defined as an unbounded linear operator on $L^p(0,1)$, with domain $D_p\bigr((-A)^\alpha\bigr)$. Definitions are consistent when $p$ varies. In the case $p=2$, the construction is simple: indeed,
\begin{gather*}
(-A)^{-\alpha}x=\sum_{i\in \N^{\star}}\lambda_i^{-\alpha} \langle x,e_i\rangle e_i, \quad x\in H,\\
(-A)^{\alpha}x=\sum_{i\in \N^{\star}}\lambda_i^\alpha \langle x,e_i\rangle e_i, \quad x\in D_2\bigl((-A)^{\alpha}\bigr)=\left\{x\in H ; \sum_{i=1}^{\infty}\lambda_{i}^{2\alpha}\langle x,e_i\rangle^2<\infty\right\}.
\end{gather*}

We use the natural norms on $D_p((-A)^\alpha)$, denoted by $|(-A)^\alpha \cdot|_{L^p}$. 

\subsection{{Useful inequalities}\label{sec:Sobolev}}
{For $p\ne 2$}, the norm of $D_p((-A)^\alpha)$ does not in general coincide with the norm of the standard Sobolev spaces $W^{2\alpha,p}=W^{2\alpha,p}(0,1)$; see~\cite[Section~4.2.1]{Triebel} for their definitions. When $2\alpha$ is not an integer, we may use the norm defined in~\cite[Section~4.4.1, Remark 2]{Triebel}. {In this article, $\alpha\in (0,1/2)$ and in this case, this norm writes:}
\begin{equation}\label{e8}
|x|_{W^{2\alpha,p}}=|x|_{L^p}+\int_0^1\int_0^1 \frac{|x(\xi)-x(\eta)|^p}{|\xi-\eta|^{1+2\alpha p}} d\xi d\eta. 
\end{equation}

{It is useful to compare the two scales of spaces $D_p((-A)^\alpha)$ and $W^{2\alpha,p}$}. {Below we use a series of results from \cite{Triebel}. Let us choose  $\epsilon>0$, by the definition in section~4.2.1, Theorem 1, section 4.3.1 which asserts that 2.4.2 (16) holds, we have:
$$
W^{2\alpha-\epsilon,p}=B^{2\alpha-\epsilon}_{p,p}= (L^p,W^{2,p})_{\alpha-\frac\epsilon2,p}.
$$
where $B^{2\alpha-\epsilon}_{p,p}$ is the Besov space and $(\cdot,\cdot)_{\theta,p}$ denotes the interpolation spaces. Then, we use 1.3.3 (e), the equality 
$W^{2,p}=D((-A))$ and 1.15.2 (d) to obtain:
$$
(L^p,W^{2,p})_{\alpha-\frac\epsilon2,p}\subset (L^p,W^{2,p})_{\alpha, 1}\subset D((-A)^\alpha).
$$
It follows
$$
W^{2\alpha-\epsilon,p} \subset D((-A)^\alpha).
$$
The same arguments imply
$$
D((-A)^\alpha) \subset W^{2\alpha+\epsilon,p}.
$$}
(Note that for $p=2$, we can take $\epsilon=0$ and we have in fact  $D((-A)^\alpha) \subset W^{2\alpha,2}$.)

We deduce the following inequalities:
\begin{equation}\label{eq:Sobolev-domain}
|x|_{W^{2\alpha-\epsilon,p}}\le c_{\alpha,\epsilon,p} |(-A)^\alpha x|_{L^p}, \; x\in D_p((-A)^\alpha) \quad;\quad |(-A)^\alpha x|_{L^p}\le c_{\alpha,\epsilon,p}|x|_{W^{2\alpha+\epsilon,p}},
\; x\in W^{2\alpha+\epsilon,p},
\end{equation}
for $c_{\alpha,\epsilon,p}\in(0,\infty)$. 

\medskip

{We also need inequalities for composition and products in these spaces. Let us consider a Lipschitz continuous function $g:\R\to \R$. It satisfies:
$$
|g(t)|\le L(1+|t|),\quad |g(t)-g(s)|\le L|t-s|,\quad t,s\in \R,
$$
for some constant $L$. It follows for $x\in W^{2\alpha,p}$, $\alpha <\frac12$:
$$
|g(x)|_{L^p}\le L(1+|x|_{L^p}),\quad \int_0^1\int_0^1 \frac{|g(x(\xi))-g(x(\eta))|^p}{|\xi-\eta|^{1+2\alpha p}} d\xi d\eta
\le L \int_0^1\int_0^1 \frac{|x(\xi)-x(\eta)|^p}{|\xi-\eta|^{1+2\alpha p}} d\xi d\eta
$$
Recalling  the definition \eqref{e8} of the norm on $W^{2\alpha,p}$}  we get the following inequality: for $\alpha<\frac12$ and $\epsilon>0$, any $x\in D_p\bigl((-A)^{\alpha+\epsilon}\bigr)$, and any Lipschitz continuous function $g:\R\to \R$,
\begin{equation}\label{eq:Sobolev-Lipschitz}
\big|(-A)^{\alpha}g(x)\big|_{L^p}\le c_{\alpha,\epsilon,p}|g(x)|_{W^{2\alpha+\epsilon,p}}\le c_{\alpha,\epsilon,p,g}\big(1+|x|_{W^{2\alpha+\epsilon,p}}\bigr)\le 
c_{\alpha,\epsilon,p,g}\big(1+|(-A)^{\alpha+\epsilon}x|_{L^p}\bigr).
\end{equation}
{Also, by H\"older inequality and \eqref{e8}}, for $\alpha<\frac12$, and $x\in W^{2\alpha,q},y\in W^{2\alpha,r}$ such that $\frac1p=\frac1q+\frac1r$, one has 
\begin{equation}\label{eq:product_1}
|xy|_{W^{2\alpha,p}}\le c_{\alpha,q,r}\bigl(|x|_{L^{q}}|y|_{W^{2\alpha,r}}+|x|_{W^{2\alpha,q}} |y|_{L^{r}}\bigr)\le c_{\alpha,q,r}|x|_{W^{2\alpha,q}}|y|_{W^{2\alpha,r}}.
\end{equation}
Using then~\eqref{eq:Sobolev-domain}, \eqref{eq:product_1} yields that for $\alpha\in (0,\frac12),\; \epsilon >0,\; \frac1p=\frac1q+\frac1r$, and $x\in D_q\bigl((-A)^{\alpha+\epsilon}\bigr)$, $y\in D_r\bigl((-A)^{\alpha+\epsilon}\bigr)$, one has
\begin{equation}\label{eq:product_2}
|(-A)^{\alpha} xy|_{L^p}\le c_{\alpha,\epsilon,q,r}|(-A)^{\alpha+\epsilon}x|_{L^q}|(-A)^{\alpha+\epsilon}y|_{L^r}.
\end{equation}

\medskip
{Below, we also need to estimate products of two functions, one of which belongs to a space of negative regularity, in a space
$D((-A)^{-\alpha})$
 with $\alpha\in (0,\frac12)$. More precisely, given $x\in D((-A)^{-\alpha})$, we want to give a meaning to the product $xy$. For functions defined on the whole 
 space $\R$, this is classically treated thanks to paraproduct. In the case of the interval treated here, we provide an alternate argument to 
 do this. As in the case of $\R$, the sum of the regularity of $x$ and $y$ has to be positive and the product is defined only in spaces of negative regularity.}
 
{We use a duality argument. Let us first consider smooth $x$ and $y$, then for $z$ smooth. Let $\alpha\in (0,\frac12)$, $\epsilon>0,\; \frac1p=\frac1q+\frac1r$, $\frac1q+\frac1{q'}=1$; note that $\frac1{q'}=\frac1r+\frac1{p'}$:
$$
\langle xy,z\rangle
= \langle(-A)^{-\alpha}x, (-A)^{\alpha}(yz)\rangle
= \int_0^1 ((-A)^{-\alpha}x)(\xi) ((-A)^{\alpha}(yz))(\xi) d\xi\le |(-A)^{-\alpha}x|_{L^q} |(-A)^{\alpha}(yz)|_{L^{q'}}.
$$
From~\eqref{eq:Sobolev-domain} and~\eqref{eq:product_1}, we obtain
\begin{align*}
|(-A)^{\alpha}(yz)|_{L^{q'}}&\le c|yz|_{W^{2\alpha+\epsilon,q'}} \\
&\le c |y|_{W^{2\alpha+\epsilon,r}}|z|_{W^{2\alpha+\epsilon,p'}}\\
&\le c |(-A)^{\alpha+\epsilon}y|_{L^{r}}|(-A)^{\alpha+\epsilon}z|_{L^{p'}}.
\end{align*}
We deduce:
$$
\langle xy,z\rangle
\le c|(-A)^{-\alpha}x|_{L^q} |(-A)^{\alpha+\epsilon}y|_{L^{r}}|(-A)^{\alpha+\epsilon}z|_{L^{p'}}.
$$
Since $|(-A)^{-\alpha-\epsilon}(xy)|_{L^p}=  \sup\frac{ \langle xy,z\rangle}{|(-A)^{\alpha+\epsilon}z|_{L^{p'}}}$, we obtain:
\begin{equation}\label{eq:product_3}
|(-A)^{-\alpha-\epsilon}(xy)|_{L^p}\le c_{\alpha,\epsilon,q,r}|(-A)^{-\alpha}x|_{L^{q}}|(-A)^{\alpha+\epsilon}y|_{L^r}.
\end{equation}
} 
By density, this inequality remains true for all $x,y,z$ such that the right hand side is finite.
%

{Finally, for every $p\in[2,\infty)$, we have the Sobolev embedding: $L^p\subset W^{\frac12-\frac1p,2}$, see for instance~\cite[Theorem~7.58]{Adams}.  Since $D_2\bigl((-A)^{\frac{1}{4}-\frac{1}{2p}}\bigr)$ is a closed
subspace of $W^{\frac12-\frac1p,2}$ and the norms $|\cdot|_{W^{2\alpha,2}}$ and $|(-A)^{\alpha}\cdot|_{L^2}$ are equivalent on $D_2\bigl((-A)^{\alpha}\bigr)$ (see \cite{Triebel}, Theorem 1.18.10), we deduce:
\begin{equation}\label{eq:lem_Sobolev}
|x|_{L^p}\le C(p) \big|(-A)^{\frac{1}{4}-\frac{1}{2p}}x|_{L^2}.
\end{equation}
}
%

\subsection{{Nonlinear terms $G$ and $\sigma$}}

The drift $G$ is the sum of a Nemytskii and of a Burgers type nonlinearities: $G=F_1+BF_2$, where $Bx=\partial_\xi x\in L^p(0,1)$ for $x\in W^{1,p}(0,1)$, and where $F_1$ and $F_2$ are Nemytskii coefficients. Precisely, let $\tilde{F}_1,\tilde{F}_2\in\mathcal{C}_b^3(\R)$ be two real-valued functions. We assume that they are bounded to simplify the presentation, but this {could easily be relaxed}.  Then we set, for every $x\in L^p$, with $p\in[1,\infty]$, $F_i(x)(\cdot)=\tilde{F}_i\bigl(x(\cdot)\bigr)$, for $i\in\left\{1,2\right\}$.

Straightforward applications of H\"older inequality yield Property~\ref{ass:F} below.

\begin{prop}\label{ass:F}
Let $F\in\left\{F_1,F_2\right\}$.

For every $p\in[1,\infty]$, there exists $C_p\in(0,\infty)$ such that for every $x\in L^2, h\in L^p$
\begin{equation*}
|F(x)|_{L^p}\le C_p \quad,\quad |F'(x).h|_{L^p}\le C_p|h|_{L^p};
\end{equation*}
moreover, if $q_1,q_2,r_1,r_2,r_3\in[1,\infty]$ are such that $\frac{1}{q_1}+\frac{1}{q_2}=\frac{1}{p}$ and $\frac{1}{r_1}+\frac{1}{r_2}+\frac{1}{r_3}=\frac{1}{p}$, there exists $C_{p}(q_1,q_2)$ and $C_p(r_1,r_2,r_3)$ such that for every $x\in L^2$
\begin{equation*}
\begin{gathered}
|F^{(2)}(x).(h_1,h_2)|_{L^p}\le C_p(q_1,q_2)|h_1|_{L^{q_1}}|h_2|_{L^{q_2}}, \quad \forall~h_1\in L^{q_1},h_2\in L^{q_2}\\
|F^{(3)}(x).(h_1,h_2,h_3)|_{L^p}\le C_p(r_1,r_2,r_3)|h_1|_{L^{r_1}}|h_2|_{L^{r_2}}|h_3|_{L^{r_3}}, \quad \forall~h_1\in L^{r_1},h_2\in L^{r_2},h_3\in L^{r_3}.
\end{gathered}
\end{equation*}
\end{prop}

In order to control terms of the form $BF_2(x)$, we will use the following property
\begin{equation}\label{eq:norm_AB}
\big|(-A)^{-\alpha}B(-A)^{-\beta} \big|_{\mathcal{L}(L^p)}<\infty, \mbox{ for } \alpha+\beta>\frac12.
\end{equation}
Indeed, this inequality is a direct consequence of~\eqref{eq:Sobolev-domain} when $\alpha=0$, and uses a duality argument when $\beta=0$. The general case
follows by an interpolation argument.

\bigskip

The diffusion coefficient $\sigma$ is a linear operator of Nemytskii type. Precisely, let $\tilde\sigma\in\mathcal{C}_b^3(\R)$ be a real-valued, bounded, function, with bounded derivatives up to order $3$. Then, for every $p\in[1,\infty]$, define $\bigl(\sigma(x)h\bigr)(\cdot)=\tilde \sigma\bigl(x(\cdot)\bigr)h(\cdot)$
for all $x,h\in L^p$.

\begin{prop}\label{ass:sigma}
For every $p,q\in[1,\infty]$, $\sigma:L^2\to \mathcal{L}(L^p,L^q)$ is of class $\mathcal{C}^3$. Moreover, the following conditions on the derivatives of $\sigma$ hold true.

For every $p\in[2,\infty]$, there exists $C_p\in(0,\infty)$ such that for every $x\in L^2$
\[
|\sigma(x)|_{\mathcal{L}(L^p)}\le C_p.
\]

For every $p\in(2,\infty)$, there exists $C_p\in(0,\infty)$ such that for every $x\in L^2$
\begin{eqnarray}
&&\big|(-A)^{-\frac{1}{2p}}\bigl(\sigma'(x).h\bigr)\big|_{\mathcal{L}(L^2)}\le C_p |h|_{L^p},~\forall~h\in L^p, \label{eq:estim_sigma_1} \\
&&\big|(-A)^{-\frac{1}{2p}}\bigl(\sigma''(x).(h,k)\bigr)\big|_{\mathcal{L}(L^2)}\le C_p |h|_{L^{2p}}|k|_{L^{2p}},~\forall~h,k\in L^{2p}, \label{eq:estim_sigma_2}\\
&&\big|(-A)^{-\frac{1}{2p}}\bigl(\sigma^{(3)}(x).(h,k_1,k_2)\bigr)\big|_{\mathcal{L}(L^2)}\le C_p |h|_{L^{2p}}|k_1|_{L^{4p}}|k_2|_{L^{4p}},~\forall~h\in L^{2p},k_1,k_2\in L^{4p} \label{eq:estim_sigma_3}.
\end{eqnarray}

Finally, for every $x\in L^2$ and $h\in L^p,k_1,k_2\in L^{2p}$
\begin{equation}\label{eq:sigma_star}
\sigma(x)^\star=\sigma(x) \quad,\quad \bigl(\sigma'(x).h\bigr) ^\star=\sigma'(x).h \quad,\quad \bigl(\sigma''(x).(k_1,k_2)\bigr)^{\star}=\sigma''(x).(k_1,k_2).
\end{equation}
\end{prop}

We sketch the proof of~\eqref{eq:estim_sigma_1}, the two other estimates~\eqref{eq:estim_sigma_2} and~\eqref{eq:estim_sigma_3} are obtained in the same way. For every $y,z\in L^2$,
\begin{align*}
\langle (\sigma'(x).h)y,(-A)^{-\frac{1}{2p}}z\rangle&\le C|h|_{L^p}|y|_{L^2}|(-A)^{-\frac{1}{2p}}z|_{L^r}\\
&\le C|h|_{L^p}|y|_{L^2}|(-A)^{\frac{1}{4}-\frac{1}{2r}-\frac{1}{2p}}z|_{L^2},
\end{align*}
thanks to H\"older inequality, with $\frac1r+\frac1p=\frac12$, and inequality \eqref{eq:lem_Sobolev}.

When no confusion is possible, we will often use the notations $F_i$ for $\tilde{F}_i$, and $\sigma$ for $\tilde{\sigma}$.

\subsection{Test functions $\varphi$}\label{s2.4}

We now give the regularity assumptions on the test functions $\varphi$. Typically, $\varphi$ is only
 defined on $L^p(0,1)$, for some $p\in[2,\infty)$ and is not a $C^n$ function on $L^p$ for $n\ge 2$. It possesses derivatives only in restricted directions, that is in a smaller space which is in general $L^q$ for $q>p$. To state the assumption on the test functions allowed, we consider regularized version $\varphi_\delta$, 
 defined in Assumption~\ref{ass:phi} below.

\begin{hyp}\label{ass:phi} Let $p\in [2,\infty)$ and 
{$\varphi:L^p(0,1)\to \R$.} For every $\delta\in(0,1)$, define $\varphi_\delta(\cdot)=\varphi\bigl(e^{\delta A}\cdot)$. We assume that $\varphi_\delta$ is of class $\mathcal{C}^3$ on $H$, for every $\delta\in(0,1)$. Moreover, we assume that the derivatives satisfy the following conditions, uniformly with respect to $\delta\in(0,1)$: there exist $q\in[2,\infty)$, $K\in \N^\star \cup\left\{0\right\}$, and $C(p,q,K)\in(0,\infty)$ such that for every $x\in L^p$, and $h_1,h_2,h_3\in L^q$
\begin{equation}
\label{e20.1}
\big| D\varphi_\delta(x).h_1\big| \le C(p,q,K)\bigl(1+|x|_{L^{p}}\bigr)^{K}|h_1|_{L^{q}},
\end{equation}
\begin{equation}
\label{e20.2}
\big| D^2\varphi_\delta(x).(h_1,h_2)\big| \le C(p,q,K)\bigl(1+|x|_{L^{p}}\bigr)^{K}|h_1|_{L^{q}}|h_2|_{L^{q}},
\end{equation}
\begin{equation}
\label{e20.3}
\big| D^3\varphi_\delta(x).(h_1,h_2,h_3)\big| \le C(p,q,K)\bigl(1+|x|_{L^{p}}\bigr)^{K}|h_1|_{L^{q}}|h_2|_{L^{q}}|h_3|_{L^{q}}.
\end{equation}
\end{hyp}

Interesting examples of test functions $\varphi$ are constructed as follows. Let $\phi\in\mathcal{C}^3(\R)$ a function of class $\mathcal{C}^3$; we assume that the derivatives of $\phi$ have at most polynomial growth. Define
\begin{equation*}
\varphi(x)=\int_{0}^{1}\phi\bigl(x(\xi)\bigr)d\xi,
\end{equation*}
for $x\in L^n(0,1)$, where $n\in\N^\star$ is such that $\underset{x\in \R}\sup \frac{|\varphi'(x)|}{(1+|x|)^n}<\infty$.

Since derivatives of $\varphi$ take the form $D^{(n)}\varphi(x).\bigl(h_1,\ldots,h_n\bigr)=\phi^{(n)}\bigl(x(\cdot)\bigr)h_1(\cdot)\ldots h_n(\cdot)$, Assumption~\ref{ass:phi} is satisfied by applying H\"older inequality, with appropriately chosen parameters $p,q$.

If we assume that the derivatives of $\phi$ are bounded, we may choose $K=0$ and $p=2$; the estimate on the third order derivative requires to choose $q=3$.

\section{Main results}\label{sec:results}

We consider the stochastic evolution equation~\eqref{eq:SPDE_intro}, which we recall here:
\begin{equation}\label{eq:SPDE}
dX_t=AX_tdt+G(X_t)dt+\sigma(X_t)dW(t) , \quad X(0)=x,
\end{equation}
where $x\in H$ is an arbitrary initial condition.

For every time $T\in(0,\infty)$, equation~\eqref{eq:SPDE} admits a unique mild solution in $C([0,T];H)$, {\it i.e.} $X=\bigl(X_t\bigr)_{t\in[0,T]}$ is a $H$-valued continuous stochastic process such that for every $0\le t\le T$
\begin{equation}
X_t=e^{tA}x+\int_{0}^{t}e^{(t-s)A}G(X_s)ds+\int_{0}^{t}e^{(t-s)A}\sigma(X_s)dW(s),
\end{equation}
where the $H$-valued stochastic integral is interpreted in It\^o sense. We refer for instance to~\cite{DPZ} for a proof of this standard result.

To emphasize on the influence of the initial condition $x$, we often use the notation $X(t,x)$. However, in many computations we omit this dependence and write $X_t$ for simplicity.

A rigorous treatment of the problem is made easier by considering regularized coefficients $G_\delta$ and $\sigma_\delta$, for $\delta>0$, defined as follows:
\[
G_{\delta}=e^{\delta A}G\bigl(e^{\delta A}\cdot\bigr)=e^{\delta A}F_1\bigr(e^{\delta A}\cdot\bigr)+Be^{\delta A}F_2\bigl(e^{\delta A}\cdot\bigr) \quad,\quad \sigma_{\delta}=e^{\delta A}\sigma\bigl(e^{\delta A}\cdot\bigr)e^{\delta A}.
\]
It is straightforward to check that Properties~\ref{ass:F} and~\ref{ass:sigma} are preserved after regularization, with constants which are uniform with respect to $\delta$. Indeed,  $e^{\delta A}$ is bounded with norm equal to $1$, from $L^p$ to $L^p$, for every $p\in[1,\infty]$ and $\delta\in(0,1)$. 
Also  for $\delta>0$, $e^{\delta A}$ is a bounded operator from $L^2$ to $L^p$ for any $p>2$, and thus the regularized coefficients $F_\delta$ and $\sigma_\delta$ are $\mathcal{C}^3_b$ on $H$ (but with norm depending on $\delta$). Note that $B$ and $e^{\delta A}$ do not commute.

\begin{rem}
We cannot use standard regularization methods in our setting, such as spectral Galerkin projections, like in~\cite{Debussche:11}. Indeed, the associated projection operators are not uniformly bounded (with respect to dimension), in $L^p$ spaces for $p>2$.

The regularization we use in this article does not provide finite dimensional approximation of the process.

Alternatively, the not so different regularization proposed in~\cite{Hairer-Stuart-Voss:07} (see Lemma~3.1) may be used. It is based on an additional truncation of modes larger than $N(\delta)$, in the definition of $e^{\delta A}$, for a well-chosen integer $N(\delta)$.
\end{rem}

In the computations below, we often omit to mention the dependence on $\delta$. All the estimates we state and prove are uniform in $\delta$.

Working with regularized coefficients $F_\delta$ and $\sigma_\delta$, with $\delta\in(0,1)$, we introduce the regularized SPDE
\begin{equation}\label{eq:SPDE_reg}
dX_t^\delta=AX_t^\delta dt+G_\delta(X_t^\delta)dt+\sigma_\delta(X_t^\delta)dW(t) , \quad X^\delta(0)=x.
\end{equation}
When $\delta\to 0$, $X^\delta$ converges (in a suitable sense) to $X$. Consistently, the notation $X^0=X$ will be used.

For every $\delta\in(0,1)$, introduce the function $u_\delta:[0,T]\times L^2\to \R$, defined by
\begin{equation}\label{eq:u_N}
u_\delta(t,x)=\E\bigl[\varphi_\delta(X^\delta(t,x)\bigr)\bigr],
\end{equation}
and the function $u:[0,T]\times L^2\to \R$
\begin{equation}\label{eq:u}
u(t,x)=\E\bigl[\varphi(X(t,x)\bigr)\bigr].
\end{equation}
The function $u_\delta$, resp. $u$, is formally solution of the Kolmogorov equations associated to \eqref{eq:SPDE_reg}, resp. \eqref{eq:SPDE}. As already mentioned, the regularity results proved in this article could be used to prove that these functions are in fact strict solutions of these Kolmogorov equations.

Consistently, we use the notation $u_0=u$. Indeed, results on $u$ will be obtained from results proved for $\delta>0$ and passing to the limit $\delta \to 0$.

Thanks to~\cite{Andersson_Hefter_Jentzen_Kurniawan:16} or~\cite{Cerrai:01}, for every $\delta\in(0,1)$ and $t\ge 0$, $u_\delta(t,\cdot)$ is a function of class $\mathcal{C}^3$ on $L^2$.

\subsection{Regularity estimates on the derivatives of the Kolmogorov equation solution}\label{sec:results_Kolmogorov}

The first main results of this article are new estimates on the first and second order spatial derivatives of $u$.

For our results given below, we consider the setting of Section~\ref{sec:setting_coeffs} and Section~\ref{s2.4}. Note that all the results are valid for the parameter $q$, defined in Assumption~\ref{ass:phi}, satisfying $q\in[2,\infty)$. The proofs of the cases $q=2$ and $q\in(2,\infty)$ need to be treated separately. We only provide detailed proofs in the case $q\in(2,\infty)$. The case $q=2$ is easier.

\begin{theo}\label{theo:D1}
For every $\beta\in [0,1)$ and $T\in(0,\infty)$, there exists $C_{\beta}(T)$, such that for every $\delta\in [0,1)$, $t\in (0,T]$, $x\in L^p$ and $h\in L^{q}$
\begin{equation}\label{eq:theo_D13_1}
\big|Du_\delta(t,x).h\big|\le \frac{C_\beta(T)}{t^\beta}(1+|x|_{L^{\max(p,2q)}})^{K+1}|(-A)^{-\beta}h|_{L^{2q}}.
\end{equation}
\end{theo}

This result can be interpreted as a regularization property: for every $t>0$ and $\beta\in(0,1)$, we have $(-A)^{\beta}Du(t,x)\in L^{r}$, where $r$ is the conjugate exponent of $2q$, {\it i.e.} $\frac1{r}+\frac1{2q}=1$. For $t=0$, from Assumption \ref{ass:phi}, we formally have $Du(0,x)=D\varphi(x)\in L^r \subset L^{q'}$ where $q'$ is the conjugate exponent of $q'$. No information on $D\varphi$ in $D((-A)^\beta)$ is available.

Theorem~\ref{theo:D1} is not difficult for $\beta\in[0,\frac12)$ (see~\cite{Andersson_Hefter_Jentzen_Kurniawan:16},~\cite{Debussche:11}). Getting the result for $\beta\in[0,1)$ with standard arguments is possible only in the case of additive noise. We recall below in Section~\ref{sec:formal} where the limitation $\beta<\frac12$ comes from in direct approaches, when $\sigma$ is nonlinear. Then we give a formal description of our strategy of proof of Theorem~\ref{theo:D1} and introduce new arguments.

The constant $C_\beta(T)$ depends on $\varphi$ through the constants appearing in Assumption  \ref{ass:phi}. More precisely, it depends on the constant in the right hand side of \eqref{e20.1} and \eqref{e20.2}. It may seem surprising that we need information on the second differential of $\varphi$ to get an estimate on the first differential
of $u$. This is due to the final step of the proof where we use an interpolation argument to get rid of an extra smoothing parameter $\tau$ introduced below. We do not 
know whether this is optimal. 
\bigskip

We now turn to the result on $D^2u$, which is also a regularization property.

\begin{theo}\label{theo:D2}
For every $\beta,\gamma\in [0,\frac12)$ and $T\in(0,\infty)$, there exists $C_{\beta,\gamma}(T)$, such that for every $\delta\in [0,1)$, $t\in (0,T]$, $x\in L^p$ and $h_1,h_2\in L^{4q}$
\begin{equation}\label{eq:theo_D2}
\big|D^2u_\delta(t,x).\bigl(h_1,h_2\bigr)\big|\le \frac{C_{\beta,\gamma}(T)}{t^{\beta+\gamma}}\bigl(1+|x|_{L^{\max(p,2q)}}\bigr)^{K+1}|(-A)^{-\beta}h_1|_{L^{4q}}|(-A)^{-\gamma}h_2|_{L^{4q}}.
\end{equation}
\end{theo}

Again, the novelty in Theorem~\ref{theo:D2} is the range  $[0,\frac12)$ for the parameters $\beta$ and $\gamma$. More precisely, we remove the restriction $\beta+\gamma<\frac12$, for which a direct proof works, see~\cite{Andersson_Hefter_Jentzen_Kurniawan:16},~\cite{Debussche:11}.

As above, the constant $C_{\beta,\gamma}(T)$ depends on $\varphi$ through the constants appearing in Assumption  \ref{ass:phi}. Now, it depends on the constant in the right hand side of \eqref{e20.1},  \eqref{e20.2} and \eqref{e20.3}. 

Another novelty is that we consider SPDEs with a spatial derivative in the nonlinear term. Moreover, Nemytskii type diffusion and nonlinear terms are allowed. This requires bounds depending on $L^q$ norms and not only on $L^2$ norms. 

\begin{rem}
The presence of $L^{2q}$ and $L^{4q}$ norms in the right-hand side of ~\eqref{eq:theo_D13_1} and ~\eqref{eq:theo_D2} is not optimal. A careful inspection of the proof reveals that norms on the right-hand side may be replaced with weaker $L^{q+\epsilon}$ and $L^{2q+\epsilon}$ norms, where $\epsilon$ is arbitrarily close to $0$. Moreover, at the price of increasing the singularity in $T$, one may use the Markov property to get estimates which depend on $L^r$ with much smaller $r$.
\end{rem}

The main motivation and application of Theorem~\ref{theo:D1} and Theorem~\ref{theo:D2} is the analysis of weak convergence rates for numerical discretizations of the SPDE~\eqref{eq:SPDE}. For that purpose, being able to choose both $\beta$ and $\gamma$ arbitrarily close to $\frac12$ is fundamental. Theorem~\ref{theo:D1} with $\beta\in [0,\frac12)$ is sufficient to consider the case with 
$F_2=0$, but we need $\beta$ close to $1$ to treat the Burger type nonlinearity $BF_2$.

In the additive noise case, it is possible to choose $\beta,\gamma\in[0,1)$, such that $\beta+\gamma<1$ in Theorem~\ref{theo:D2}. Then we may choose for instance $\beta\in[\frac12,1)$,
and this simplifies several arguments in the weak convergence analysis - and also in the argument presented below to give a meaning to the trace term in~\eqref{eq:Kolmogorov_intro}. We believe that 
the same strategy as for the proof of Theorem~\ref{theo:D1} can be adapted to prove that indeed the conclusion of Theorem~\ref{theo:D2} is still valid for $\beta,\gamma\in[0,1)$ with $\beta+\gamma<1$.
Substantial generalizations of the arguments are however required, and they will be studied in future works.

In addition to the analysis of weak convergence errors, Theorems~\ref{theo:D1} and~\ref{theo:D2} can be used to give a meaning to the different terms in the right-hand side of~\eqref{eq:Kolmogorov_intro}. 
First the terms $\langle Ax,Du(t,x)\rangle$ has a meaning as soon as $|(-A)^{1-\beta}x|_{L^q}<\infty$, for $\beta$ arbitrarily close to $1$. Choosing $\beta>\frac34$ is fundamental, since the solution $X(t,x)$ takes values in $D_q\bigl((-A)^\alpha\bigr)$ only for $\alpha<\frac14$. The term $\langle G(x),Du(t,x)\rangle=\langle F_1(x)+BF_2(x),Du(t,x)\rangle$ is well-defined also, choosing $\beta>\frac12$  
thanks to \eqref{eq:norm_AB}. The trace term is more delicate.
Thanks to Theorem~\ref{theo:D2}, for $\beta,\gamma\in [0,\frac12)$ and $x\in L^p$ we have
\[
{\rm Tr} \left(\sigma(x)\sigma^*(x)D^2u(t,x)\right)=\sum_{n} D^2u(t,x).(\sigma^2(x)e_n,e_n),
\]
and 
\begin{align*}
\sum_{n} \big|D^2u(t,x).(\sigma^2(x)e_n,e_n)\big|
&\le \frac{C_{\beta,\gamma}(T)}{t^{\beta+\gamma}}\bigl(1+|x|_{L^{p}}^{K}\bigr)\sum_{n}|(-A)^{-\beta}(\sigma^2(x)e_n)|_{L^{4q}}|(-A)^{-\gamma}e_n|_{L^{4q}}\\
& \le \frac{C_{\beta,\gamma}(T)}{t^{\beta+\gamma}}\bigl(1+|x|_{L^{p}}^{K}\bigr)\sum_{n}|(-A)^{-\beta}(\sigma^2(x)e_n)|_{L^{4q}}\lambda_n^{-\gamma},
\end{align*}
where we have used $\underset{n\in\N^\star}\sup|e_n|_{L^{4q}}<\infty$ thanks to Property~\ref{ass:A}.

Nevertheless, taking $\gamma<\frac12$ arbitrarily close to $\frac12$ and $\beta=0$ is not sufficient, since $\sum_{n\in\N^\star}\lambda_n^{-\gamma}=\infty$. To overcome this issue, we use~\eqref{eq:product_3}, then~\eqref{eq:Sobolev-Lipschitz}:
$$
|(-A)^{-\beta}(\sigma^2(x)e_n)|_{L^{4q}}\le c |(-A)^{\beta}\sigma^2(x)|_{L^{8q}}|(-A)^{-\beta+\epsilon}
e_n|_{L^{8q}}\le c (1+ |(-A)^{\beta+\epsilon}x|_{L^{8q}})|(-A)^{-\beta+\epsilon}
e_n|_{L^{8q}}.
$$
We choose $\gamma,\beta\in [0,\frac12)$ and $\epsilon >0$ such that $\gamma+\beta-\epsilon >\frac12$:
$$
\sum_{n} \big|D^2u(t,x).(\sigma^2(x)e_n,e_n)\big|\le \frac{C_{\beta,\gamma}(T)}{t^{\beta+\gamma}}\bigl(1+|x|_{L^{p}}^{K}\bigr)(1+ |(-A)^{\beta+\epsilon}x|_{L^{8q}})\sum_{n}\lambda_n^{-\gamma-\beta+\epsilon}.
$$
Note that it is possible to choose $\beta,\epsilon$ arbitrarily close to $0$. Therefore the trace term in 
\eqref{eq:Kolmogorov_intro} is meaningful as soon as $x\in D_{8q}((-A)^\alpha)$ for some $\alpha >0$. Again the exponant $8q$ is not optimal.

\bigskip

For completeness, we also state a regularity result on the third order derivatives of $u_\delta$. This result is useful to prove the two results above and in the analysis of the weak convergence rate for numerical approximations below. Contrary to Theorems~\ref{theo:D1} and~\ref{theo:D2}, since we consider a restrictive range for the parameters $\alpha,\beta,\gamma$, {\it i.e.} with the constraint $\alpha+\beta+\gamma<1/2$, standard arguments are
sufficient and the proof is left to the reader. The arguments used for Theorems~\ref{theo:D1} and~\ref{theo:D2} could be naturally extended to generalize Proposition~\ref{theo:D3}, under appropriate assumptions, as well as to higher order derivatives. We leave the study of such generalizations to future works.

\begin{propo}\label{theo:D3}

For every $\alpha,\beta,\gamma\in [0,\frac12)$ such that $\alpha+\beta+\gamma <\frac12$, and $T\in (0,\infty)$, there exists $C_{\beta}(T)$, such that for every 
$\delta\in (0,1)$,  $h_1,h_2,h_3\in L^{3q}$
\begin{equation}\label{eq:theo_D13_3}
\big|D^3u_\delta(t,x).(h_1,h_2,h_3)\big|\le \frac{C_{\alpha,\beta,\gamma}(T)}{t^{\alpha+\beta+\gamma}}(1+|x|_{L^p})^K|(-A)^{-\alpha}h_1|_{L^{3q}}|(-A)^{-\beta}h_2|_{L^{3q}}|(-A)^{-\gamma}h_3|_{L^{3q}}.
\end{equation}
\end{propo}
The constant $C_{\alpha,\beta,\gamma}(T)$ depends on $\varphi$ through the constants appearing in Assumption  \ref{ass:phi}.

The results in Theorems~\ref{theo:D1},~\ref{theo:D2} are proved for the function $u_\delta$, defined by~\eqref{eq:u_N}, for $\delta\in(0,1)$. 
Thanks to the result on the third order derivatives of $u_\delta$, we may take the limit $\delta\to 0$ in Theorems~\ref{theo:D1} and~\ref{theo:D2}; this provides G\^ateau differentiability of first and second order of the function $u$, at points $x\in L^p$ and in directions $h_1,h_2\in L^q$.

%

If $\varphi$ is a $C^2$ function on $H$ satisfying Assumption \ref{ass:phi} with $p=2=q=2$, using standard arguments, we can prove similar estimates on $D^ku_\delta$, $k=1,2,3$ with $\beta=\gamma=0$ for $x\in H$ and $h,h_1,h_2,h_3\in H$. Thus in this case, we can prove that $u$ is a $C^2$ function on $H$.

\subsection{Weak convergence of numerical approximations}\label{sec:results_num}


As an application of the results of Section~\ref{sec:results_Kolmogorov}, we study the discretization of~\eqref{eq:SPDE} by the following semi-implicit Euler scheme (also known as the linear implicit Euler scheme). Let $T\in(0,\infty)$ be given, and let $\Delta t\in(0,T)$ denote the time-step size of the scheme, such that $N=\frac{T}{\Delta t}\in\N^\star$ is an integer.

Then for $n\in\left\{0,\ldots,N-1\right\}$, define
\begin{equation}\label{eq:scheme_result}
X_{n+1}-X_n=\Delta t\bigl(AX_{n+1}+G(X_n)\bigr)+\sigma(X_n)\bigl(W\bigl((n+1)\Delta t\bigr)-W(n\Delta t)\bigr),\; X_0=x.
\end{equation}
The nonlinear terms $G$ and $\sigma$ are treated explicitly (which is possible thanks to global Lipschitz continuity assumptions), whereas the linear operator $A$ is treated implicitly. Note that~\eqref{eq:scheme_num} can be rewritten in an explicit form
\[
X_{n+1}=S_{\Delta t}X_n+\Delta tS_{\Delta t}G(X_n)+S_{\Delta t}\sigma(X_n)\bigl(W\bigl((n+1)\Delta t\bigr)-W(n\Delta t)\bigr),
\]
where 
\begin{equation}\label{e32}
S_{\Delta t}=\bigl(I-\Delta t A\bigr)^{-1}. 
\end{equation}
This proves the well-posedness of the scheme, thanks to nice regularization properties of $S_{\Delta t}$, see Lemmas~\ref{lem:S_1} and~\ref{lem:S_2}.

The weak convergence result is given by Theorem~\ref{theo:num}; its proof is given in Section~\ref{sec:proof_num}. It generalizes the statement that the weak rate, equal to $\frac12$, is twice the strong order $\frac14$, which has been obtained for instance in~\cite{Printems:01}. Recall that the values of  $p,q$ and $K$ are determined by Assumption~\ref{ass:phi}. 
\begin{theo}\label{theo:num}

For every $\kappa\in(0,\frac12)$, $T\in(0,\infty)$ and every $\Delta t_0\in(0,1)$, there exists $C_\kappa(T,\Delta t_0,\varphi)$, such that for every $\Delta t\in(0,\Delta t_0)$, with $N=\frac{T}{\Delta t}\in\N^\star$, for every $x\in L^p\cap L^{\blue{8q}}$
\begin{equation}\label{eq:theo_num}
\big|\E\varphi\bigl(X(T)\bigr)-\E\varphi\bigl(X_N\bigr)\big|\le C_\kappa(T,\Delta t_0,\varphi)\bigl(1+|x|_{L^{\max(p,8q)}}\bigr)^{K+3}\Delta t^{\frac12-\kappa}.
\end{equation}
\end{theo}

The proof is a generalization of \cite{Debussche:11}, with several non trivial modifications, due to the assumptions made on the nonlinear drift term and diffusion coefficients. In this article, we work in $L^p$ spaces, and it seems that it is the first time that a weak convergence result is provided for SPDEs with Burgers type drift coefficients, {\it i.e.} with a spatial derivative in the drift nonlinear term. More importantly, our main contribution is the treatment of non constant diffusion coefficients $\sigma$ (the multiplicative noise case), under realistic assumptions. In particular, we drop the artificial assumption on $\sigma$ from~\cite{Debussche:11}.

As mentioned in the introduction, the approach using mild It\^o calculus, see~\cite{ConusJentzenKurniawan:14}, \cite{Hefter_Jentzen_Kurniawan:16}, \cite{JacobeJentzenWelti:15}, \cite{JentzenKurniawan:15}, has also recently been able to deal with such non constant diffusion coefficients. The main difference is in the way the discretization error is analyzed: our approach is in our opinion somewhat simpler, and closer to the standard approaches from finite dimensional cases. We require also lower regularity on the drift and diffusion terms.

Our proof is based on a decomposition of the error depending on the solution $u$ of the Kolmogorov equation. In particular, Theorem~\ref{theo:D1} (to handle Burgers type nonlinear drift coefficients), resp.~Theorem~\ref{theo:D2} (to handle nonlinear diffusion coefficients), removing the condition $\beta<\frac12$, resp. the condition $\beta+\gamma<\frac12$, are essential tools.

\section{Proofs of {Theorems~\ref{theo:D1} and~\ref{theo:D2}}}\label{sec:proof_reg}

{ The aim of this section is to provide the proofs of the new regularity results, Theorems~\ref{theo:D1} and~\ref{theo:D2}.

The key ideas of our original approach are explained in Section~\ref{sec:formal}, however only at a formal level: indeed the {stochastic integrals in~\eqref{eq:two-sided} below involve anticipative integrands and are not well defined so that \eqref{eq:two-sided} cannot be used. These integrals are in fact two-sided integrals and should be defined appropriately. This would considerably lengthen our article}. A discrete time approximation is used to make the analysis rigorous: it is introduced in Section~\ref{sec:discrete-approx}.

In addition to the auxiliary temporal discretization, with parameter $\Delta t$, another approximation is used, with parameter $\tau$. The most difficult part of the proof is to obtain the auxiliary regularity results which are stated in Section~\ref{sec:auxiliary}, for positive $\tau$. The proofs of these results are performed in three steps. First, the new expressions suggested by the formal arguments of Section~\ref{sec:formal} are rigorously derived at the discrete time level in Section~\ref{sec:derivatives}. The key ingredients are discrete time versions of the two-sided stochastic integrals, and of an appropriate Malliavin calculus duality formula. Second, Section~\ref{sec:control} is devoted to proving bounds for the terms appearing in these new expressions. Finally, it remains to pass to the limit $\Delta t\to 0$. 

The proof of Theorems~\ref{theo:D1} and~\ref{theo:D2} is then concluded in Section~\ref{sec:proof_tau}, getting rid of the auxiliary parameter $\tau$. Finally, technical auxiliary lemmas are proved in Section~\ref{sec:proof_aux_lemmas}.
}

\subsection{Formal arguments in continuous time}\label{sec:formal}

In this section, we explain how Theorems~\ref{theo:D1} and~\ref{theo:D2} {could be obtained if one first constructs suitable stochastic integrals}. We first recall the origins of the limitations on parameters $\beta$ and $\gamma$ in standard approaches. We then present the strategy of the proof, in particular what are the two-sided stochastic integrals that are required.

As explained in the introduction, we do not intend to give a rigorous meaning in the continuous time setting to the objects introduced below, and do not justify the computations. { As will be clear below, some expressions do not make sense as standard objects.} In order to simplify the presentation, since we want to focus on the difficulties due to the diffusion coefficient $\sigma$ being non constant, in this section we assume that $F_1=F_2=0$. Moreover, we work in an abstract setting: we assume that the diffusion coefficient $\sigma$ is a function on $H$ of class $\mathcal{C}^2$, with bounded derivatives -- this property not being true for the Nemystkii coefficients considered in this paper.  We also assume that the test function $\varphi$ is of class $\mathcal{C}_b^2$.

First, differentiating \eqref{eq:u}, we obtain for $h\in H$:
$$
Du(t,x).h = \E\bigl[ D\varphi(X(t,x)).\eta^{h,x}(t) \bigr]
$$
where $\eta^{h,x}(t)$ is the solution of 
\[
d\eta^{h,x}(t)= A\eta^{h,x}(t)dt + \sigma'(X(t,x)).\eta^{h,x}(t)dW(t), \eta^{h,x}(0)=h.
\]
Using the mild formulation of $\eta^{h,x}(t)$ and It\^o isometry,
\begin{align*}
\E|\eta^{h,x}(t)|^2&=\E\left| e^{tA}h + \int_0^t e^{(t-s)A}\sigma'(X(s,x)).\eta^{h,x}(s) dW(s)\right|^2\\
&=\big|e^{tA}h\big|^2+\int_0^t \big|e^{(t-s)A}\sigma'(X(s,x)).\eta^{h,x}(s)\big|_{\mathcal{L}_2(H)}^2 ds\\
&\le Ct^{-2\beta}|(-A)^{-\beta}h|^{2}+C\int_{0}^{t}(t-s)^{-\frac12-\kappa}\E|\eta^{h,x}(s)|^2 ds.
\end{align*}
Indeed, for $y,h\in H$ and $\kappa\in(0,\frac12)$
$$
\left|e^{A(t-s)} \sigma'(y).k\right|_{{\mathcal L}_2(H)}^2 \le \left|e^{A(t-s)}\right|_{{\mathcal L}_2(H)}^2\left| \sigma'(y).k\right|_{{\mathcal L}(H)}^2 \le C\sum_{i\in\N^\star} e^{-2\lambda_i(t-s)} |k|^2
\le c (t-s)^{-\frac12-\kappa} |k|^2,
$$
since $\sum_{i\in\N^\star}\lambda_i^{-\frac12-\kappa}<\infty$. Assuming that $2\beta<1$, and applying Gronwall Lemma, we get
\begin{equation*}
\sup_{t\in (0,T]} t^{2\beta} \E\left(\left|\eta^{h,x}(t)\right|^2\right) \le C |(-A)^{-\beta}h|^2,
\end{equation*}
which then yields the required regularity result, for $\beta\in[0,\frac12)$:
$$
|Du(t,x)\cdot h |\le c \| \varphi\|_{1}  t^{-\beta}|(-A)^{-\beta}h|.
$$
The limitation $\beta<\frac12$ in previous articles thus comes from the fact that  It\^o formula is used  to control the stochastic integral, and naturally squares appear in integrals. In the additive noise case, since $\sigma'=0$, no stochastic integral appears in the definition of $\eta^{h,x}(t)$, and thus choosing $\beta\in[0,1)$ is possible.

A similar difficulty appears for the second order derivative: differentiating twice~\eqref{eq:u}, for $h,k\in H$ yields
\[
D^2u(t,x).(h,k) = \E\left[ D^2\varphi(X(t,x)).\bigl(\eta^{h,x}(t),\eta^{k,x}(t)\bigr)+ 
D\varphi(X(t,x)).\zeta^{h,k,x}(t) \right],
\]
where $\zeta^{h,k,x}(t)$ is the solution of
\[
d\zeta^{h,k,x}(t)=A\zeta^{h,k,x}(t)dt 
+\sigma'(X(t,x)).\zeta^{h,k,x}(t)dW(t)+\sigma''(X(t,x)).\bigl(\eta^{h,x}(t),\eta^{k,x}(t)\bigr)dW(t),
\]
with the initial condition $\zeta^{h,k,x}(0)=0$. The issue lies again in the control of the stochastic integral: indeed, It\^o isometry for the mild formulation of the equation gives
\[
\E|\zeta^{h,k,x}(t)|^2\le C\int_{0}^{t}(t-s)^{-\frac12-\kappa}\bigl(\E|\zeta^{h,k,x}(s)|^2+\E\bigl[|\eta^{h,x}(s)|^2|\eta^{k,x}(s)|^2\bigr]\bigr)ds,
\]
and, generalizing the previous estimate on $\eta$ to handle the fourth moment, we have
\[
\E\bigl[|\eta^{h,x}(s)|^2|\eta^{k,x}(s)|^2\bigr]\le Cs^{-2\beta-2\gamma}|(-A)^{-\beta}h|^2|(-A)^{-\gamma}k|^2,
\]
and $\int_{0}^{t}(t-s)^{-\frac12-\kappa}s^{-2\beta-2\gamma}ds<\infty$ if and only if $\beta+\gamma<\frac12$. Under this condition, we obtain
\[
\big|D^2u(t,x).(h,k)\big|\le Ct^{-\beta-\gamma}|(-A)^{-\beta}h||(-A)^{-\gamma}k|.
\]

\bigskip

In order to overcome the limitations on $\beta$ and $\gamma$, we introduce new formulas for $Du$ and for $D^2u$. The idea is to use { the Malliavin calculus duality formula}, in order to replace stochastic It\^o integrals, which require square integrability in time, with { integrals with respect to Lebesgue measure}, which require only integrability in time.

First, define $\tilde{\eta}^{h,x}(t)=\eta^{h,x}(t)-e^{tA}h$, and write
\begin{equation}\label{e31}
Du(t,x).h=\E\bigl(D\varphi(X(t,x)).e^{tA}h+D\varphi(X(t,x)).\tilde{\eta}^{h,x}(t)\bigr).
\end{equation}
The first term on the right-hand side is easily bounded by $t^{-\beta}|(-A)^{-\beta}h|$, for $\beta\in[0,1)$. To control the second term, note that
\[
d\tilde{\eta}^{h,x}(t)=\Bigl(A\tilde{\eta}^{h,x}(t)dt+\sigma'(X(t,x)).\tilde{\eta}^{h,x}dW(t)\Bigr)+\sigma'(X(t,x)).e^{tA}h dW(t).
\]
Formally, $\zeta^{h,k,x}$ and $\tilde{\eta}^{h,x}$ are the solutions of the same type of equations, { and we would like to write the following identities} 
\begin{equation}\label{eq:two-sided}
\begin{gathered}
\tilde{\eta}^{h,x}(t)=\int_{0}^{t}\Pi(t,s)\sigma'(X(s,x)).e^{sA}h dW(s),\\
\zeta^{h,x}(t)= \int_0^t \Pi(t,s)\sigma''(X(s,x))\cdot (\eta^{h,x}(s),\eta^{h,x}(s)) 
 dW(s),
\end{gathered}
\end{equation}
where $\Pi(t,s)$ is the evolution operator associated with the linear equation
\[
dZ_{t,s}=AZ_{t,s}dt+\sigma'(X(t,x)).Z_{t,s}dW(t) \quad,\quad Z_{s,s}=z,
\]
{\it i.e.} $\Pi(t,s)z=Z_{t,s}$. 

{Formulas \eqref{eq:two-sided} are not well defined since the integrals contain anticipative integrands. Unfortunately, Skohorod integrals or  
 two-sided stochastic integrals do not work. If we use such integrals in \eqref{eq:two-sided}, the formula makes sense but does not provide a solution of the equations. 
 In fact, we can guess what would be the correct integral for our purpose. From the discrete formulas below, it should be:
 $$
 \lim_{\delta\to 0} \sum_{\ell}\Pi(t,t_{\ell+1})\sigma'(X(t_\ell,x)).e^{t_\ell A}h \left(W(t_\ell+1)-W(t_\ell)\right)
 $$
 for the first one and a similar expression for the other. As usual the limit is taken on subdivisions of $[0,t]$ with $\delta=\max (t_{\ell+1}-t_\ell) $ converging 
 to $0$.}

As explained in the introduction, it may be possible to adapt the arguments from~\cite{Alos_Nualart_Viens:00}, \cite{Nualart_Pardoux:88}, \cite{Nualart_Viens:00}, \cite{Pardoux:87} and~\cite{Pardoux_Protter:87} and give a rigorous meaning to~\eqref{eq:two-sided} {using such new integral}. This is not the strategy we follow; instead, we work on time-discrete approximations of the problem, for which every object is easily defined and only standard tools of stochastic analysis are used.

{Let us anyway go on with the formal argument and show why \eqref{eq:two-sided} is useful. We consider the second term in \eqref{e31} and rewrite using this formula:
$$
\E\bigl(D\varphi(X(t,x)).\tilde{\eta}^{h,x}(t)\bigr)= \E\left(D\varphi(X(t,x)). \int_{0}^{t}\Pi(t,s)\sigma'(X(s,x)).e^{sA}h dW(s)  \right).
$$
The next step is to apply a Malliavin duality formula. This would replace  
the right hand side above by  $\E\left(\int_0^t {\mathcal D}_s \left(D\varphi(X(t,x))\right).\Pi(t,s)\sigma'(X(s,x)).e^{sA}h ds  \right).
$
Thus, we are now dealing with a standard integral and do not need square integrability in time and a higher order singularity is allowed. Unfortunately, this is not correct: the integral has not been defined and no generalization of the duality formula has been proved. In fact, we may guess from the 
discrete computations below that the correct formula should contain an additional term and write:
$$
\begin{array}{ll}
\displaystyle \E\bigl(D\varphi(X(t,x)).\tilde{\eta}^{h,x}(t)\bigr)&\displaystyle = \E\bigg(\int_0^t {\mathcal D}_s \left(D\varphi(X(t,x))\right).\Pi(t,s)\sigma'(X(s,x)).e^{sA}h ds \\
&\displaystyle +\int_0^t D\varphi(X(t,x)). {\mathcal D}_s \left(\Pi(t,s)\right)\sigma'(X(s,x)).e^{sA}h ds  \bigg).
\end{array}
$$
Again, we have chosen to avoid the rigorous construction of the two-sided integral and the proof of the associate duality formula by working on discrete time 
approximations.}



\subsection{Discrete time approximation}\label{sec:discrete-approx}

In order to give a rigorous meaning to the arguments presented above in Section~\ref{sec:formal}, we replace the continuous time processes $\bigl(X^\delta(t)\bigr)_{t\in[0,T]}$, with $\delta \in[0,1)$, with discrete-time approximations. We use a numerical scheme, with time-step size $\Delta t=\frac{T}{N}\in(0,1)$, with $N\in\N^\star$. We prove regularity results for fixed $N$, with upper bounds not depending on $N$, and finally pass to the limit $N\to \infty$.

We also require an additional regularization parameter, $\tau\in(0,1)$. Some estimates depend on $\tau$; when it is the case, it will always be stated precisely.

The discrete-time processes are defined using the linear-implicit Euler scheme: for $0\le n\le N-1$
\begin{equation}\label{eq:scheme}
X_{n+1}^{\delta,\tau,\Delta t}=S_{\Delta t}X_{n}^{\delta,\tau,\Delta t}+\Delta t S_{\Delta t}G_\delta\bigl(X_{n}^{\delta,\tau,\Delta t}\bigr)+e^{\tau A}S_{\Delta t}\sigma_\delta\bigl(X_{n}^{\delta,\tau,\Delta t}\bigr)\Delta W_n,
\end{equation}
with the standard notation $\Delta W_n=W\bigl((n+1)\Delta t\bigr)-W\bigl(n\Delta t\bigr)$, and $S_{\Delta t}=(I-\Delta t A)^{-1}$. Note that we have added the regularization operator in the diffusion coefficient: $e^{\tau A}$. 

{Below we omit to write the dependance on $\delta,\tau,\Delta t$ and write $X_n$ instead of $X_{n}^{\delta,\tau,\Delta t}$. All constants are independent on 
$\delta,\tau,\Delta t$.} 
{ Moment estimates for $X_n$ are given by Lemma~\ref{lem:moments_X} below.}
\begin{lemma}\label{lem:moments_X}
For every $p\in[2,\infty)$, $\alpha\in[0,\frac14)$, $M\in\N^\star$ and $T\in(0,\infty)$, there exists $C(p,M,T)$, such that for every $n\in\left\{1,\ldots,N\right\}$ (with $N\Delta t=T$), and every $x\in D_p\bigl((-A)^\alpha\bigr)$
\begin{equation}\label{eq:lem_moments_X}
\begin{aligned}
\E|(-A)^{\alpha}X_n(x)|_{L^p}^{2M}&\le C(p,\alpha,M,T)\bigl(1+t_{n}^{-2M\alpha}|x|_{L^p}^{2M}\bigr),\\
\E|(-A)^{\alpha}X_n(x)|_{L^p}^{2M}&\le C(p,\alpha,M,T)\bigl(1+|(-A)^\alpha x|_{L^p}^{2M}\bigr).
\end{aligned}
\end{equation}
\end{lemma}

{ The proof of Lemma~\ref{lem:moments_X} uses the two following results.} 
\begin{lemma}\label{lem:S_1}
For every $\beta\in[0,1)$ and $p\in[2,\infty)$, there exists $C(p,\beta)$ such that for every $n\in\N^\star$
\begin{equation*}
|(-A)^{\beta}S_{\Delta t}^{n}|_{\mathcal{L}(L^p)}\le \frac{C(p,\beta)}{t_{n}^{\beta}}.
\end{equation*}
\end{lemma}

\begin{lemma}\label{lem:S_2}
For every $\beta\in[0,\frac34)$, $p\in[2,\infty)$, and $\kappa\in(0,\frac34-\beta)$, there exists $C_\kappa(p,\beta)$ such that for every $n\in\N^\star$
\begin{equation*}
\big|(-A)^{\beta}S_{\Delta t}^{n}\big|_{R(L^2,L^p)}\le \frac{C_\kappa(p,\beta)}{t_{n}^{\frac14+\beta+\kappa}}.
\end{equation*}
\end{lemma}

Both results in the case $p=2$ are obtained by straightforward computations, thanks to expansions using the eigenbasis $\bigl(e_n\bigr)_{n\in\N^\star}$ of $A$. 
When $p\in(2,\infty)$, the arguments use properties of the analytic semigroup $\bigl(e^{tA}\bigr)_{t\ge 0}$ in $L^p$. The proofs are given below since these results are not standard in the literature for SPDEs. Arguments from~\cite{Pazy} are used. The results are in fact valid for $p\in(1,\infty)$.

\begin{proof}[Proof of Lemma~\ref{lem:S_1}]

The case $\beta=0$ follows from the two inequalities $|S_{\Delta t}|_{\mathcal{L}(L^2)}\le 1$ (which is proved using expansions in the Hilbert space $L^2$ with the complete orthonormal system $\bigl(e_k\bigr)_{k\in \N^\star}$) and $|S_{\Delta t}|_{\mathcal{L}(L^\infty)}\le 1$. By a standard interpolation argument, we thus have $\big|S_{\Delta t}^{n+1}\big|_{\mathcal{L}(L^p)}\le 1$ for every $p\in[2,\infty]$.

Define the resolvent $R(\lambda,A)=\int_{0}^{\infty}e^{-\lambda t}e^{tA}dt$, for $\lambda\in(0,\infty)$. Then $S_{\Delta t}=\frac{1}{\Delta t}R(\frac{1}{\Delta t},A)$. 
First, for $x\in L^p$, we set $y=S_{\Delta t}x$. Then $|y|_{L^q}\le |x|_{L^q}$, and $Ay= \frac{1}{\Delta t} (y-x)$. We thus obtain
$$
|Ay|_{L^q}=|AS_{\Delta t} x |_{L^q}\le \frac2{\Delta t} |x|_{L^q}.
$$

Second, when $n\in\N^\star$,
\begin{align*}
n! \big|(-A) R(\lambda,A)^{n+1}x\big|_{L^p}&=\big|\int_{0}^{\infty}t^n e^{-\lambda t}(-A) e^{tA}x dt\big|_{L^p}\\
&\le C(p,\beta)\int_{0}^{\infty}e^{-\lambda t}t^{n-1}dt |x|_{L^p}\\
&\le C(p,\beta) (n-1)! \lambda^{-n}|x|_{L^p}.
\end{align*}
This gives $\big|(-A)S_{\Delta t}^{n+1}\big|_{\mathcal{L}(L^p)}\le \frac{C(p,\beta)}{(n+1)\Delta t}$, for $n\in \N^\star$. Thus the result is proved 
for $\beta=1$. The case $\beta\in [0,1)$ follows by an interpolation argument {(see \cite{Triebel}, Theorem 1.15.3)}:
\[
\big|(-A)^{\beta}S_{\Delta t}^{n+1}x\big|_{L^p}\le c_p \big|(-A)S_{\Delta t}^{n+1}x\big|_{L^p}^{\beta}\big|S_{\Delta t}^{n+1}x\big|_{L^p}^{1-\beta}\le  \frac{C(p,\beta)}{\bigl((n+1)\Delta t\bigr)^{\beta}}|x|_{L^p}.
\]
This concludes the proof of Lemma~\ref{lem:S_1}.
\end{proof}

\begin{proof}[Proof of Lemma~\ref{lem:S_2}]

Let $\bigl(\tilde{\gamma}_k\bigr)_{k\in\N^\star}$ denote a sequence of independent standard real-valued Gaussian random variables, $\tilde{\gamma}_k\sim \mathcal{N}(0,1)$.

Then, using standard properties concerning moments of Gaussian random variables,
\begin{align*}
\big|(-A)^{\beta}S_{\Delta t}^{n}\big|_{R(L^2,L^p)}^2&=\E\big|\sum_k \gamma_k(-A)^\beta S_{\Delta t}^{n}e_k\big|_{L^p}^{2}\\
&\le \Bigl(\E\big|\sum_k \gamma_k (-A)^\beta S_{\Delta t}^{n}e_k\big|_{L^p}^{p}\Bigr)^{\frac2p}\\
&\le \Bigl(\int_{0}^{1}\E\big|\sum_k \lambda_k^\beta \frac{1}{(1+\lambda_k\Delta t)^n}e_k(\xi) \gamma_k\big|^p d\xi \Bigr)^{\frac2p}\\
&\le \Bigl(\int_{0}^{1}\bigl(\E\big|\sum_k \lambda_k^\beta \frac{1}{(1+\lambda_k\Delta t)^n}e_k(\xi) \gamma_k\big|^2\bigr)^{\frac{p}2} d\xi \Bigr)^{\frac2p}\\
&\le \Bigl(\int_{0}^{1}\bigl(\sum_k \lambda_k^{2\beta} \frac{1}{(1+\lambda_k\Delta t)^{2n}}e_k(\xi)^2\bigr)^{\frac{p}2} d\xi \Bigr)^{\frac2p}.
\end{align*}

Using Property~\ref{ass:A}, and the estimate
\[
\sum_k \lambda_k^{2\beta} \frac{1}{(1+\lambda_k\Delta t)^{2n}}{=} \sum_k \lambda_{k}^{-\frac12-2\kappa}\big|(-A)^{\frac14+\kappa+\beta}S_{\Delta t}^n e_k|_{L^2}^{2}\le C_\kappa t_{n}^{-\frac12-2\kappa-2\beta},
\]
which follows from Lemma~\ref{lem:S_1}, we get the result.

\end{proof}


\begin{proof}[Proof of Lemma~\ref{lem:moments_X}]
First, note that, for $0\le n\le N$,
\begin{equation}
X_n=S_{\Delta t}^{n}x+\Delta t\sum_{k=0}^{n-1}S_{\Delta t}^{n-k}BG(X_k)+\sum_{k=0}^{n-1}S_{\Delta t}^{n-k}e^{\tau A}\sigma(X_k)\Delta W_k.
\end{equation}

Thanks to Property~\ref{ass:F}, inequality~\eqref{eq:norm_AB}, and Lemmas~\ref{lem:S_1} and~\ref{lem:S_2}, we have for $\kappa>0$ such that $2\alpha+2\kappa<\frac12$:
\begin{align*}
\E|(-A)^\alpha X_n|_{L^p}^2&\le C|(-A)^{\alpha}S_{\Delta t}^{n}x|_{L^p}^{2}+C\bigl(\Delta t\sum_{k=0}^{n-1} |(-A)^{\alpha+\frac12+\kappa}S_{\Delta t}^{n-k}|_{\mathcal{L}(L^p)} \bigr)^2\\
&+C\E\big|\sum_{k=0}^{n-1}(-A)^{\alpha}S_{\Delta t}^{n-k}e^{\tau A}\sigma(X_k)\Delta W_k\big|_{L^p}^2\\
&\le Ct_{n}^{-2\alpha}|x|_{L^p}^{2}+C\bigl(\Delta t\sum_{\ell=0}^{n-1}t_{n-k}^{-\frac12-2\alpha-\kappa}\bigr)^2+C\Delta t\sum_{k=0}^{n-1}\E\big|(-A)^{\alpha}S_{\Delta t}^{n-k}e^{\tau A}\sigma(X_k)\big|_{R(L^2,L^p)}^{2}\\
&\le C(1+t_{n}^{-2\alpha}|x|_{L^p}^{2})+C\Delta t\sum_{k=0}^{n-1}\big|(-A)^{\alpha}S_{\Delta t}^{n-k}\big|_{R(L^2,L^p)}^{2}\E\big|e^{\tau A}\sigma(X_k)\big|_{\mathcal{L}(L^2)}^{2}\\
&\le C\bigl(1+t_{n}^{-2\alpha}|x|_{L^p}^{2}+\Delta t\sum_{k=0}^{n-1}t_{n-k}^{-\frac12-2\alpha-2\kappa}\bigr).
\end{align*}

This proves~\eqref{eq:lem_moments_X} in the case $M=1$. The case $M\ge 1$ and the second estimate of~\eqref{eq:lem_moments_X} are obtained with similar computations combined with standard arguments. This concludes the proof of Lemma~\ref{lem:moments_X}.
\end{proof}

{
\subsection{Regularity results for an auxiliary process}\label{sec:auxiliary}

The objective of this section is to state regularity results for the first and second order spatial derivatives of the function $u^{\delta,\tau}$ defined by
\begin{equation}\label{eq:u_aux_delta-tau}
u^{\delta,\tau}(t,x)=\E\bigl[\varphi_\delta\bigl(X^{\delta,\tau}(t,x)\bigr)\bigr],
\end{equation}
where $\tau$ is an auxiliary regularization parameter, and the process $X^{\delta,\tau}$ is solution of the SPDE
\begin{equation}\label{eq:SPDE_reg-tau}
dX_t^{\delta,\tau}=AX_t^{\delta,\tau} dt+G_\delta(X_t^{\delta,\tau})dt+e^{\tau A}\sigma_\delta(X_t^{\delta,\tau})dW(t) , \quad X^{\delta,\tau}(0)=x.
\end{equation}

Let us also define the function $u^{\delta,\tau,\Delta t}:\left\{0,\Delta t,\ldots,(N-1)\Delta t,N\Delta t=T\right\}\times H\to \R$, by
\begin{equation}\label{eq:u_aux}
u^{\delta,\tau,\Delta t}(n\Delta t,x)=\E\bigl[\varphi_\delta\bigl(X_{n}^{\delta,\tau,\Delta t}(x)\bigr)\bigr],
\end{equation} 
where $X_{n}^{\delta,\tau,\Delta t}(x)$ is the solution of~\eqref{eq:scheme} with initial condition $x$.

Note that when $\tau\to 0$, $X^{\delta,\tau}$ converges to $X^{\delta}$, for all $\delta >0$. In addition, the discrete time process defined by~\eqref{eq:scheme} is obtained by temporal discretization of $X^{\delta,\tau}$. Sections~\ref{sec:derivatives} and~\ref{sec:control} are devoted to proving new regularity estimates for $Du^{\delta,\tau,\Delta t}(N\Delta t,x)$ and $D^2u^{\delta,\tau,\Delta t}(N\Delta t,x)$, with $T=N\Delta t$. We omit writing these expressions, which contain many terms vanishing in the limit $\Delta t\to 0$. Indeed, passing to the limit $\Delta t\to 0$, the following regularity results for the auxiliary function $u^{\delta,\tau}$ are obtained.

\begin{propo}\label{propo:u_delta_tau_1}
For every $\beta\in [0,1)$ and $\kappa\in(0,1)$, there exists $C_{\beta,\kappa}(T)$, such that for every $\delta,\tau\in(0,1)$, $x\in L^p$ and $h\in L^{2q}$
\begin{equation}\label{eq:propo_u_delta_tau_1}
\big|Du^{\delta,\tau}(T,x).h\big|\le \frac{C_{\beta,\kappa}(T)}{\tau^{\kappa} T^{\beta}}\bigl(1+|x|_{L^{\max(p,2q)}}^{K+1}\bigr)|(-A)^{-\beta}h|_{L^{2q}}.
\end{equation}
\end{propo}

\begin{propo}\label{propo:u_delta_tau_2}
For every $\beta,\gamma\in [0,\frac12)$ and $\kappa\in(0,1)$, there exists $C_{\beta,\gamma,\kappa}(T)$ such that for every $\delta,\tau \in(0,1)$, $x\in L^p$ and $h,k\in L^{4q}$
\begin{equation}\label{eq:propo_u_delta_tau_2}
\big|D^2u^{\delta,\tau}(T,x).\bigl(h,k\bigr)\big|\le \frac{C_{\beta,\gamma,\kappa}(T)}{\tau^{\kappa} T^{\beta+\gamma}}\bigl(1+|x|_{L^{\max(p,2q)}}^{K+1}\bigr)|(-A)^{-\beta}h|_{L^{4q}}|(-A)^{-\gamma}h|_{L^{4q}}.
\end{equation}
\end{propo}

Observe that the right-hand sides of~\eqref{eq:propo_u_delta_tau_1} and~\eqref{eq:propo_u_delta_tau_2} contain a singular factor $\tau^{-\kappa}$. It is important to note that the exponent $\kappa$ is positive but arbitrarily small. Proofs of Theorems~\ref{theo:D1} and~\ref{theo:D2} require the use of an interpolation argument to get rid of the parameter $\tau$, indeed passing to the limit $\tau\to 0$ is not sufficient. Details are provided in Section~\ref{sec:proof_tau}.

In Section~\ref{sec:derivatives} and~\ref{sec:control} below, the analysis is performed with fixed parameters $\tau>0$ and $\Delta t>0$, {\it i.e.} at a discrete time level.

\subsection{{ New expressions for derivatives in the discrete-time framework}}\label{sec:derivatives}

{ The goal of this section is to derive new expressions for $Du^{\delta,\tau,\Delta t}(t,x).h$ and $D^2u^{\delta,\tau,\Delta t}(t,x).(h,k)$. This is done by repeating the discussion of Section~\ref{sec:formal}, and the formal {formulas} are turned into rigorous ones for the discret objects.


Thanks to the regularity properties of $G=F_1+BF_2$, $\sigma$ and $\varphi$, see Properties~\ref{ass:F},~\ref{ass:sigma} and Assumption ~\ref{ass:phi}, for every $n\in\left\{0,1,\ldots,N\right\}$, $x\in L^2\mapsto u^{\delta,\tau,\Delta t}(t_n,x)$ is of class $\mathcal{C}^2$, and it is straightforward to prove recursively that:
\begin{itemize}
\item the first order derivative satisfies
\begin{equation}\label{eq:Du}
Du^{\delta,\tau,\Delta t}(t_n,x).h=\E\bigl[D\varphi_\delta(X_n(x)).\eta_{n}^{h}\bigr]
\end{equation}
with $\eta_0^h=h$ and, for $n\in\left\{0,\ldots,N-1\right\}$,
\begin{equation}\label{eq:eta}
\eta_{n+1}^{h}=S_{\Delta t}\eta_n^h+\Delta tS_{\Delta t}G_\delta'(X_n).\eta_n^h+S_{\Delta t}e^{\tau A}\bigl(\sigma_\delta'(X_n).\eta_n^h\bigr)\Delta W_n.
\end{equation}
Recall that $S_{\Delta t}$ is defined in \eqref{e32}

\item the second order derivative satisfies
\begin{equation}\label{eq:D2u}
D^2 u^{\delta,\tau,\Delta t}(t_n,x).(h,k)=\E\bigl[D^2\varphi_\delta\bigl(X_{n}(x)\bigr).\bigl(\eta_{n}^{h},\eta_{n}^{k}\bigr)\bigr]+\E\bigl[D\varphi_\delta\bigl(X_{n}(x)\bigr).\zeta_{n}^{h,k}\bigr],
\end{equation}
with $\zeta_0^{h,k}=0$ and, for $n\in\left\{0,\ldots,N-1\right\}$,
\begin{equation}\label{eq:zeta}
\begin{aligned}
\zeta_{n+1}^{h,k}&=S_{\Delta t}\zeta_{n}^{h,k}+\Delta tS_{\Delta t}G_\delta'(X_n).\zeta_{n}^{h,k}+S_{\Delta t}e^{\tau A}\bigl(\sigma_\delta'(X_n).\zeta_{n}^{h,k}\bigr)\Delta W_n\\
&~+\Delta tS_{\Delta t}G_\delta''(X_n).(\eta_n^h,\eta_n^k)+S_{\Delta t}e^{\tau A}\bigl(\sigma_\delta''(X_n).(\eta_n^h,\eta_n^k)\bigr)\Delta W_n.
\end{aligned}
\end{equation}

\end{itemize}

Define the auxiliary process $\bigl(\tilde{\eta}_n^h\bigr)_{0\le n\le N}$, by
\begin{equation}\label{eq:eta_tilde}
\tilde{\eta}_n^{h}=\eta_n^h-S_{\Delta t}^{n}h \quad,\quad \tilde{\eta}_0^h=0.
\end{equation}

Again, in order to simplify the notation, most of the time we do not mention the parameters $\delta,\tau,\Delta t$.

\bigskip

Our objective is to obtain the following estimates, with arbitrarily small $\kappa\in(0,1)$:
\begin{equation}\label{eq:D_discrete}
\begin{gathered}
\big|Du^{\delta,\tau,\Delta t}(T,x).h\big|\le \frac{C_{\beta,\kappa}(T)}{T^\beta\tau^\kappa}(1+|x|_{L^{\max(p,2q)}})^{K+1}|(-A)^{-\beta}h|_{L^{2q}}, \quad \beta\in[0,1),\\
\big|D^2u^{\delta,\tau,\Delta t}(T,x).(h,k)\big|\le \frac{C_{\beta,\gamma,\kappa}(T)}{T^{\beta+\gamma}\tau^\kappa}(1+|x|_{L^{\max(p,2q)}})^{K+1}|(-A)^{-\beta}h|_{L^{4q}}|(-A)^{-\gamma}k|_{L^{4q}}, \quad \beta,\gamma\in[0,\frac12).
\end{gathered}
\end{equation}
Note that the right-hand sides do not depend on $\Delta t$ and on $\delta$. Passing to the limit when these parameters go to $0$ is straightforward. 

\bigskip

{We now do perform similar computations as in section ~\ref{sec:formal} but on the discrete processes so that we do not have to manipulate anticipative 
integrals.}

Define random linear operators $\bigl(\Pi_n\bigr)_{0\le n\le N-1}$ as follows: for every $n\in \left\{0,\ldots,N-1\right\}$ and every $z\in H$
\begin{equation}\label{eq:Pi_n}
\Pi_{n}z=S_{\Delta t}z+\Delta tS_{\Delta t}Be^{\tau A}G'(X_n).z+S_{\Delta t}e^{\tau A}\bigl(\sigma'(X_n).z\bigr)\Delta W_n.
\end{equation}
Note that $\Pi_{n}=\Pi(X_n,\Delta W_n)$ with the deterministic linear operators $\Pi(x,w)$ defined by
\begin{equation*}
\Pi(x,w)z=S_{\Delta t}z+\Delta tS_{\Delta t}Be^{\tau A}F'(x).z+S_{\Delta t}e^{\tau A}\bigl(\sigma'(x).z\bigr)w.
\end{equation*}
We emphasize on the following key observation: $\Pi_n$ depends on the Wiener increments $\Delta W_0,\ldots,\Delta W_{n-1}$ only through the first variable of $\Pi(\cdot,\cdot)$, and depends on $\Delta W_n$ only through its second variable.

Introduce the notation $\Pi_{n-1:\ell}=\Pi_{n-1}\ldots\Pi_{\ell}$ for $\ell\in\left\{0,\ldots,n-1\right\}$, and by convention $\Pi_{n-1:n}=I$. These operators are the discrete versions of the evolution operators $\Pi(t,s)$  introduced in Section~\ref{sec:formal}.

\bigskip

Recursion formulas for ${\eta}_{\cdot}^h$, $\tilde{\eta}_{\cdot}^h$ and $\zeta_{\cdot}^{h,k}$, are rewritten in the following forms:
\begin{equation}\label{eq:recursions_Pi}
\begin{gathered}
\eta_{n+1}^{h}=\Pi_n\eta_n^h \quad,\quad \eta_0^h=h,\\
\tilde{\eta}_{n+1}^{h}=\Pi_n\tilde{\eta}_n^h+\Delta tS_{\Delta t}G'(X_n).S_{\Delta t}^n h+S_{\Delta t}e^{\tau A}\bigl(\sigma'(X_n).S_{\Delta t}^nh)\Delta W_n,\\
\zeta_{n+1}^{h,k}=\Pi_n\zeta_{n}^{h,k}+\Delta tS_{\Delta t}G''(X_n).(\eta_n^h,\eta_n^k)+S_{\Delta t}e^{\tau A}\bigl(\sigma''(X_n).(\eta_n^h,\eta_n^k)\bigr)\Delta W_n
\end{gathered}
\end{equation}

A straightforward consequence of the first equality in~\eqref{eq:recursions_Pi} is the equality
\begin{equation}\label{eq:eta_Pi}
\eta_n^h=\Pi_{n-1:0}h,
\end{equation}
for every $n\in\left\{0,\ldots,N\right\}$. Moreover, we get the following discrete-time analogs of~\eqref{eq:two-sided}, now taking into account also nonlinear drift terms:
\begin{equation}\label{eq:two-sided_discrete}
\begin{gathered}
\tilde{\eta}_n^h=\Delta t\sum_{\ell=0}^{n-1}\Pi_{n-1:\ell+1}S_{\Delta t}G'(X_\ell).S_{\Delta t}^\ell h+\sum_{\ell=0}^{n-1}\Pi_{n-1:\ell+1}S_{\Delta t}e^{\tau A}\bigl(\sigma'(X_\ell).S_{\Delta t}^\ell h)\Delta W_\ell,\\
\zeta_n^{h,k}=\Delta t\sum_{\ell=0}^{n-1}\Pi_{n-1:\ell+1}S_{\Delta t}G''(X_\ell).(\eta_\ell^h,\eta_\ell^k)+\sum_{\ell=0}^{n-1}\Pi_{n-1:\ell+1}S_{\Delta t}e^{\tau A}\bigl(\sigma''(X_\ell).(\eta_\ell^h,\eta_\ell^k)\bigr)\Delta W_\ell.
\end{gathered}
\end{equation}

We treat separately the contributions of the drift and  diffusion terms and introduce
\begin{equation}\label{eq:eta_zeta}
\begin{gathered}
\tilde{\eta}_n^{h,1}=\Delta t\sum_{\ell=0}^{n-1}\Pi_{n-1:\ell+1}S_{\Delta t}G'(X_\ell).S_{\Delta t}^\ell h \quad,\quad
\tilde{\eta}_n^{h,2}=\sum_{\ell=0}^{n-1}\Pi_{n-1:\ell+1}S_{\Delta t}e^{\tau A}\bigl(\sigma'(X_\ell).S_{\Delta t}^\ell h)\Delta W_\ell;\\
\zeta_n^{h,k,1}=\Delta t\sum_{\ell=0}^{n-1}\Pi_{n-1:\ell+1}S_{\Delta t}G''(X_\ell).(\eta_\ell^h,\eta_\ell^k) \quad,\quad \zeta_n^{h,k,2}=\sum_{\ell=0}^{n-1}\Pi_{n-1:\ell+1}S_{\Delta t}e^{\tau A}\bigl(\sigma''(X_\ell).(\eta_\ell^h,\eta_\ell^k)\bigr)\Delta W_\ell.
\end{gathered}
\end{equation}


Using $\eta_n^h=S_{\Delta t}^{n}h+\tilde{\eta}_n^{h,1}+\tilde{\eta}_n^{h,2}$, we obtain the decomposition
\begin{equation}\label{eq:du_discrete}
\begin{aligned}
Du^{\delta,\tau,\Delta t}(T,x).h&=\E\bigl[D\varphi(X_N).\bigl(S_{\Delta t}^N h\bigr)\bigr]+\E\bigl[D\varphi(X_N).\tilde{\eta}_{N}^{h,1}\bigr]+\E\bigl[D\varphi(X_N).\tilde{\eta}_{N}^{h,2}\bigr]\\
&=\mathcal{D}_N^{h,0}+\mathcal{D}_N^{h,1}+\mathcal{D}_N^{h,2},
\end{aligned}
\end{equation}
where to simplify the notation we write $\varphi$ instead of $\varphi_\delta$.

We also obtain the following decomposition for the second-order derivative:
\begin{equation}\label{eq:d2u_discrete}
\begin{aligned}
D^2u^{\delta,\tau,\Delta t}(T,x).(h,k)&=\E\bigl[D^2\varphi\bigl(X_{N}\bigr).\bigl(\eta_{N}^{h},\eta_{N}^{k}\bigr)\bigr]+\E\bigl[D\varphi\bigl(X_{N}\bigr).\zeta_{N}^{h,k,1}\bigr]+\E\bigl[D\varphi\bigl(X_{N}\bigr).\zeta_{N}^{h,k,2}\bigr]\\
&=\mathcal{E}_N^{h,k,0}+\mathcal{E}_N^{h,k,1}+\mathcal{E}_N^{h,k,2}.
\end{aligned}
\end{equation}

The term $\mathcal{D}_N^{h,0}$ is straightforward to estimate using Lemma~\ref{lem:S_1}. The terms $\mathcal{E}_N^{h,k,0}$, $\mathcal{D}_N^{h,1}$ and $\mathcal{E}_N^{h,k,1}$ are not very difficult thanks to Lemma~\ref{lem:Pi} stated below.


Finally, the terms $\mathcal{D}_N^{h,2}$ and $\mathcal{E}_N^{h,k,2}$  contain the discretized two-sided stochastic integrals and are treated using { the Malliavin calculus duality formula}.
Note that in the discrete time setting, this {formula} can simply be considered as a standard integration by parts formula in the weighted $L^2$ space corresponding with Gaussian density.

Let us first consider the first order derivative term $\mathcal{D}_N^{h,2}$. Introducing the adjoint $\Pi_{N-1:\ell+1}^\star$ of the operator $\Pi_{N-1:\ell+1}$, we get
\begin{align*}
\mathcal{D}_{N}^{h,2}&=\E\bigl[\langle D\varphi(X_N),\sum_{\ell=0}^{N-1}\Pi_{N-1:\ell+1}S_{\Delta t}e^{\tau A}\bigl(\sigma'(X_\ell).S_{\Delta t}^\ell h\bigr)\Delta W_\ell\rangle\bigr]\\
&=\sum_{\ell=0}^{N-1}\E\bigl[\langle \Pi_{N-1:\ell+1}^{\star} D\varphi\bigl(X_{N}\bigr),\int_{\ell \Delta t}^{(\ell+1)\Delta t}  S_{\Delta t}e^{\tau A}\bigl(\sigma'(X_\ell).S_{\Delta t}^\ell h\bigr)dW(s) \rangle\bigr]\\
&=\sum_{\ell=0}^{N-1}\mathcal{D}_{N,\ell}^{h,2}.
\end{align*}
We now { apply the Malliavin calculus duality formula}, and we get for every $\ell\in\left\{0,\ldots,N-1\right\}$
\begin{align*}
\mathcal{D}_{N,\ell}^{h,2}&=\sum_{i\in\N^\star}\E\int_{\ell \Delta t}^{(\ell+1)\Delta t}\langle \D_s^{e_i}\bigl(\Pi_{N-1:\ell+1}^{\star}D\varphi(X_N)\bigr),e^{\tau A}S_{\Delta t}\bigl(\sigma'(X_\ell).S_{\Delta t}^\ell h\bigr)e_i\rangle ds\\
&=\sum_{i\in\N^\star}\E\int_{\ell \Delta t}^{(\ell+1)\Delta t} \D_s^{e_i}\langle D\varphi(X_N),\Pi_{N-1:\ell+1}e^{\tau A}S_{\Delta t}\bigl(\sigma'(X_\ell).S_{\Delta t}^\ell h\bigr)e_i\rangle ds\\
&=\mathcal{D}_{N,\ell}^{h,2,1}+\mathcal{D}_{N,\ell}^{h,2,2},
\end{align*}
where, thanks to the chain rule,
\begin{equation}\label{eq:decompD}
\begin{gathered}
\mathcal{D}_{N,\ell}^{h,2,1}=\sum_{i\in\N^\star}\E\int_{\ell \Delta t}^{(\ell+1)\Delta t} D^2\varphi(X_n).\Bigl(\D_s^{e_i}X_N,\Pi_{N-1:\ell+1}e^{\tau A}S_{\Delta t}\bigl(\sigma'(X_\ell).S_{\Delta t}^\ell h\bigr)e_i\Bigr) ds,\\
\mathcal{D}_{N,\ell}^{h,2,2}=\sum_{i\in \N^\star}\E\int_{\ell \Delta t}^{(\ell+1)\Delta t}\langle D\varphi(X_N),\D_s^{e_i}\Bigl(\Pi_{N-1:\ell+1}e^{\tau A}S_{\Delta t}\bigl(\sigma'(X_\ell).S_{\Delta t}^\ell h\bigr)e_i\Bigr)\rangle ds.
\end{gathered}
\end{equation}
Similarly, for the second order derivative term $\mathcal{E}_N^{h,k,2}$, we write
\begin{align*}
\mathcal{E}_{N}^{h,k,2}&=\sum_{\ell=0}^{N-1}\E\bigl[\langle D\varphi\bigl(X_{N}\bigr),\Pi_{N-1:\ell+1}S_{\Delta t}e^{\tau A}\bigl(\sigma''(X_\ell).(\eta_\ell^h,\eta_\ell^k)\bigr)\Delta W_\ell \rangle\bigr]\\
&=\sum_{\ell=0}^{N-1}\E\bigl[\langle \Pi_{N-1:\ell+1}^{\star} D\varphi\bigl(X_{N}\bigr),\int_{\ell \Delta t}^{(\ell+1)\Delta t}  S_{\Delta t}e^{\tau A}\bigl(\sigma''(X_\ell).(\eta_\ell^h,\eta_\ell^k)\bigr)dW(s) \rangle\bigr]\\
&=\sum_{\ell=0}^{N-1}\bigl(\mathcal{E}_{N,\ell}^{h,k,2,1}+\mathcal{E}_{N,\ell}^{h,k,2,2}\bigr),
\end{align*}
with
\begin{equation}\label{eq:decompE}
\begin{gathered}
\mathcal{E}_{N,\ell}^{h,k,2,1}=\sum_i\E\int_{\ell \Delta t}^{(\ell+1)\Delta t}D^2\varphi(X_N).\Bigl(\D_s^{e_i}X_N,\Pi_{N-1:\ell+1}e^{\tau A}S_{\Delta t}\bigl(\sigma''(X_{\ell}).(\eta_{\ell}^h,\eta_{\ell}^k)\bigr)e_i\Bigr)ds,\\
\mathcal{E}_{N,\ell}^{h,k,2,2}=\sum_i\E\int_{\ell \Delta t}^{(\ell+1)\Delta t} \langle D\varphi(X_N),\D_s^{e_i}\Bigl(\Pi_{N-1:\ell+1}e^{\tau A}S_{\Delta t}\bigl(\sigma''(X_{\ell}).(\eta_{\ell}^h,\eta_{\ell}^k)\bigr)e_i\Bigr)\rangle ds.
\end{gathered}
\end{equation}

{ For completeness, the new expressions for the first and second order derivatives of $u^{\delta,\tau,\Delta t}$, obtained by our original strategy, are rewritten in the following proposition. They are obtained by inserting~\eqref{eq:decompD} and~\eqref{eq:decompE} in~\eqref{eq:du_discrete} and~\eqref{eq:d2u_discrete}, and using~\eqref{eq:zeta}.
\begin{propo}\label{propo:expressions}
The first and second order derivatives of $u^{\delta,\tau,\Delta t}$ have the following expressions:
\begin{align*}
Du^{\delta,\tau,\Delta t}(T,x).h&=\E\bigl[D\varphi(X_N).\bigl(S_{\Delta t}^N h\bigr)\bigr]+\E\Bigl[D\varphi(X_N).\bigl(\Delta t\sum_{\ell=0}^{n-1}\Pi_{n-1:\ell+1}S_{\Delta t}G'(X_\ell).S_{\Delta t}^\ell h\bigr)\Bigr]\\
&+\sum_{\ell=0}^{N-1}\sum_{i\in\N^\star}\E\int_{\ell \Delta t}^{(\ell+1)\Delta t} D^2\varphi(X_n).\Bigl(\D_s^{e_i}X_N,\Pi_{N-1:\ell+1}e^{\tau A}S_{\Delta t}\bigl(\sigma'(X_\ell).S_{\Delta t}^\ell h\bigr)e_i\Bigr) ds\\
&+\sum_{\ell=0}^{N-1}\sum_{i\in \N^\star}\E\int_{\ell \Delta t}^{(\ell+1)\Delta t}\langle D\varphi(X_N),\D_s^{e_i}\Bigl(\Pi_{N-1:\ell+1}e^{\tau A}S_{\Delta t}\bigl(\sigma'(X_\ell).S_{\Delta t}^\ell h\bigr)e_i\Bigr)\rangle ds,
\end{align*}

\begin{align*}
D^2u^{\delta,\tau,\Delta t}(T,x).(h,k)&=\E\bigl[D^2\varphi\bigl(X_{N}\bigr).\bigl(\eta_{N}^{h},\eta_{N}^{k}\bigr)\bigr]+\E\Bigl[D\varphi\bigl(X_{N}\bigr).\bigl(\Delta t\sum_{\ell=0}^{n-1}\Pi_{n-1:\ell+1}S_{\Delta t}G''(X_\ell).(\eta_\ell^h,\eta_\ell^k)\bigr)\Bigr]\\
&+\sum_{\ell=0}^{N-1}\sum_i\E\int_{\ell \Delta t}^{(\ell+1)\Delta t}D^2\varphi(X_N).\Bigl(\D_s^{e_i}X_N,\Pi_{N-1:\ell+1}e^{\tau A}S_{\Delta t}\bigl(\sigma''(X_{\ell}).(\eta_{\ell}^h,\eta_{\ell}^k)\bigr)e_i\Bigr)ds\\
&+\sum_{\ell=0}^{N-1}\sum_i\E\int_{\ell \Delta t}^{(\ell+1)\Delta t} \langle D\varphi(X_N),\D_s^{e_i}\Bigl(\Pi_{N-1:\ell+1}e^{\tau A}S_{\Delta t}\bigl(\sigma''(X_{\ell}).(\eta_{\ell}^h,\eta_{\ell}^k)\bigr)e_i\Bigr)\rangle ds.
\end{align*}

\end{propo}
}

\subsection{Estimate of the derivatives}\label{sec:control}

{
\subsubsection{{Auxiliary results}}

To {control  the terms appearing in the expressions of $Du^{\delta,\tau,\Delta t}(T,x).h$ and $D^2u^{\delta,\tau,\Delta t}(T,x).(h,k)$ in Proposition~\ref{propo:expressions}}, we see that estimates on the random operators $\Pi_{N-1:\ell+1}$, and on the Malliavin derivatives $\D_sX_N$ and $\D_s\Pi_{N-1:\ell+1}$, for $s\in\bigl(\ell\Delta t,(\ell+1)\Delta t\bigr)$ are needed.

\begin{lemma}\label{lem:Pi}
For any $q\in[2,\infty)$, $M\in\N^\star$, $T\in(0,\infty)$, and $\beta\in[0,\frac12)$, $\gamma\in [0,\frac14)$ if $q=2$ and $\gamma\in [0,\frac14-\frac1{2q})$ for $q\ne 2$, there exists $C_{\beta,\gamma}(M,q,T)$, such that for any $0\le \ell<n\le N$, and any $\sigma\bigl(\Delta W_0,\ldots,\Delta W_{\ell-1}\bigr)$-measurable random vector $z_\ell$, then
\begin{equation}\label{eq:lem_Pi}
\left(\E|(-A)^\gamma \Pi_{n-1:\ell}z_{\ell}|_{L^q}^{2M}\right)^{\frac1{2M}}\leq C_{\beta,\gamma}(M,p,T)t_{n-\ell}^{-\beta-\gamma}\left(\E|(-A)^{-\beta}z_{\ell}|_{L^q}^{2M}\right)^{\frac1{2M}}.
\end{equation}
\end{lemma}

\begin{lemma}\label{lem:Mall_1}
Let $\ell\in\left\{0,\ldots,n-1\right\}$, and $s\in\bigl(\ell\Delta t,(\ell+1)\Delta t\bigr)$. Then
\begin{equation}\label{eq:lem_Mall_1}
\D_sX_n=\Pi_{n-1:\ell+1}S_{\Delta t}e^{\tau A}\sigma(X_\ell).
\end{equation}
\end{lemma}

\begin{lemma}\label{lem:Mall_2}
For any $q\in(2,\infty)$, $\kappa\in(0,\frac12)$, and $T\in(0,\infty)$, there exists $C_\kappa(q,T)\in(0,\infty)$, such that for any $0\le \ell<N-1$, any $s\in\bigl(\ell\Delta t,(\ell+1)\Delta t\bigr)$, any $z\in L^{2q}$ and any $\sigma\bigl(\Delta W_0,\ldots,\Delta W_{\ell-1}\bigr)$-measurable random vector $\theta_\ell$, then
\begin{equation}\label{eq:lem:Mall_2}
\E|\D_s^{\theta_\ell}\Pi_{n-1:\ell+1}z|_{L^q}^{2}\le C_\kappa(q,T)\bigl(\E|\theta_{\ell}|_{L^{2q}}^{4}\bigr)^{\frac12}\Bigl(\Delta t\bigl(1+\frac{1}{t_{n-\ell-1}^{\frac{1}{2}+\frac{1}{q}+\kappa}}\bigr)|z|_{L^{2q}}^{2}+\mathds{1}_{n>\ell+2}\bigl(1+ \frac{1}{t_{n-\ell-2}^{\frac{1}{2}+\frac{1}{q}+\kappa}}\bigr)\big|(-A)^{-\frac12+\kappa}z\big|_{L^{2q}}^{2}\Bigr).
\end{equation}
Moreover, when $\ell=N-1$, $\D_s\Pi_{N-1:\ell+1}z=0$.
\end{lemma}

In~\eqref{eq:lem:Mall_2}, the quantity $\D_s^{\theta_\ell} \Pi_{N-1:\ell+1}z$ is interpreted as the image of $\theta_\ell$ by the linear operator $\D_s \Pi_{N-1:\ell+1}z$. The assumption that the random vector $\theta_\ell$ is $\sigma\bigl(\Delta W_0,\ldots,\Delta W_{\ell-1}\bigr)$-measurable is crucial.

The proofs of Lemmas~\ref{lem:Pi},~\ref{lem:Mall_1}, and~\ref{lem:Mall_2} are very technical and are postponed to Section~\ref{sec:proof_aux_lemmas}.
}

\subsubsection{Estimate of $\mathcal{D}_{N}^{h,1}$ and of $\mathcal{E}_{N}^{h,k,1}$}

Using the Cauchy-Schwarz inequality, Lemma~\ref{lem:moments_X}, and Assumption~\ref{ass:phi} on $\varphi$, we have
\[
\big|\mathcal{D}_N^{h,1}\big|\le C\bigl(1+|x|_{L^p}\bigr)^{K}\bigl(\E|\tilde{\eta}_N^{h,1}|_{L^q}^2\bigr)^{\frac12} \quad,\quad \big|\mathcal{E}_{N}^{h,k,1}\big|\le C\bigl(1+|x|_{L^p}\bigr)^{K}\bigl(\E|\tilde{\zeta}_N^{h,k,1}|_{L^q}^2\bigr)^{\frac12},
\]
and below we control the moments of $\tilde{\eta}_n^{h,1}$ and $\zeta_n^{h,k,1}$, for every $n\le N$.

We treat $\mathcal{D}_{N}^{h,1}$ first. Thanks to~\eqref{eq:eta_zeta}, applying Lemma~\ref{lem:Pi} gives, for  $\kappa\in(0,\frac12)$,
\begin{align*}
\bigl(\E|\tilde{\eta}_n^{h,1}|_{L^q}^2\bigr)^{\frac12}&\le C\Delta t\sum_{\ell=0}^{n-1}\bigl(\E\big|\Pi_{n-1:\ell+1}S_{\Delta t}G'(X_\ell).S_{\Delta t}^\ell h\big|_{L^q}^{2}\bigr)^{\frac12}\\
&\le C_\kappa\Delta t\sum_{\ell=0}^{n-1}t_{n-\ell}^{-\frac12+\kappa}\bigl(\E\big|(-A)^{-\frac12+\kappa}S_{\Delta t}G'(X_\ell).S_{\Delta t}^\ell h\big|_{L^q}^{2}\bigr)^{\frac12}\\
&\le C_\kappa\Delta t\sum_{\ell=0}^{n-1}t_{n-\ell}^{-\frac12+\kappa}\bigl(\E\big|F_1'(X_\ell).S_{\Delta t}^\ell h\big|_{L^q}^{2}\bigr)^{\frac12}\\
&~+C_\kappa\Delta t\sum_{\ell=0}^{n-1}t_{n-\ell}^{-\frac12+\kappa}\bigl(\E\big|(-A)^{-\frac12+\kappa}BF_2'(X_\ell).S_{\Delta t}^\ell h\big|_{L^q}^{2}\bigr)^{\frac12}.
\end{align*}
By Property~\ref{ass:F}, we get
\[
\bigl(\E\big|F_1'(X_\ell).S_{\Delta t}^\ell h\big|_{L^q}^{2}\bigr)^{\frac12}\le C|S_{\Delta t}^{\ell}h|_{L^q}\le C\mathds{1}_{\ell\neq 0}t_{\ell}^{-\beta}|(-A)^{-\beta}h|_{L^q}+\mathds{1}_{\ell=0}|h|_{L^q}.
\]

Using succesively~\eqref{eq:norm_AB},~\eqref{eq:product_2},~\eqref{eq:Sobolev-domain} and~\eqref{eq:Sobolev-Lipschitz}, and recalling that $F_2'(x).h=F_2'(x)h$ is a product,
\begin{align*}
\bigl(\E\big|(-A)^{-\frac12+\kappa}BF_2'(X_\ell).S_{\Delta t}^\ell h\big|_{L^q}^{2}\bigr)^{\frac12}&\le C\bigl(\E\big|(-A)^{2\kappa}F_2'(X_\ell).S_{\Delta t}^\ell h\big|_{L^q}^{2}\bigr)^{\frac12}\\
&\le C \E\bigl(\big|(-A)^{3\kappa}F_2'(X_\ell)\big|_{L^{2q}}^2\big|(-A)^{3\kappa}S_{\Delta t}^\ell h\big|_{L^{2q}}^{2}\bigr)^{\frac12}\\
&\le C \bigl(|1+\E|(-A)^{4\kappa}X_\ell|_{L^{2q}}^2\bigr)^{\frac12}|(-A)^{3\kappa}S_{\Delta t}^{\ell}h|_{L^{2q}}.
\end{align*}

Let $\kappa>0$ be such that $\beta+7\kappa<1$. Then, thanks to Lemmas~\ref{lem:moments_X} and~\ref{lem:S_1},
\begin{align*}
\bigl(\E|\tilde{\eta}_n^{h,1}|_{L^q}^2\bigr)^{\frac12}&\le Ct_n^{-\frac12+\kappa}\Delta t |h|_{L^q}+Ct_n^{\frac12+\kappa-\beta}|(-A)^{-\beta} h|_{L^q}\\
&~+Ct_n^{-\frac12+\kappa}\Delta t(1+|(-A)^{4\kappa}x|_{L^{2q}})|(-A)^{3\kappa}h|_{L^{2q}} +Ct_n^{\frac12-6\kappa-\beta}(1+|x|_{L^{2q}})|(-A)^{-\beta} h|_{L^{2q}}\\
&\le C\Delta t^{\frac12+\kappa}(1+|(-A)^{4\kappa}x|_{L^{2q}})|(-A)^{3\kappa}h|_{L^{2q}} +Ct_n^{\frac12-6\kappa-\beta}(1+|x|_{L^{2q}})|(-A)^{-\beta} h|_{L^{2q}}.
\end{align*}
and we conclude that
\begin{align*}
\big|\mathcal{D}_N^{h,1}\big|&\le C\Delta t^{\frac12+\kappa}\bigl(1+|x|_{L^p}\bigr)^{K}(1+|(-A)^{4\kappa}x|_{L^{2q}})|(-A)^{3\kappa}h|_{L^{2q}} \\
&~+Ct_N^{\frac12-6\kappa-\beta}\bigl(1+|x|_{L^p}\bigr)^{K}(1+|x|_{L^{2q}})|(-A)^{-\beta}h|_{L^{2q}}.
\end{align*}

We now treat $\mathcal{E}_{N}^{h,k,1}$ with similar arguments. Thanks to~\eqref{eq:eta_zeta}, applying Lemma~\ref{lem:Pi} gives, for $\kappa\in(0,\frac12)$,
\begin{align*}
\bigl(\E|\tilde{\zeta}_{n}^{h,k,1}|_{L^q}^2\bigr)^{\frac12}&\le C\Delta t\sum_{\ell=0}^{n-1}\bigl(\E\big|\Pi_{n-1:\ell+1}S_{\Delta t}G''(X_\ell).(\eta_\ell^h,\eta_\ell^k)\big|_{L^q}^{2}\bigr)^{\frac12}\\
&\le C_\kappa\Delta t\sum_{\ell=0}^{n-1}t_{n-\ell}^{-\frac12+\kappa}\bigl(\E\big|F_1''(X_\ell).(\eta_\ell^h,\eta_\ell^k)\big|_{L^q}^{2}\bigr)^{\frac12}\\
&~+C_\kappa\Delta t\sum_{\ell=0}^{n-1}t_{n-\ell}^{-\frac12+\kappa}\bigl(\E\big|(-A)^{-\frac12+\kappa}BF_2''(X_\ell).(\eta_\ell^h,\eta_\ell^k)\big|_{L^q}^{2}\bigr)^{\frac12}.
\end{align*}

Recall from~\eqref{eq:eta_Pi} that $\eta_\ell^h=\Pi_{\ell-1:0}h$. Property~\ref{ass:F} then gives
\[
\bigl(\E\big|F_1''(X_\ell).(\eta_\ell^h,\eta_\ell^k)\big|_{L^q}^{2}\bigr)^{\frac12}\le C\mathds{1}_{\ell\neq 0}t_{\ell}^{-\beta-\gamma}|(-A)^{-\beta}h|_{L^{2q}}|(-A)^{-\gamma}k|_{L^{2q}}+C\mathds{1}_{\ell=0}|h|_{L^{2q}}|k|_{L^{2q}}.
\]
The remaining term is treated similarly to the one in $\mathcal{D}_N^{h,1}$. Using successively~\eqref{eq:norm_AB},~\eqref{eq:product_2},~\eqref{eq:Sobolev-domain}, and~\eqref{eq:Sobolev-Lipschitz}:
\begin{align*}
\bigl(\E\big|(-A)^{-\frac12+\kappa}BF_2''(X_\ell).&(\eta_\ell^h,\eta_\ell^k)\big|_{L^q}^{2}\bigr)^{\frac12}\le C\bigl(\E\big|(-A)^{2\kappa}F_2''(X_\ell).(\eta_\ell^h,\eta_\ell^k)\big|_{L^q}^{2}\bigr)^{\frac12}\\
&\le C \E\bigl(\big|(-A)^{3\kappa}F_2''(X_\ell)\big|_{L^{2q}}^2\big|(-A)^{4\kappa}\eta_\ell^h\big|_{L^{4q}}^{2}\big|(-A)^{4\kappa}\eta_\ell^k\big|_{L^{4q}}^{2}\bigr)^{\frac12}\\
&\le C \bigl(|1+\E|(-A)^{4\kappa}X_\ell|_{L^{2q}}^6\bigr)^{\frac16}\left(\E\big|(-A)^{4\kappa}\eta_\ell^h\big|_{L^{4q}}^{6}\right)^{\frac16}\left(\E\big|(-A)^{4\kappa}\eta_\ell^k\big|_{L^{4q}}^{6}\right)^{\frac16}\\
&\le C\mathds{1}_{\ell\neq 0}t_\ell^{-12\kappa-\beta-\gamma}(1+|x|_{L^{2q}})|(-A)^{-\beta} h|_{L^{4q}}|(-A)^{-\gamma} k|_{L^{4q}}\\
&~+C\mathds{1}_{\ell=0}\bigl(|1+|(-A)^{4\kappa}x|_{L^{2q}}\bigr)\big|(-A)^{4\kappa}h\big|_{L^{4q}}\big|(-A)^{4\kappa}k\big|_{L^{4q}},
\end{align*}
thanks to Lemma~\ref{lem:Pi}, for $\kappa>0$ chosen sufficiently small to have $4\kappa<\frac{1}{4}-\frac{1}{8q}$.

We thus obtain, if $\beta+\gamma+12\kappa<1$,
\begin{align*}
\bigl(\E|\tilde{\zeta}_{N}^{h,k,1}|_{L^q}^{2}\bigr)^{\frac12}&\le C\Delta t^{\frac12+\kappa}|h|_{L^{2q}}|k|_{L^{2q}}+C\Delta t^{\frac12+\kappa}\bigl(1+|(-A)^{4\kappa}x|_{L^{2q}}\bigr)|(-A)^{4\kappa}h|_{L^{4q}}|(-A)^{4\kappa}k|_{L^{4q}}\\
&~+Ct_{N}^{\frac12-11\kappa-\beta-\gamma}(1+|x|_{L^{2q}})|(-A)^{-\beta}h|_{L^{4q}}|(-A)^{-\gamma}k|_{L^{4q}}\\
&\le C\Delta t^{\frac12+\kappa}\bigl(1+|(-A)^{4\kappa}x|_{L^{2q}}\bigr)|(-A)^{4\kappa}h|_{L^{4q}}|(-A)^{4\kappa}k|_{L^{4q}}\\
&~+Ct_{N}^{\frac12-11\kappa-\beta-\gamma}(1+|x|_{L^{2q}})|(-A)^{-\beta}h|_{L^{4q}}|(-A)^{-\gamma}k|_{L^{4q}}.
\end{align*}
and we conclude that
\begin{align*}
\big|\mathcal{E}_{N}^{h,k,1}\big|&\le C\bigl(1+|x|_{L^p}\bigr)^{K}\Delta t^{\frac12+\kappa}\bigl(1+|(-A)^{4\kappa}x|_{L^{2q}}\bigr)|(-A)^{4\kappa}h|_{L^{4q}}|(-A)^{4\kappa}k|_{L^{4q}}\\
&+Ct_{N}^{\frac12-11\kappa-\beta-\gamma}\bigl(1+|x|_{L^p}\bigr)^{K}(1+|x|_{L^{2q}})|(-A)^{-\beta}h|_{L^{4q}}|(-A)^{-\gamma}k|_{L^{4q}}.
\end{align*}

\subsubsection{Treatment of $\mathcal{D}_{N}^{h,2}$ and of $\mathcal{E}_{N}^{h,k,2}$}

We use the following  basic identities:
\begin{equation}\label{eq:transfo_ij}
\begin{gathered}
\bigl(\sigma'(X_\ell).S_{\Delta t}^{\ell}h\bigr)e_i=\sum_{j\in\N^\star}\langle \bigl(\sigma'(X_\ell).S_{\Delta t}^{\ell}h\bigr)e_i,e_j\rangle e_j=\sum_{j\in\N^\star}\langle \bigl(\sigma'(X_\ell).S_{\Delta t}^{\ell}h\bigr)e_j,e_i\rangle e_j\\
\bigl(\sigma''(X_\ell).(\eta_\ell^h,\eta_\ell^k)\bigr)e_i=\sum_{j\in \N^\star}\langle \bigl(\sigma''(X_\ell).(\eta_\ell^h,\eta_\ell^k)\bigr)e_i,e_j\rangle e_j=\sum_{j\in \N^\star}\langle \bigl(\sigma''(X_\ell).(\eta_\ell^h,\eta_\ell^k)\bigr)e_j,e_i\rangle e_j,
\end{gathered}
\end{equation}
thanks to~\eqref{eq:sigma_star}, from Property~\ref{ass:sigma}.

The parameter $\tau>0$ plays an important role in the estimates below, to ensure summability with respect to $j\in\N^\star$. Indeed, since Lemmas~\ref{lem:Pi} and~\ref{lem:Mall_2} are restricted to powers of $-A$ strictly less than $\frac12$,  the computations below for $\tau=0$ would only provide upper bounds in terms of $\sum_{j}\lambda_{j}^{-\frac12+\kappa}=+\infty$.

\subsubsection*{Control of $\mathcal{D}_{N,\ell}^{h,2,1}$}

From~\eqref{eq:decompD},~\eqref{eq:transfo_ij}, and Assumption~\ref{ass:phi}, we have
\begin{align*}
\big|\mathcal{D}_{N,\ell}^{h,2,1}\big|&=\Big|\sum_{j\in\N^\star}\E\int_{\ell\Delta t}^{(\ell+1)\Delta t}D^2\varphi(X_N).\bigl(\mathcal{D}_sX_N\bigl(\sigma'(X_\ell).S_{\Delta t}^{\ell}h\bigr)e_j,\Pi_{N-1:\ell+1}e^{\tau A}S_{\Delta t}e_j\bigr)ds\Big|\\
&\le C(1+|x|_{L^p})^K\sum_{j\in\N^\star}\int_{\ell\Delta t}^{(\ell+1)\Delta t}\bigl(\E\big|\Pi_{N-1:\ell+1}e^{\tau A}S_{\Delta t}e_j|_{L^q}^{2}\E\big|\mathcal{D}_sX_N\bigl(\sigma'(X_\ell).S_{\Delta t}^{\ell}h\bigr)e_j\big|_{L^q}^2\bigr)^{\frac12}ds.
\end{align*}

On the one hand, by Lemma~\ref{lem:Pi}, for $\ell\in\left\{0,\ldots,N-2\right\}$,
\[
\bigl(\E\big|\Pi_{N-1:\ell+1}e^{\tau A}S_{\Delta t}e_j|_{L^q}^{2}\bigr)^{\frac12}\le C t_{N-\ell-1}^{-\frac12+\kappa}\big|(-A)^{-\frac12+\kappa}e^{\tau A}S_{\Delta t}e_j\big|_{L^q} \le C t_{N-\ell-1}^{-\frac12+\kappa}\tau^{-2\kappa}\lambda_{j}^{-\frac12-\kappa}.
\]
When $\ell=N-1$, $\bigl(\E\big|\Pi_{N-1:\ell+1}e^{\tau A}S_{\Delta t}e_j|_{L^q}^{2}\bigr)^{\frac12}=\big|e^{\tau A}S_{\Delta t}e_j|_{L^q}\le C\Delta t^{-\frac12-\kappa}\lambda_{j}^{-\frac12-\kappa}$.

On the other hand, using Lemmas~\ref{lem:Mall_1} and~\ref{lem:Pi}, and then Properties~\ref{ass:A} and~\ref{ass:sigma},
\begin{align*}
\bigl(\E\big|\mathcal{D}_sX_N\bigl(\sigma'(X_\ell).S_{\Delta t}^{\ell}h\bigr)e_j\big|_{L^q}^2\bigr)^{\frac12}&=\bigl(\E\big|\Pi_{N-1:\ell+1}S_{\Delta t}e^{\tau A}\sigma(X_\ell)\bigl(\sigma'(X_\ell).S_{\Delta t}^{\ell}h\bigr)e_j\big|_{L^q}^2\bigr)^{\frac12}\\
&\le C\bigl(\E\big|\sigma(X_\ell)\bigl(\sigma'(X_\ell).S_{\Delta t}^{\ell}h\bigr)e_j\big|_{L^q}^2\bigr)^{\frac12}\\
&\le C\big|S_{\Delta t}^{\ell}h\big|_{L^q}\\
&\le C\mathds{1}_{\ell\neq 0}t_{\ell}^{-\beta}|(-A)^{-\beta}h|_{L^q}+C\mathds{1}_{\ell=0}|h|_{L^q}.
\end{align*}

Recall that $\sum_{j\in\N^\star}\lambda_{j}^{-\frac12-\kappa}<\infty$ by Property~\ref{ass:A}. This yields
\[
\sum_{\ell=0}^{N-1}\big|\mathcal{D}_{N,\ell}^{h,2,1}\big|\le \frac{C(1+|x|_{L^p})^K}{\tau^{2\kappa}}\bigl(t_{N-1}^{-\frac12+\kappa}\Delta t |h|_{L^q}+t_{N-1}^{\frac12+\kappa-\beta}\big|(-A)^{-\beta}h\big|_{L^q}\bigr).
\]

\subsubsection*{Control of $\mathcal{D}_{N,\ell}^{h,2,2}$}

Thanks to~\eqref{eq:decompD}, \eqref{eq:transfo_ij}, and Assumption~\ref{ass:phi}, we get
\begin{align*}
\big|\mathcal{D}_{N,\ell}^{h,2,2}\big|&=\Big|\sum_{j\in\N^\star}\E\int_{\ell\Delta t}^{(\ell+1)\Delta t}\langle D\varphi(X_N),\D_s^{(\sigma'(X_\ell).S_{\Delta t}^{\ell}h)e_j}\Pi_{N-1:\ell+1}e^{\tau A}S_{\Delta t}e_j\rangle ds\Big|\\
&\le C(1+|x|_{L^p})^{K}\sum_{j\in\N^\star}\int_{\ell\Delta t}^{(\ell+1)\Delta t}\bigl(\E\big|\D_s^{(\sigma'(X_\ell).S_{\Delta t}^{\ell}h)e_j}\Pi_{N-1:\ell+1}e^{\tau A}S_{\Delta t}e_j\big|_{L^q}^{2}\bigr)^{\frac12}ds.
\end{align*}

In addition, observe that $\mathcal{D}_{N,N-1}^{h,2,2}=0$, thanks to the second part of Lemma~\ref{lem:Mall_2}. Applying the estimate in Lemma~\ref{lem:Mall_2}, for $\ell\in\left\{0,\ldots,N-2\right\}$, one has
\begin{align*}
\bigl(\E\big|&\D_s^{(\sigma'(X_\ell).S_{\Delta t}^{\ell}h)e_j}\Pi_{N-1:\ell+1}e^{\tau A}S_{\Delta t}e_j|_{L^q}^{2}\bigr)^{\frac12}\\
&\le C\bigl(\E\big|(\sigma'(X_\ell).S_{\Delta t}^{\ell}h)e_j\big|_{L^{2q}}\big|^4\bigr)^{\frac14}
\Delta t^{\frac12}|S_{\Delta t}e^{\tau A}e_j|_{L^{2q}}\bigl(1+\frac{1}{t_{N-\ell-1}^{\frac{1}{4}+\frac{1}{2q}+\kappa}}\bigr)\\
&~+C\bigl(\E\big|(\sigma'(X_\ell).S_{\Delta t}^{\ell}h)e_j\big|_{L^{2q}}\big|^4\bigr)^{\frac14}\mathds{1}_{\ell<N-2}|(-A)^{-\frac12+\kappa}S_{\Delta t}e^{\tau A}e_j|_{L^{2q}}\bigl(1+\frac{1}{t_{N-\ell-2}^{\frac{1}{4}+\frac{1}{2q}+\kappa}}\bigr)\\
&\le C\Bigl(\mathds{1}_{\ell\neq 0}t_{\ell}^{-\beta}|(-A)^{-\beta}h|_{L^{2q}}+\mathds{1}_{\ell=0}|h|_{L^{2q}}\Bigr)\tau^{-2\kappa}\lambda_{j}^{-\frac12-\kappa}\bigl(1+\frac{\Delta t^\kappa}{t_{N-\ell-1}^{\frac{1}{4}+\frac{1}{2q}+\kappa}}+
\frac{\mathds{1}_{\ell<N-2}}{t_{N-\ell-2}^{\frac{1}{4}+\frac{1}{2q}+\kappa}}\bigr).
\end{align*}

This yields
\[
\sum_{\ell=0}^{N-1}\big|\mathcal{D}_{N,\ell}^{h,2,2}\big|\le \frac{C(1+|x|_{L^p})^{K}}{\tau^{2\kappa}}\bigl(1+t_{N-1}^{\frac{1}{2}-\frac{1}{2q}-\kappa-\beta}\bigr)\Bigl(|(-A)^{-\beta}h|_{L^{2q}}+\frac{\Delta t}{t_{N-1}}|h|_{L^{2q}}\Bigr).
\]

\subsubsection*{Control of $\mathcal{E}_{N,\ell}^{h,k,2,1}$}

Thanks to \eqref{eq:decompE} and Assumption~\ref{ass:phi} we get
\begin{align*}
\big|\mathcal{E}_{N,\ell}^{h,k,2,1}\big|&=\Big|\sum_{j\in\N^\star}\E\int_{\ell\Delta t}^{(\ell+1)\Delta t}D^2\varphi(X_N).\bigl(\mathcal{D}_sX_N\bigl(\sigma''(X_{\ell}).(\eta_{\ell}^h,\eta_{\ell}^k)\bigr)e_j,\Pi_{N-1:\ell+1}e^{\tau A}S_{\Delta t}e_j\bigr)ds\Big|\\
&\le C(1+|x|_{L^p})^K\sum_{j\in\N^\star}\int_{\ell\Delta t}^{(\ell+1)\Delta t}\bigl(\E\big|\Pi_{N-1:\ell+1}e^{\tau A}S_{\Delta t}e_j|_{L^q}^{2}\E\big|\mathcal{D}_sX_N\bigl(\sigma''(X_{\ell}).(\eta_{\ell}^h,\eta_{\ell}^k)\bigr)e_j\big|_{L^q}^2\bigr)^{\frac12}ds.
\end{align*}

On the one hand, by Lemma~\ref{lem:Pi}, for $\ell\in\left\{0,\ldots,N-2\right\}$,
\[
\bigl(\E\big|\Pi_{N-1:\ell+1}e^{\tau A}S_{\Delta t}e_j|_{L^q}^{2}\bigr)^{\frac12}\le C t_{N-\ell-1}^{-\frac12+\kappa}\tau^{-2\kappa}\lambda_{j}^{-\frac12-\kappa}.
\]
When $\ell=N-1$, $\bigl(\E\big|\Pi_{N-1:\ell+1}e^{\tau A}S_{\Delta t}e_j|_{L^q}^{2}\bigr)^{\frac12}=\big|e^{\tau A}S_{\Delta t}e_j|_{L^q}\le C\Delta t^{-\frac12-\kappa}\lambda_{j}^{-\frac12-\kappa}$.

On the other hand, using Lemmas~\ref{lem:Mall_1} and~\ref{lem:Pi}, and Property~\ref{ass:sigma},
\begin{align*}
\bigl(\E\big|\mathcal{D}_sX_N\bigl(\sigma''(X_{\ell}).(\eta_{\ell}^h,\eta_{\ell}^k)\bigr)e_j\big|_{L^q}^2\bigr)^{\frac12}&=\bigl(\E\big|\Pi_{N-1:\ell+1}S_{\Delta t}e^{\tau A}\sigma(X_\ell)\bigl(\sigma''(X_{\ell}).(\eta_{\ell}^h,\eta_{\ell}^k)\bigr)e_j\big|_{L^q}^2\bigr)^{\frac12}\\
&\le C\bigl(\E\big|\sigma(X_\ell)\bigl(\sigma''(X_{\ell}).(\eta_{\ell}^h,\eta_{\ell}^k)\bigr)e_j\big|_{L^q}^2\bigr)^{\frac12}\\
&\le C\bigl(\E|\eta_{\ell}^{h}|_{L^{2q}}^{4}\bigr)^{\frac14}\bigl(\E|\eta_{\ell}^{k}|_{L^{2q}}^{4}\bigr)^{\frac14}\\
&\le C\mathds{1}_{\ell\neq 0}t_{\ell}^{-\beta-\gamma}|(-A)^{-\beta}h|_{L^{2q}}|(-A)^{-\gamma}k|_{L^{2q}}+C\mathds{1}_{\ell=0}|h|_{L^{2q}}|k|_{L^{2q}}.
\end{align*}

This yields
\[
\sum_{\ell=0}^{N-1}\big|\mathcal{E}_{N,\ell}^{h,k,2,1}\big|\le \frac{C(1+|x|_{L^p})^K}{\tau^{2\kappa}}\bigl(t_{N-1}^{-\frac12+\kappa}\Delta t |h|_{L^{2q}}|k|_{L^{2q}}+t_{N-1}^{\frac12+\kappa-\beta-\gamma}|(-A)^{-\beta}h|_{L^{2q}}|(-A)^{-\gamma}k|_{L^{2q}}\bigr).
\]

\subsubsection*{Control of $\mathcal{E}_{N,\ell}^{h,k,2,2}$}

Thanks to~\eqref{eq:decompE}, \eqref{eq:transfo_ij} and Assumption~\ref{ass:phi} we get
\begin{align*}
\big|\mathcal{E}_{N,\ell}^{h,k,2,2}\big|&=\Big|\sum_{j\in\N^\star}\E\int_{\ell\Delta t}^{(\ell+1)\Delta t}\langle D\varphi(X_N),\D_s^{(\sigma''(X_{\ell}).(\eta_{\ell}^h,\eta_{\ell}^k))e_j}\Pi_{N-1:\ell+1}e^{\tau A}S_{\Delta t}e_j\rangle ds\Big|\\
&\le C(1+|x|_{L^p})^{K}\sum_{j\in\N^\star}\int_{\ell\Delta t}^{(\ell+1)\Delta t}\bigl(\E\big|\D_s^{(\sigma''(X_{\ell}).(\eta_{\ell}^h,\eta_{\ell}^k))e_j}\Pi_{N-1:\ell+1}e^{\tau A}S_{\Delta t}e_j\big|_{L^q}^{2}\bigr)^{\frac12}ds.
\end{align*}

In addition, observe that $\mathcal{E}_{N,N-1}^{h,2,2}=0$, thanks to the second part of Lemma~\ref{lem:Mall_2}. Applying the estimate in Lemma~\ref{lem:Mall_2}, for $\ell\in\left\{1,\ldots,N-2\right\}$, one has
\begin{align*}
\bigl(\E\big|&\D_s^{(\sigma''(X_{\ell}).(\eta_{\ell}^h,\eta_{\ell}^k))e_j}\Pi_{N-1:\ell+1}e^{\tau A}S_{\Delta t}e_j|_{L^q}^{2}\bigr)^{\frac12}\\
&\le C\bigl(\E\big|\bigl(\sigma''(X_{\ell}).(\eta_{\ell}^h,\eta_{\ell}^k)\bigr)e_j\big|_{L^{2q}}\big|^4\bigr)^{\frac14}
\Delta t^{\frac12}|S_{\Delta t}e^{\tau A}e_j|_{L^{2q}}\bigl(1+\frac{1}{t_{N-\ell-1}^{\frac{1}{4}+\frac{1}{2q}+\kappa}}\bigr)\\
&~+C\bigl(\E\big|\bigl(\sigma''(X_{\ell}).(\eta_{\ell}^h,\eta_{\ell}^k)\bigr)e_j\big|_{L^{2q}}\big|^4\bigr)^{\frac14}\mathds{1}_{\ell<N-2}|(-A)^{-\frac12+\kappa}S_{\Delta t}e^{\tau A}e_j|_{L^{2q}}\bigl(1+\frac{1}{t_{N-\ell-2}^{\frac{1}{4}+\frac{1}{2q}+\kappa}}\bigr)\\
&\le C\bigl(\E|\eta_\ell^h|_{L^{2q}}^{8}\bigr)^{\frac18}\bigl(\E|\eta_\ell^k|_{L^{2q}}^{8}\bigr)^{\frac18} \tau^{-2\kappa}\lambda_{j}^{-\frac12-\kappa}\bigl(1+\frac{\Delta t^\kappa}{t_{N-\ell-1}^{\frac{1}{4}+\frac{1}{2q}+\kappa}}+
\frac{\mathds{1}_{\ell<N-2}}{t_{N-\ell-2}^{\frac{1}{4}+\frac{1}{2q}+\kappa}}\bigr)\\
&\le C\Bigl(\mathds{1}_{\ell\neq 0}t_{\ell}^{-\beta-\gamma} |(-A)^{-\beta}h|_{L^{2q}}|(-A)^{-\gamma}k|_{L^{2q}}+\mathds{1}_{\ell=0}|h|_{L^{2q}}|k|_{L^{2q}}\Bigr) \tau^{-2\kappa}\lambda_{j}^{-\frac12-\kappa}\bigl(1+\frac{\Delta t^\kappa}{t_{N-\ell-1}^{\frac{1}{4}+\frac{1}{2q}+\kappa}}+
\frac{\mathds{1}_{\ell<N-2}}{t_{N-\ell-2}^{\frac{1}{4}+\frac{1}{2q}+\kappa}}\bigr)
\end{align*}
This yields
\[
\sum_{\ell=0}^{N-1}\big|\mathcal{E}_{N,\ell}^{h,k,2,2}\big|\le \frac{C(1+|x|_{L^p})^{K}}{\tau^{2\kappa}}\bigl(1+t_{N-1}^{\frac{1}{2}-\frac{1}{2q}-\kappa-\beta-\gamma}\bigr)\Bigl(|(-A)^{-\beta}h|_{L^{2q}}|(-A)^{-\gamma}k|_{L^{2q}}+\frac{\Delta t}{t_{N}}|h|_{L^{2q}}|k|_{L^{2q}}\Bigr).
\]

\subsubsection{Estimate with $\tau>0$}

Gathering all above estimates, we have - recall that $t_N=T$ - for $\beta\in [0,1)$:
\begin{align*}
\Big|Du^{\delta,\tau,\Delta t}(T,x).h\Big|&=\Big|\mathcal{D}_N^{h,0}+\mathcal{D}_N^{h,1}+\mathcal{D}_N^{h,2}\Big|\\
&\le CT^{-\beta}\bigl(1+|x|_{L^p}\bigr)^{K}|(-A)^{-\beta}h|_{L^{q}}\\
&+ C\Delta t^{\frac12+\kappa}\bigl(1+|x|_{L^p}\bigr)^{K}(1+|(-A)^{4\kappa}x|_{L^{2q}})|(-A)^{3\kappa}h|_{L^{2q}} \\
&+CT^{\frac12-6\kappa-\beta}\bigl(1+|x|_{L^p}\bigr)^{K}(1+|x|_{L^{2q}})|(-A)^{-\beta}h|_{L^{2q}}\\
&+\frac{C(1+|x|_{L^p})^K}{\tau^{2\kappa}}\bigl(T^{-\frac12+\kappa}\Delta t |h|_{L^q}+t_{N-1}^{\frac12+\kappa-\beta}\big|(-A)^{-\beta}h\big|_{L^q}\bigr)\\
&+\frac{C(1+|x|_{L^p})^{K}}{\tau^{2\kappa}}\bigl(1+T^{\frac{1}{2}-\frac{1}{2q}-\kappa-\beta}\bigr)\bigl(|(-A)^{-\beta}h|_{L^{2q}}+\frac{\Delta t}{T}|h|_{L^{2q}}\bigr).
\end{align*}

{Letting $\Delta t\to 0$ yields Proposition~\ref{propo:u_delta_tau_1}. Similarly, gathering all above estimates and using the identity $D^2u^{\delta,\tau,\Delta t}(T,x).(h,k)
=\mathcal{E}_N^{h,k,0}+\mathcal{E}_N^{h,k,1}
+\mathcal{E}_N^{h,k,2}$, gives a similar estimate for $D^2u^{\delta,\tau,\Delta t}(T,x).(h,k)$. Letting $\Delta t\to 0$ then yields Proposition~\ref{propo:u_delta_tau_2}.
}

\subsection{Conclusion of the proof}\label{sec:proof_tau}

{To deduce Theorems~\ref{theo:D1} and~\ref{theo:D2} from Propositions~\ref{propo:u_delta_tau_1} and~\ref{propo:u_delta_tau_2}, it remains to explain how to get rid of the singular factor $\tau^{-\kappa}$ in~\eqref{eq:propo_u_delta_tau_1} and~\eqref{eq:propo_u_delta_tau_2}. This is done thanks to an interpolation argument.} We need the following result, which is not optimal -- we expect an order $\frac14$ in~\eqref{eq:propo_aux_3} as in~\eqref{eq:propo_aux_1} --
but sufficient for our purpose.
\begin{propo}\label{propo:aux}
For every  $\kappa\in(0,1)$, $T>0$, there exists $C_{\kappa,\epsilon}(T)\in(0,\infty)$, such that for every $\delta,\tau\in(0,1)$, $x\in L^p$ and $h,k\in L^{2q}$
\begin{align}
&|u^{\delta,\tau}(T,x)-u_\delta(T,x)|\le  C_\kappa(T)\tau^{\frac14-\kappa}\bigl(1+|x|_{L^{p}}^{K}\bigr)\label{eq:propo_aux_1}\\
&\big|\bigl(Du^{\delta,\tau}(T,x)-Du_\delta(T,x)\bigr).h\big|\le C_{\kappa}(T)\tau^{\frac14-\kappa}\bigl(1+|x|_{L^{p}}^{K}\bigr)|h_1|_{L^{q}}\label{eq:propo_aux_2}\\
&\big|\bigl(D^2u^{\delta,\tau}(T,x)-D^2u_\delta(T,x)\bigr).\bigl(h,k\bigr)\big|\le C_{\kappa}(T)\tau^{\frac1{8}-\kappa}\bigl(1+|x|_{L^{p}}^{K}\bigr)|h|_{L^{3q}}|k|_{L^{3q}}.\label{eq:propo_aux_3}
\end{align}
\end{propo}

{The proof of Proposition~\ref{propo:aux} is postponed to the end of the section.}

{We are now in position to conclude the proof of Theorem~\ref{theo:D1}, as a consequence of Propositions~\ref{propo:u_delta_tau_1} and~\ref{propo:aux}.} Identifying the first order derivative with the gradient, and letting $\frac1r+\frac1{2q}=1$, we may rewrite~\eqref{eq:propo_u_delta_tau_1} and~\eqref{eq:propo_aux_2} as
\begin{align*}
\big|(-A)^{\beta}Du^{\delta,\tau}(T,x)\big|_{L^r}\le \frac{C_{\beta,\kappa}(T)}{\tau^{\kappa} T^{\beta}}\bigl(1+|x|_{L^{p}}\bigr)^{K}(1+|x|_{L^{2q}}),\\
\big|Du^{\delta,\tau}(T,x)-Du_\delta(T,x)\big|_{L^r}\le C_\kappa(T)\tau^{\frac14-\kappa}\bigl(1+|x|_{L^{p}}\bigr)^{K}.
\end{align*}
for $\beta \in [0,1)$. Take $\tau_k=2^{-k}$, $0<\beta<\tilde\beta<1$, $\lambda=\frac\beta{\tilde\beta}$ and 
$\kappa<\frac14(1-\lambda)$. Then we may write:
\begin{align*}
\big|(-A)^{\beta}Du_{\delta}(T,x)\big|_{L^r}&\le\sum_{k\in\N}\big|(-A)^\beta\left(Du^{\delta,\tau_{k+1}}(T,x)-Du^{\delta,\tau_{k}}(T,x)\right)\big|_{L^r}\\
&\le\sum_{k\in\N}\big|(-A)^{\tilde\beta}\left(Du^{\delta,\tau_{k+1}}(T,x)-Du^{\delta,\tau_{k}}(T,x)\right)\big|_{L^r}^{\lambda}\big|Du^{\delta,\tau_{k+1}}(T,x)-Du^{\delta,\tau_{k}}(T,x)\big|_{L^r}^{1-\lambda}\\
&\le \frac{C_{\beta,\kappa}(T)}{T^{\beta}}\bigl(1+|x|_{L^{p}}\bigr)^{K}(1+|x|_{L^{2q}})\sum_{k\in\N} 2^{k(-\kappa\lambda+(\frac14-\kappa)(1-\lambda))}\\
&\le \frac{C_{\beta,\kappa}(T)}{T^{\beta}}\bigl(1+|x|_{L^{p}}\bigr)^{K}(1+|x|_{L^{2q}}).
\end{align*}
This yields~\eqref{eq:theo_D13_1}, and concludes the proof of Theorem~\ref{theo:D1}.

We proceed similarly for the proof of Theorem \ref{theo:D2}, and thus we will not provide all the details. Identifying the second order derivative with the Hessian, and letting $\frac{1}{r}+\frac{1}{4q}=1$, we may rewrite~\eqref{eq:propo_u_delta_tau_2} and~\eqref{eq:propo_aux_3} as
\begin{align*}
\big|(-A)^\gamma D^2u^{\delta,\tau}(T,x)(-A)^\beta h\big|_{L^r}\le \frac{C_{\beta,\gamma,\kappa}(T)}{\tau^{\kappa} T^{\beta+\gamma}}\bigl(1+|x|_{L^{p}}^{K}\bigr)(1+|x|_{L^{2q}})|h|_{L^{4q}}\\
\big|\left(D^2u^{\delta,\tau}(T,x)-D^2u_\delta(T,x)\right)h\big|_{L^r}\le C_{\kappa,\epsilon}(T)\tau^{\frac1{8}-\kappa}\bigl(1+|x|_{L^{p}}^{K}\bigr)|h|_{L^{4q}}.
\end{align*}
Let us first take $\beta=0$ and take $\gamma<\tilde\gamma<\frac12$, $\lambda=\frac\gamma{\tilde\gamma}$ and $\kappa<\frac18(1-\lambda)$;
then, for $\tau_1\le \tau_2$,
\begin{align*}
&\big|(-A)^\gamma \left(D^2u^{\delta_1,\tau}(T,x)-D^2u^{\delta_2,\tau}(T,x)\right)h\big|_{L^r}\\
&\le \big|(-A)^{\tilde\gamma} \left(D^2u^{\delta_1,\tau}(T,x)-D^2u^{\delta_2,\tau}(T,x)\right)h\big|_{L^r}^\lambda
\big| \left(D^2u^{\delta_1,\tau}(T,x)-D^2u^{\delta_2,\tau}(T,x)\right)h\big|_{L^r}^{1-\lambda}\\
&\le\frac{C_{\gamma,\kappa}(T)}{T^{\gamma}}\tau_2^{\frac18(1-\lambda)-\kappa} \bigl(1+|x|_{L^{p}}\bigr)^{K}(1+|x|_{L^{2q}})|h|_{L^{4q}}.
\end{align*}
Since $D^2u$ is symmetric, it follows replacing $\gamma$ by $\beta\in [0,\frac12)$:
$$
\big| \left(D^2u^{\delta_1,\tau}(T,x)-D^2u^{\delta_2,\tau}(T,x)\right)(-A)^\beta h\big|_{L^r}\le\frac{C_{\beta,\alpha_0}(T)}{T^{\beta}}\tau_2^{\alpha_0} \bigl(1+|x|_{L^{p}}\bigr)^{K}(1+|x|_{L^{2q}})|h|_{L^{4q}},
$$
for $\alpha_0<\frac18(1-\beta)$. We then repeat the argument to conclude the proof of Theorem~\ref{theo:D2}.

{To conclude this section, we give a proof of Proposition~\ref{propo:aux}.}
\begin{proof}[Proof of Proposition~\ref{propo:aux}]
Again, we omit to write the dependence on $\delta$, for instance we write $u^{\tau}$ and $u$ instead of $u^{\delta,\tau}$ and $u_\delta$. Also, we only treat the case $q>2$. 
For every $\tau\in[0,1)$, let $\bigl(X_t^\tau\bigr)_{0\le t\le T}$ denote the solution of
\begin{equation*}
dX_t^\tau=AX_t^\tau dt+G(X_t^\tau)dt+e^{\tau A}\sigma(X_t^\tau)dW(t) , \quad X_0^\tau=x,
\end{equation*}
so that $u^\tau(T,x)=\E\bigl[\varphi(X_t^\tau)\bigr]$, $X^0=X$ and $u^0=u$.

We first prove~\eqref{eq:propo_aux_1}. Due to the regularity conditions on the test functions $\varphi$, see Assumption~\ref{ass:phi}, it is sufficient to prove the following bounds: for every $M\in\N^\star$ and every $p,q\in[2,\infty)$, for every $\gamma\in[0,\frac12)$ and $\kappa>0$ sufficiently small, there exists $C_{M,p,q}(T)\in(0,\infty)$, such that for every $0<t\le T$ and every $x\in L^p$, we have
\begin{equation}\label{eq:upper_tau}
\begin{gathered}
\bigl(\E|X^\tau(t)|_{L^p}^{2M}\bigr)^{\frac{1}{2M}}\le C_{\gamma,\kappa,p,q,M}(T)(1+|x|_{L^p}),\\
\bigl(\E|X^\tau(t)-X^0(t)|_{L^q}^{2M}\bigr)^{\frac{1}{2M}}\le C_{\gamma,\kappa,p,q,M}(T)\tau^{\frac14-\kappa}\\
\end{gathered}
\end{equation}

For simplicity, we treat only the case $M=1$. The first inequality is easy because $F_1$, $F_2$,
and $\sigma$ are bounded.

Since $\bigl(e^{tA}\bigr)_{t\ge 0}$ is an analytic semi-group on $L^p$ for every $p\in[2,\infty)$, it is standard that for $\alpha\in[0,1)$, there exists $C(p,\alpha)\in(0,\infty)$ such that
\begin{equation}
\big|(-A)^{-\alpha}(e^{\tau A}-I)\big|_{\mathcal{L}(L^p,L^p)}\le C(p,\alpha)\tau^\alpha.
\end{equation}

Let us write $e_\tau= X^\tau-X^0$, $e_\tau=e_\tau^1+e_\tau^2$ with 
\begin{align*}
e^1_\tau&=\int_{0}^{t}e^{(t-s)A}\bigl(F_1(X_{s}^{\tau})-F_1(X_{s}^{0})\bigr)ds+( \int_{0}^{t}e^{(t-s+\tau)A}\bigl(\sigma(X_{s}^{\tau})-\sigma(X_{s}^{0})\bigr)dW(s)\\
& +\int_{0}^{t}\bigl(e^{\tau A}-I\bigr)e^{(t-s)A}\sigma(X_{s}^{0})dW(s),
\end{align*}
which yields, thanks to Properties \ref{ass:F} and \ref{ass:sigma},
\begin{align*}
\E|e_\tau^1(t)|_{L^q}^2&\le C\int_0^t \E|X_s^\tau-X_s^0|_{L^q}^2ds\\
&+C\int_{0}^{t}\big|(-A)^{\frac{1}{2q}}e^{(t-s+\tau)A}\big|_{R(L^2,L^q)}^2\E|X_s^\tau-X_s^0|_{L^q}^2ds\\
&+C\int_{0}^{t}\big|(-A)^{-\frac14+\kappa}(e^{\tau A}-I)\big|_{\mathcal{L}(L^q,L^q)}^2 \big|(-A)^{\frac14-\kappa}e^{(t-s+\tau)A}\big|_{R(L^2,L^q)}^2 ds\\
&\le C\int_{0}^{t}\bigl(\frac1{(t-s)^{\frac12+\frac1q+\kappa}}+1\bigr)\E|X_s^\tau-X_s^0|_{L^q}^2ds + C \tau^{\frac12-2\kappa},
\end{align*}
using a continuous time version of Lemma~\ref{lem:S_2}.

The equation for $e^2$ is
$$
\frac{d}{dt}e_\tau^2=Ae_\tau^2+\left(BF_2(X^\tau)- BF_2(X^0)\right), \quad e_\tau^2(0)=0.
$$
We estimate $e_\tau^2$ by an energy method. Recall that we work in fact with regularized coefficients, $G_\delta= B e^{\delta A}F_2(e^\delta \cdot) + e^{\delta A}F_1(e^\delta \cdot)$, so that both $X^\tau$ and $X$ are sufficiently regular to justify all the computations. Multiply the equation by $(e_\tau^2)^{q-1}$, integrate in space to get thanks to standard manipulations as in the proof of Lemma 
\ref{lem:Pi_mild}:
$$
\frac{d}{dt}|e_\tau^2|_{L^q}^q\le c  |X^\tau-X_0|_{L^q}^2|e_\tau^2|_{L^q}^{q-2}
$$
and 
$$
\frac{d}{dt}|e_\tau^2|_{L^q}^2\le c  |X^\tau-X_0|_{L^q}^2.
$$
Integrating in time and adding with the inequality above yields: 
$$
\E|X_t^\tau-X_t^0|_{L^q}^2\le C\int_{0}^{t}\left(\frac1{(t-s)^{\frac12+\frac1q+\kappa}}+1\right)\E|X_s^\tau-X_s^0|_{L^q}^2ds + C \tau^{\frac12-2\kappa}
$$
and~\eqref{eq:propo_aux_1} follows from Gronwall Lemma.

The proof of~\eqref{eq:propo_aux_2} is similar but longer; details are left to the reader. Finally, instead of proving~\eqref{eq:propo_aux_3} with similar long but straightforward arguments (and a better estimate with $\tau^{\frac14-\kappa}$ is obtained), it is simpler to use
Proposition~\ref{theo:D3} for $k_1,k_2,k_3\in L^{3q}$:
$$
\big|D^3u_\delta(T,x).(k_1,k_2,k_3)\big|\le C_\beta(T)(1+|x|_{L^p})^K|k_1|_{L^{3q}}|k_2|_{L^{3q}}|k_3|_{L^{3q}}.
$$
and get the result by an interpolation argument. {This concludes the proof of Proposition~\ref{propo:aux}.}
\end{proof}

\subsection{Proof of the auxiliary lemmas}\label{sec:proof_aux_lemmas}

\begin{proof}[Proof of Lemma~\ref{lem:Mall_1}]
Thanks to~\eqref{eq:scheme}, we obtain
\begin{align*}
X_n&=S_{\Delta t}^{n-\ell}X_{\ell}+\Delta tS_{\Delta t}^{n-\ell}G(X_\ell)+S_{\Delta t}^{n-\ell}e^{\tau A}\sigma(X_\ell)\Delta W_\ell\\\
&~+\Delta t\sum_{m=\ell+1}^{n-1}S_{\Delta t}^{n-m}G(X_m)+\sum_{m=\ell+1}^{n-1}S_{\Delta t}^{n-m}e^{\tau A}\sigma(X_m)\Delta W_m,
\end{align*}
where $X_{\ell}$ is $\sigma\bigl(\Delta W_0,\ldots,\Delta W_{\ell-1}\bigr)$-measurable. Thus $\mathcal{D}_sX_{\ell}=0$ for $s>\ell\Delta t$.

Moreover,  $\D_s^\theta \Delta W_{\ell}=\theta$ and, for $m>\ell$, $\D_s^\theta \Delta W_{m}=0$ for $s\in\bigl(\ell\Delta t,(\ell+1)\Delta t\bigr)$. Using the chain rule, we thus obtain, for $n>\ell$ and any $\theta\in H$,
\[
\D_s^\theta X_n=S_{\Delta t}^{n-\ell}e^{\tau A}\sigma(X_\ell)\theta+\Delta t\sum_{m=\ell+1}^{n-1}S_{\Delta t}^{n-m}G'(X_m).\D_s^{\theta}X_m+\sum_{m=\ell+1}^{n-1}S_{\Delta t}^{n-m}e^{\tau A}\bigl(\sigma'(X_m)\D_s^\theta X_m\bigr)\Delta W_m,
\]
which in turn gives $\D_s^\theta X_n=\Pi_{n-1}\D_s^\theta X_{n-1}$ by definition~\eqref{eq:Pi_n}. Since $\D_s^\theta X_{\ell+1}=S_{\Delta t}e^{\tau A}\sigma(X_\ell)$, equality~\eqref{eq:lem_Mall_1} is satisfied, and this concludes the proof of Lemma~\ref{lem:Mall_1}.
\end{proof}

Lemmas~\ref{lem:Pi} and~\ref{lem:Mall_2} are both consequences of the following technical result.

\begin{lemma}\label{lem:Pi_mild}
Let $q\in[2,\infty)$, $M\in\N^\star$, $T\in(0,\infty)$, and $\beta\in[0,\frac12)$. There exists $C_\beta(M,q,T)$ such that the following holds true.

Let $\ell\in\left\{0,\ldots,N-1\right\}$ and consider a $\sigma\bigl(\Delta W_0,\ldots,\Delta W_{\ell-1}\bigr)$-measurable random vector $z_\ell$, and two sequences $\bigl(Z_{n}^{j}\bigr)_{n\ge \ell, j\in\left\{1,2\right\}}$, such that $Z_n^j$ is $\sigma\bigl(\Delta W_0,\ldots,\Delta W_{n-1}\bigr)$-measurable.

Define the sequence $\bigl(Y_n^\ell\bigr)_{\ell\le n\le N}$ by $Y_{\ell}^{\ell}=z_\ell$, and for $n>\ell$
\[
Y_{n}^{\ell}=\Pi_{n-1}Y_{n-1}^\ell+\Delta t S_{\Delta t}G_{n-1}+S_{\Delta t}e^{\tau A} \bigl(\sigma''(X_{n-1}).(Z_{n-1}^{1},Z_{n-1}^2)\bigr)\Delta W_{n-1},
\]
with $G_{n-1}=G''(X_{n-1}).(Z_{n-1}^1,Z_{n-1}^2)$.

Then, when $q>2$, and every $n\ge \ell+1$,
\begin{align*}
\left(\E|Y_n^\ell|_{L^q}^{2M}\right)^{\frac1{M}}&\leq C_\beta(M,q,T)\bigg(t_{n-\ell}^{-2\beta}(\E|(-A)^{-\beta}z_{\ell}|_{L^q}^{2M})^{\frac1{M}}
+\Delta t \sum_{m=\ell}^{n-1} \bigl(1+\frac{1}{t_{n-m}^{\frac{1}{2}+\frac{1}{q}+\kappa}}\bigr)
\E\left(|Y_{m}^{\ell}|_{L^q}^{2M}\right)^\frac1M\\
&+\Delta t\sum_{m=\ell}^{n-1}\bigl(1+\frac{1}{t_{n-m}^{\frac{1}{2}+\frac{1}{q}+\kappa}}\bigr)(\E|Z_{n}^1|_{L^{2q}}^{4M})^{\frac1{2M}}(\E|Z_{n}^2|_{L^{2q}}^{4M})^{\frac1{2M}}
\bigg).
\end{align*}
When $q=2$, for every $n\ge \ell+1$
\begin{align*}
\left(\E|Y_n^\ell|^{2M}\right)^{\frac1{M}}&\leq C_\beta(M,\kappa,T)\bigg(t_{n-\ell}^{-2\beta}(\E|(-A)^{-\beta}z_{\ell}|^{2M})^{\frac1{M}}
+\Delta t \sum_{m=\ell}^{n-1} \bigl(1+\frac{1}{t_{n-m}^{\frac{1}{2}+\kappa}}\bigr)
\E\left(|Y_{m}^{\ell}|_{L^2}^{2M}\right)^\frac1M\\
&+\Delta t\sum_{m=\ell}^{n-1}\bigl(1+\frac{1}{t_{n-m}^{\frac{1}{2}+\kappa}}\bigr)(\E|Z_{n}^1|_{L^{4}}^{4M})^{\frac1{2M}}(\E|Z_{n}^2|_{L^{4}}^{4M})^{\frac1{2M}}
\bigg).
\end{align*}

\end{lemma}

Before we give the proof of this result, let us mention that it will be useful when combined with the following discrete Gronwall Lemma, see for instance Lemma~7.1 in~\cite{Elliott_Larsson:92} for details. Lemma~\ref{lem:Gronwall} will also be used repeatedly in Section~\ref{sec:proof_num}.

\begin{lemma}\label{lem:Gronwall}
Let $\mu,\nu\in(0,1)$, and $T\in(0,\infty)$. Assume that $\Delta t=\frac{T}N$, for some $N\in\mathbb{N}^\star$; for $1\le n\le N$, let $t_n=n\Delta t$.

Assume that the sequence $\bigl(\phi_n\bigr)_{0\le n\le N}$, with values in $(0,\infty)$, satisfies the following condition: there exists $C_1,C_2$ such that for every $1\le n\le N$
\[
\phi_n\le C_1\bigl(1+t_n^{-1+\mu}\bigr)+C_2\Delta t\sum_{j=0}^{n-1}t_{n-j}^{-1+\nu}\phi_j.
\]
Then there exists $C$ such that $\phi_n\le C(1+t_{n}^{-1+\mu})$ for every $1\le n\le N$.
\end{lemma}

We now give a detailed proof of Lemma~\ref{lem:Pi_mild}. We only consider the case $q\in(2,\infty)$; the case $q=2$ is treated with similar arguments, but with a slightly different treatment of the stochastic integral.
\begin{proof}[Proof of Lemma~\ref{lem:Pi_mild}]
Note that $Y_{n}^{\ell}=Y_n^{1,\ell}+Y_{n}^{2,\ell}$, where
\begin{align*}
Y_{n}^{1,\ell}&=S_{\Delta t}^{n-\ell}z_\ell+\Delta t\sum_{m=\ell}^{n-1}S_{\Delta t}^{n-m}F_1'(X_m).Y_{m}^{\ell}+ \sum_{m=\ell}^{n-1}S_{\Delta t}^{n-m}e^{\tau A}\bigl(\sigma'(X_m).Y_{m}^{\ell}\bigr)\Delta W_m\\
&+\Delta t\sum_{m=\ell}^{n-1}F_{n,1}+\sum_{m=\ell}^{n-1}S_{\Delta t}^{n-m}e^{\tau A}\sigma''(X_m).(Z_m^1,Z_m^2) \Delta W_m
\end{align*}
and
\[
Y_{n}^{2,\ell}=\Delta t\sum_{m=\ell}^{n-1}S_{\Delta t}^{n-m}BF_2'(X_m).Y_m^\ell+\Delta t\sum_{m=\ell}^{n-1}S_{\Delta t}^{n-m}BF_{n,2},
\]
where $F_{n,j}=F_j''(X_m).(Z_m^1,Z_m^2)$, $j\in\left\{1,2\right\}$, are such that $G_{n-1}=F_{n-1,1}+BF_{n-1,2}$. By Property~\ref{ass:F},
\[
|F_{n,j}|_{L^q}^{2M}\le C|Z_{n}^1|_{L^{2q}}^{2M}\E|Z_{n}^2|_{L^{2q}}^{2M}.
\]

The quantity $Y_n^{1,\ell}$ is treated using properties of $S_{\Delta t}^n$ whereas energy inequalities are used for $Y_n^{2,\ell}$, which contains all the terms where the linear operator $B$ appears.

Using a discrete time version of formula~\eqref{eq:Ito_gamma} and the corresponding Burkh\"older-Davies-Gundy inequality, as well as the ideal property~\eqref{eq:ideal}, we get
\begin{align*}
\E|Y_{n}^{1,\ell}|_{L^q}^{2M}&\le C\E|S_{\Delta t}^{n-\ell}z_\ell|_{L^q}^{2M}+C\E\left(\Delta t\sum_{m=\ell}^{n-1}\bigl(\big|S_{\Delta t}^{n-m}F_1'(X_m).Y_m^{\ell}|_{L^q}^{2}+\big|S_{\Delta t}^{n-m}e^{\tau A}\bigl(\sigma'(X_m).Y_m^\ell\bigr)\big|_{R(L^2,L^q)}^{2}\bigr)\right)^M\\
&~+C\E\left(\Delta t\sum_{m=\ell}^{n-1}\bigl(\big|S_{\Delta t}^{n-m}F_{m,1}\big|_{L^q}^{2}+\big|S_{\Delta t}^{n-m}e^{\tau A}\sigma''(X_{m}).(Z_{m}^1,Z_{m}^2)\big|_{R(L^2,L^q)}^{2}\bigr)\right)^M\\
&\le Ct_{n-\ell}^{-2\beta M}\E|(-A)^{-\beta}z_\ell|_{L^q}^{2M}+C\E\left(\Delta t\sum_{m=\ell}^{n-1}(1+|(-A)^{\frac{1}{2q}}S_{\Delta t}^{n-m}|_{R(L^2,L^q)}^{2}\bigr)\E|Y_{m}^{\ell}|_{L^q}^{2} \right)^M\\
~&+C\E\left(\Delta t\sum_{m=\ell}^{n-1}(1+|(-A)^{\frac{1}{2q}}S_{\Delta t}^{n-m}|_{R(L^2,L^q)}^{2}\bigr)\E\bigl[|Z_{m}^{1}|_{L^{2q}}^{2}|Z_{m}^{2}|_{L^{2q}}^{2}\bigr]\right)^M,
\end{align*}
thanks to Lemma~\ref{lem:S_1} and Properties~\ref{ass:F} and~\ref{ass:sigma}. Thanks to Lemma~\ref{lem:S_2} and Minkowskii inequality, we obtain
\begin{align*}
\left(\E|Y_{n}^{1,\ell}|_{L^q}^{2M}\right)^{\frac1{M}}&\le Ct_{n-\ell}^{-2\beta}\E\left(|(-A)^{-\beta}z_\ell|^{2M}\right)^{\frac1{M}}+C\Delta t\sum_{m=\ell}^{n-1}\bigl(1+\frac{1}{t_{n-m}^{\frac{1}{2}+\frac{1}{q}+\kappa}}\bigr)\E\left(|Y_m^\ell|_{L^q}^{2M}\right)^{\frac1{M}}\\
&~+C\Delta t\sum_{m=\ell}^{n-1}\bigl(1+\frac{1}{t_{n-m}^{\frac{1}{2}+\frac{1}{q}+\kappa}}\bigr)\bigl(\E\left[|Z_m^1|_{L^{2q}}^{4M}\right]\E\left[|Z_m^2|_{L^{2q}}^{4M}\right]\bigr)^{\frac1{2M}}.
\end{align*}

We then estimate $Y_n^{2,\ell}$ with an energy inequality. First, note that 
$$
Y_n^{2,\ell}-Y_{n-1}^{2,\ell}= \Delta t \left( A Y_n^{2,\ell} + BF_2'(X_{n-1})  Y_{n-1}^{\ell}+Be^{\tau A} F_{n-1,2}\right).
$$
Then, multiply the above equation by $(Y_{n}^{2,\ell})^{q-1}$ and integrate in space. Recall that $A=\partial_{\xi\xi},\; B=\partial_\xi$, and that homogeneous Dirichlet boundary conditions are imposed. Standard manipulations, including using H\"older inequality and integration by parts, yield the following inequalities:
\begin{align*}
\frac1{q}\left( |Y_n^{2,\ell}|_{L^q}^q-|Y_{n-1}^{2,\ell}|_{L^q}^q\right)
&\le  \Delta t \int_0^1\bigl( (Y_n^{2,\ell})^{q-1}AY_n^{2,\ell} + B(F_2'(X_{n-1})Y_{n-1}^{\ell} +F_{n-1,2})(Y_n^{2,\ell})^{q-1}\bigr)d\xi \\
&\le -(q-1) \Delta t \int_0^1 (Y_n^{2,\ell})^{q-2} |\partial_\xi Y_n^{2,\ell}|^2 d\xi\\
&~+(q-1)\Delta t\int_0^1( F'(X_{n-1})Y_{n-1}^{\ell} +F_{n-1,2})(Y_n^{2,\ell})^{q-2}\partial_\xi Y_n^{2,\ell} d\xi\\
&\le -(q-1) \Delta t \int_0^1 (Y_n^{2,\ell})^{q-2} |\partial_\xi Y_n^{2,\ell}|^2 d\xi\\
&~+C\Delta t \int_0^1  ((Y_{n-1}^{\ell})^2 +(F_{n-1,2})^2)(Y_n^{2,\ell})^{q-2}d\xi+\Delta t\int_0^1 (Y_n^{2,\ell})^{q-2} |\partial_\xi Y_n^{2,\ell}|^2 d\xi\\
&\le C\Delta t \int_0^1  ((Y_{n-1}^{\ell})^2 +(F_{n-1,2})^2)(Y_n^{2,\ell})^{q-2}d\xi.
\end{align*}
Recall that we work with regularized coefficients (with $\delta>0$), so that $Y_n^{2,\ell}$ and $Y_n^{\ell}$ are sufficienctly regular so that the computations above are rigorous.

Applying H\"older inequality, then Lemma~\ref{lem:Gronwall}, we obtain
$$
 |Y_n^{2,\ell}|_{L^q}^q\le c \Delta t \sum_{m=\ell}^{n-1} (|Y_{m}^{\ell}|_{L^q}^2+|F_{m,2}|_{L^q}^2) |Y_{m+1}^{2,\ell}|_{L^q}^{q-2}.
$$
Define $\bar Y_n^{2,\ell}=\sup_{m=\ell,\dots,n} |Y_m^{2,\ell}|_{L^q}$; then
$$
|Y_{n}^{2,\ell}|_{L^q}^2\le ( \bar Y_n^{2,\ell})^2\le C \Delta t \sum_{m=\ell}^{n-1}\bigl(|Y_{m}^{\ell}|_{L^q}^2+|Z_m^{1}|_{L^{2q}}^2|Z_m^2|_{L^{2q}}^2\bigr).
$$
Finally, taking expectation and using Minkowskii inequality yield
\[
\E\left(|Y_n^{2,\ell}|_{L^q}^{2M}\right)^{\frac1M}\le C \Delta t \sum_{m=\ell}^{n-1}\Bigl(
\E\left(|Y_{m}^{\ell}|_{L^q}^{2M}\right)^\frac1M+\bigl(\E\left(|Z_{m}^1|_{L^{2q}}^{4M}\right)\E\left(|Z_{m}^2|_{L^{2q}}^{4M}\right)\bigr)^\frac{1}{2M}\Bigr).
\]
Gathering the estimates on $Y_n^{1,\ell}$ and $Y_n^{2,\ell}$ concludes the proof of Lemma~\ref{lem:Pi_mild}.

\end{proof}
\begin{rem}
For the case $q=2$; the contribution of the stochastic integral needs to be treated differently. We have for instance, for any $\kappa\in(0,\frac12)$,
\begin{align*}
\big|S_{\Delta t}^{n-m}e^{\tau A}\bigl(\sigma'(X_m).Y_m^\ell\bigr)\big|_{\mathcal L(L^2)}^{2}
&={\rm Tr}\Bigl(\bigl(\sigma'(X_m).Y_{m}^{\ell}\bigr)^2 S_{\Delta t}^{2(n-m)}e^{2\tau A}\Bigr)\\
&=\sum_{i}\big|\bigl(\sigma'(X_m).Y_{m}^{\ell}\bigr)^2\frac{e^{-\tau\lambda_i}}{(1+\Delta t\lambda_i)^{n-m}}e_i\big|_{L^2}\\
&\le C|Y_{m}^{\ell}|_{L^2}^{2}\sum_{i} \frac{1}{(1+\Delta t\lambda_i)^{2(n-m)}}|e_i|_{L^{\infty}}^{2}\\
&\le C_\kappa |Y_{m}^{\ell}|_{L^2}^{2} t_{n-m}^{-\frac12-\kappa}.
\end{align*} 
\end{rem}

\begin{proof}[Proof of Lemma~\ref{lem:Pi}]
For $\gamma=0$, Lemma~\ref{lem:Pi} is a straightforward consequence of Lemma~\ref{lem:Pi_mild} with $Z_n^1=Z_n^2=0$ and of the discrete Gronwall lemma, Lemma~\ref{lem:Gronwall}.

For $\gamma>0$, we write, with $Y_n^\ell=\Pi_{n-1:\ell}z_{\ell}$, 
$$
Y_{n}^{\ell}=S_{\Delta t}^{n-\ell}z_\ell+\Delta t\sum_{m=\ell}^{n-1}S_{\Delta t}^{n-m}G'(X_m).Y_{m}^{\ell}+ \sum_{m=\ell}^{n-1}S_{\Delta t}^{n-m}e^{\tau A}\bigl(\sigma'(X_m).Y_{m}^{\ell}\bigr)\Delta W_m,
$$
and, thanks to Lemmas~\ref{lem:S_1} and~\ref{lem:S_2}, and~\eqref{eq:norm_AB},
\begin{align*}
\left(\E(|(-A)^\gamma Y_n^\ell|_{L^q}^{2M}\right)^{\frac1{2M}}&\le c t_{n-\ell}^{-\beta-\gamma}\E\left(|(-A)^{-\beta}z_{\ell}|_{L^q}^{2M}\right)^{\frac1{2M}}
+c \Delta t \sum_{m=\ell}^{n-1} (t_{n-m}^{-\frac12-\kappa-\gamma}+1) \left(\E| Y_m^\ell|_{L^q}^{2M}\right)^{\frac1{2M}}\\
&+ c \left(\Delta t \sum_{m=\ell}^{n-1} t_{n-m}^{-\frac12-\frac1q-2\gamma-\kappa} \left(\E| Y_m^\ell|_{L^q}^{2M}\right)^{\frac1{M}}\right)^{\frac12}.
\end{align*}
Using the estimate obtained for $\gamma=0$, and the condition $\frac{1}{2}+\frac{1}{q}+2\gamma+\kappa<1$, for sufficiently small $\kappa>0$, then concludes the proof.
\end{proof}

\begin{proof}[Proof of Lemma~\ref{lem:Mall_2}]

Again, we only treat the case $q\in(2,\infty)$.

Define $Y_{n}^{\ell}=\Pi_{n-1:\ell+1}z$, where $0\le \ell\le n-1$.
If $\ell=n-1$, $Y_{n}^{\ell}=z$, and thus for $s\in\bigl(\ell\Delta t,(\ell+1)\Delta t\bigr)$ one has $\D_sY_{\ell+1}^{\ell}=0$.
If $n>\ell+1$,
\[
Y_{n}^{\ell}=\Pi_{n-1}Y_{n-1}^{\ell}=S_{\Delta t}Y_{n-1}^{\ell}+\Delta tS_{\Delta t}G'(X_{n-1}).Y_{n-1}^{\ell}+S_{\Delta t}e^{\tau A}\bigl(\sigma'(X_{n-1}).Y_{n-1}^{\ell}\bigr)\Delta W_{n-1}.
\]
Using the chain rule and the identity $\D_s\Delta W_{n-1}=0$ for $s<(\ell+1)\Delta t\le (n-1)\Delta t$, for every $\theta\in H$
\begin{align*}
\D_s^\theta Y_{n}^{\ell}&=S_{\Delta t}\D_s^\theta Y_{n-1}^{\ell}+\Delta tS_{\Delta t}G'(X_{n-1}).\D_s^\theta Y_{n-1}^{\ell}+S_{\Delta t}e^{\tau A}\bigl(\sigma'(X_{n-1}).\D_s^\theta Y_{n-1}^{\ell}\bigr)\Delta W_{n-1}\\
&~+\Delta tS_{\Delta t}G''(X_{n-1}).(\D_s^\theta X_{n-1},Y_{n-1}^{\ell})+S_{\Delta t}e^{\tau A}\bigl(\sigma''(X_{n-1}).(\D_s^\theta X_{n-1},Y_{n-1}^{\ell})\bigr)\Delta W_{n-1},
\end{align*}
We apply Lemma~\ref{lem:Pi_mild}, with $\ell$ replaced by $\ell +1$, and $z_{\ell+1}=0$, $Z_m^1=\D_s^\theta X_{m}$, $Z_m^2=Y_{m}^{\ell})$, and $M=1$. This gives, for $n\ge \ell+2$,
\begin{align*}
\E|\D_s^{\theta_\ell} Y_{n}^{\ell}|_{L^q}^{2}&\leq C\Delta t \sum_{m=\ell+1}^{n-1} \bigl(t_{n-m}^{-\frac{1}{2}-\frac{1}{q}-\kappa}+1\bigr)
\E|\D_s^{\theta_\ell} Y_{m}^{\ell}|_{L^q}^{2}\\
&~+C\Delta t\sum_{m=\ell+1}^{n-1} \bigl(t_{n-m}^{-\frac{1}{2}-\frac{1}{q}-\kappa}+1\bigr)\bigl(\E|\D_s^{\theta_\ell}X_m|_{L^{2q}}^{4}\bigr)^{\frac12}\bigl(\E|Y_m^\ell|_{L^{2q}}^{4}\bigr)^{\frac{1}{2}}.
\end{align*}


Thanks to Lemma~\ref{lem:Mall_1} and Lemma~\ref{lem:Pi},  when $m>\ell+1$,
\[
\E|\D_s^{\theta_\ell}X_m|_{L^{2q}}^{4}\le C\E|\theta_{\ell}|_{L^{2q}}^{4};
\]
and
\[
\E|Y_m^\ell|_{L^{2q}}^{4}=\E|\Pi_{m-1:\ell+1}z|_{L^{2q}}^{4}\le C \, t_{m-\ell-1}^{-4(\frac12-\kappa)}\big|(-A)^{-\frac12+\kappa}z\big|_{L^{2q}}^{4}.
\]
For $m=\ell+1$, we use $\E|Y_{\ell+1}^\ell|_{L^{2q}}^{4}=|z|_{L^{2q}}^{4}$.

Thus
\begin{align*}
C\Delta t\sum_{m=\ell+1}^{n-1}\bigl(1+&\frac{1}{t_{n-m}^{\frac{1}{2}+\frac{1}{q}+\kappa}}\bigr)\bigl(\E|\D_s^{\theta_\ell}X_m|_{L^{2q}}^{4}\bigr)^{\frac12}\bigl(\E|Y_m^\ell|_{L^{2q}}^{4}\bigr)^{\frac12}\\
&\le C \Delta t\bigl(1+\frac{1}{t_{n-l-1}^{\frac{1}{2}+\frac{1}{q}+\kappa}}\bigr)
\bigl(\E|\theta_{\ell}|_{L^{2q}}^{4}\bigr)^{\frac12}|z|_{L^{2q}}^{2}\\
&~+C\Delta t \mathds{1}_{n>\ell+2} \sum_{m=\ell+2}^{n-1}\bigl(1+\frac{1}{t_{n-m}^{\frac{1}{2}+\frac{1}{q}+\kappa}}\bigr)\frac{1}{t_{m-\ell-1}^{1-2\kappa}}\bigl(\E|\theta_{\ell}|_{L^{2q}}^{4}\bigr)^{\frac12}\big|(-A)^{-\frac12+\kappa}z\big|_{L^{2q}}^{2}\\
&\le C\Delta t\bigl(1+\frac{1}{t_{n-l-1}^{\frac{1}{2}+\frac{1}{q}+\kappa}}\bigr)\bigl(\E|\theta_{\ell}|_{L^{2q}}^{4}\bigr)^{\frac12}|z|_{L^{2q}}^{2}\\
&~+C\mathds{1}_{n>\ell+2}\bigl(1+\frac{1}{t_{n-\ell-2}^{\frac{1}{2}+\frac{1}{q}-\kappa}}\bigr)\bigl(\E|\theta_{\ell}|_{L^{2q}}^{4}\bigr)^{\frac12}\big|(-A)^{-\frac12+\kappa}z\big|_{L^{2q}}^{2},
\end{align*}
using a straightforward comparison between the series and an integral. Applying Lemma~\ref{lem:Gronwall} concludes the proof.
\end{proof}

\section{Proof of Theorem~\ref{theo:num}}\label{sec:proof_num}

Recall the definition of the scheme, see~\eqref{eq:scheme_result}: for $n\in\left\{0,\ldots,N-1\right\}$,
\begin{equation}\label{eq:scheme_num}
X_{n+1}=S_{\Delta t}X_n+\Delta tS_{\Delta t}G(X_n)+S_{\Delta t} \sigma(X_n) \Delta W_n,
\end{equation}
with the initial condition $X_0=x$, the condition $N\Delta t=T$, and the Wiener increments $\Delta W_n=W\bigl((n+1)\Delta t\bigr)-W\bigl(n\Delta t\bigr)$. Recall the notation $S_{\Delta t}=\bigl(I-\Delta t A)^{-1}$.

Let $\varphi$ be a function satisfying Assumption~\ref{ass:phi}, and $u(t,x)=\E\bigl[\varphi\bigl(X_t(x)\bigr)\bigr]$ be defined by~\eqref{eq:u}.

In order to justify all computations below, it is convenient to replace $G$ and $\sigma$ in~\eqref{eq:scheme_num} with the regularized coefficients $G_\delta$ and $\sigma_\delta$ introduced in Section~\ref{sec:results}, and to consider $u_\delta$ defined by~\eqref{eq:u_N} instead of $u$. Since all upper bounds hold true uniformly with respect to $\delta$, passing to the limit $\delta\to 0$ allows us to remove this regularization parameter. To simplify the notation, we do not mention $\delta$ in the computations.

Associated with the scheme~\eqref{eq:scheme_num}, we introduce an auxiliary, continuous-time, process $\bigl(\tilde{X}(t)\bigr)_{t\in[0,T]}$, defined on each interval $[t_n,t_{n+1}]$ by
\begin{equation}\label{eq:tilde_num_1}
\tilde{X}(t)=X_n+(t-t_n)S_{\Delta t}AX_n+(t-t_n)S_{\Delta t}G(X_n)+S_{\Delta t}\sigma(X_n)\bigl(W(t)-W(t_n)\bigr).
\end{equation}
Equivalently, $\tilde{X}(t_n)=X_n$, and for $t\in[t_n,t_{n+1}]$,
\begin{equation}\label{eq:tilde_num_2}
d\tilde{X}(t)=S_{\Delta t}AX_ndt+S_{\Delta t}G(X_n)dt+S_{\Delta t}\sigma(X_n)dW(t).
\end{equation}

Note that Lemma \ref{lem:moments_X} is still true for $\delta=\tau=0$ so that we have bounds on the moments of $X_n$ in $D((-A)^{\alpha})$, $\alpha<\frac14$. Moreover
\begin{equation}\label{eq:tilde_num_alpha}
|(-A)^{\alpha}\tilde X(t)|_{L^p}\le c |(-A)^{\alpha}X_n|_{L^p}, \; t\in [t_n,t_{n+1}).
\end{equation}

Using the notation $\ell_s=\ell$ if $s\in[t_{\ell},t_{\ell+1})$, for $s\in[0,T]$, we have the formulation
\begin{equation}
X_k=S_{\Delta t}^kx+\Delta t\sum_{\ell=0}^{k-1}S_{\Delta t}^{k-\ell}G(X_\ell)+\int_{0}^{t_k}S_{\Delta t}^{k-\ell_s}\sigma(X_{\ell_s})dW(s).
\end{equation}

Following the standard approach, introduced first in the SDE setting, see~\cite{Talay:86} and the monographs~\cite{Kloeden_Platen:92} and~\cite{Milstein_Tretyakov:04}, the weak error~\eqref{eq:theo_num} is decomposed as follows:
\begin{equation}\label{e0}
\begin{aligned}
\E\varphi\bigl(X(T)\bigr)-\E\varphi\bigl(X_N\bigr)&
=\E\bigl[u(T,x)-u(0,X_N)\bigr]\\
&=\sum_{k=0}^{N-1}\E\bigl[u(T-t_k,X_k)-u(T-t_{k+1},X_{k+1})\bigr]\\
&=\E\bigl[u(T-\Delta t,X_1)-u(T,x)\bigr]+\sum_{k=1}^{N-1}\bigl(a_k+b_k+c_k\bigr),
\end{aligned}
\end{equation}
with
\[
\begin{gathered}
a_k=\int_{t_{k}}^{t_{k+1}}\E\langle A\tilde{X}(t)-AS_{\Delta t}X_k,D u\bigl(T-t,\tilde{X}(t)\bigr)\rangle dt,\\
b_k=\int_{t_{k}}^{t_{k+1}}\E\langle G(\tilde{X}(t))-S_{\Delta t}G(X_k),D u\bigl(T-t,\tilde{X}(t)\bigr)\rangle dt,\\
c_k=\frac{1}{2}\int_{t_{k}}^{t_{k+1}}\E {\rm Tr}\Bigl(\bigl[\sigma(\tilde{X}(t))^2-S_{\Delta t}\sigma(X_k)^2S_{\Delta t}\bigr]D^2 u\bigl(T-t,\tilde{X}(t)\bigr)\Bigr)dt,
\end{gathered}
\]
where in $c_k$ we have used the property $\sigma(\cdot)^\star=\sigma(\cdot)$, see~\eqref{eq:sigma_star}, Property~\ref{ass:sigma}. 

In the following sections, we successively treat the terms $\E\bigl[u(T-\Delta t,X_1)-u(T,x)\bigr]$, $a_k$, $b_k$ and $c_k$. A technical result is given in Section~\ref{sec:aux_num}.

We will control the error terms, in terms of $\Delta t^{\frac12-\kappa}$ with positive, arbitrarily small $\kappa$. We do not try to obtain optimal constants. The value of $\kappa$ may change from line to line. At the end of the proof, gathering the estimates and choosing an appropriate $\kappa$ gives the result.

\subsection{Control of $\E\bigl[u(T-\Delta t,X_1)-u(T,x)\bigr]$}

We note that
\begin{align*}
\big|\E\bigl[u(T-\Delta t,X_1)-u(T,x)\bigr]\big|&\le \big|\E\bigl[u(T-\Delta t,X_1)-u(T-\Delta t,X(\Delta t))\bigr]\big|\\
&\le \frac{C}{(T-\Delta t)^{1-\kappa}}(1+|x|_{L^{\max(p,2q)}})^{K+1}\bigl(\E\big|(-A)^{-1+\kappa}\bigl(X_1-X(\Delta t)\bigr)\big|_{L^q}^2\bigr)^{\frac12},
\end{align*}
using Theorem~\ref{theo:D1}.

We then write $X_1-X(\Delta t)=\bigl(X_1-x\bigr)-\bigl(X(\Delta t)-x\bigr)$, and note that
\begin{align*}
\E|(-A)^{-1+\kappa}(X_1-x)|_{L^q}^{2}&\le C|(-A)^{-1+\kappa}(S_{\Delta t}-I)x|_{L^{q}}^{2}\\
&~+\Delta t|(-A)^{-1+\kappa}S_{\Delta t}G(x)|_{L^{q}}^{2}+C\Delta t|(-A)^{-1+\kappa}|_{R(L^2,L^q)}^2|S_{\Delta t}\sigma(x)|_{\mathcal{L}(L^2)}^2\\
&\le C\Delta t^{1-2\kappa}|x|_{L^q}^{2}+C\Delta t.
\end{align*}
We have used the two following inequalities. First, for every $\beta\in[0,1)$ and $q\in[2,\infty)$, there exists $C_{\beta,q}$ such that
\begin{equation}\label{eq:error_S_Deltat}
|(-A)^{-\beta}(S_{\Delta t}-I)|_{\mathcal{L}(L^q)}= \Delta t|(-A)^{1-\beta}S_{\Delta t}|_{\mathcal{L}(L^q)} \le C_{\beta,q}\Delta t^\beta,
\end{equation}
using the identity $S_{\Delta t}-I=\Delta t A S_{\Delta t}$ and Lemma~\ref{lem:S_1} (with $n=1$).

Second, adapting the proof of Lemma~\ref{lem:S_2}, for $\alpha>\frac14$, 
\[
|(-A)^{-\alpha}|_{R(L^2,L^q)}^2<\infty.
\]

Similarly,
\[
\E|(-A)^{-1+\kappa}(X(\Delta t)-x)|_{L^q}^{2}\le C\Delta t^{1-2\kappa}|x|_{L^q}^{2}+C\Delta t.
\]

We thus obtain
\begin{equation}\label{e72}
\big|\E\bigl[u(T-\Delta t,X_1)-u(T,x)\bigr]\big|\le C(T)(1+|x|_{L^{\max(p,2q)}})^{K+1}\Delta t^{\frac12-\kappa}.
\end{equation}

\subsection{Control of $a_k$}

\subsubsection{Decompositions}

For each $k\in\left\{1,\ldots,N-1\right\}$, $a_k$ is decomposed into the following terms:
\begin{equation}\label{e73}
a_k=a_k^1+a_k^2=\bigl(a_{k}^{1,1}+a_{k}^{1,2}+a_{k}^{1,3}\bigr)
+\bigl(a_{k}^{2,1}+a_{k}^{2,2}+a_{k}^{2,3}\bigr),
\end{equation}
where
\begin{gather*}
a_{k}^{1}=\E\int_{t_{k}}^{t_{k+1}}\langle A(I-S_{\Delta t})X_k,Du(T-t,\tilde{X}(t))\rangle dt,\\
a_{k}^{2}=\E\int_{t_{k}}^{t_{k+1}}\langle A(\tilde{X}(t)-X_k),Du(T-t,\tilde{X}(t))\rangle dt,
\end{gather*}
and $a_{k}^{1}$ and $a_{k}^{2}$ are further decomposed into
\begin{align*}
a_{k}^{1,1}&=-\Delta t\E\int_{t_{k}}^{t_{k+1}}\langle A^2 S_{\Delta t}^{k+1}x,Du(T-t,\tilde{X}(t))\rangle dt,\\
a_{k}^{1,2}&=-\Delta t\E\int_{t_{k}}^{t_{k+1}}\langle \Delta t\sum_{\ell=0}^{k-1}A^2S_{\Delta t}^{k-\ell+1}G(X_{\ell}),Du(T-t,\tilde{X}(t))\rangle dt,\\
a_{k}^{1,3}&=-\Delta t\E\int_{t_{k}}^{t_{k+1}}\langle \sum_{\ell=0}^{k-1}A^2S_{\Delta t}^{k-\ell+1}\sigma(X_\ell)\Delta W_\ell,Du(T-t,\tilde{X}(t))\rangle dt,
\end{align*}
and, using~\eqref{eq:tilde_num_1},
\begin{align*}
a_{k}^{2,1}&=\E\int_{t_{k}}^{t_{k+1}}(t-t_k)\langle S_{\Delta t}A^2 X_k,Du(T-t,\tilde{X}(t))\rangle dt,\\
a_{k}^{2,2}&=\E\int_{t_{k}}^{t_{k+1}}(t-t_k)\langle AS_{\Delta t}G(X_k),Du(T-t,\tilde{X}(t))\rangle dt,\\
a_{k}^{2,3}&=\E\int_{t_{k}}^{t_{k+1}}\langle AS_{\Delta t}\sigma(X_k)\bigl(W(t)-W(t_k)\bigr),Du(T-t,\tilde{X}(t))\rangle dt.
\end{align*}

Indeed, $I-S_{\Delta t}=-\Delta t AS_{\Delta t}$, and
\begin{equation}\label{eq:mild_sol_num}
X_k=S_{\Delta t}^{k}x+\Delta t\sum_{\ell=0}^{k-1}S_{\Delta t}^{k-\ell}G(X_{\ell})+\sum_{\ell=0}^{k-1}S_{\Delta t}^{k-\ell}\sigma(X_\ell)\Delta W_\ell.
\end{equation}

\subsubsection{Treatment of $a_{k}^1$}

We treat succesively the terms $a_{k}^{1,1}$, $a_{k}^{1,2}$ and $a_{k}^{1,3}$. The first quantity only needs elementary arguments and Theorem~\ref{theo:D1}, with $\beta\in[0,\frac12)$. The second quantity requires the stronger version of Theorem~\ref{theo:D1}, with $\beta\in[0,1)$, contrary to~\cite{Debussche:11}, due to the Burgers type nonlinearity. The third quantity requires the use of {the Malliavin calculus duality formula}, and of Theorem~\ref{theo:D2} with $\beta,\gamma\in[0,\frac12)$.

We also use repeatedly \eqref{eq:tilde_num_alpha} combined with Cauchy-Schwarz inequality.

\subsubsection*{Treatment of $a_{k}^{1,1}$}

Using Theorem~\ref{theo:D1}, with $\beta=\frac12-\kappa$, we get for $k\in\left\{1,\ldots,N-1\right\}$,
\begin{align*}
|a_k^{1,1}|&\le C_\kappa(1+|x|_{L^{\max(p,2q)}})^{K+1}\Delta t\int_{t_k}^{t_{k+1}}\frac{1}{(T-t)^{\frac12-\kappa}}\big| (-A)^{-\frac12+\kappa}A^2S_{\Delta t}^{k+1}x\big|_{L^{2q}}dt\\
&\le C_\kappa(1+|x|_{L^{\max(p,2q)}})^{K+1}\Delta t\int_{t_k}^{t_{k+1}}\frac{1}{(T-t)^{\frac12-\kappa}}\big| (-A)^{\frac12+2\kappa}S_{\Delta t}(-A)^{1-\kappa}S_{\Delta t}^{k}x\big|_{L^{2q}}dt\\
&\le C_\kappa(1+|x|_{L^{\max(p,2q)}})^{K+1}|x|_{L^q}\frac{\Delta t^{\frac12-2\kappa}}{t_{k}^{1-\kappa}}\int_{t_k}^{t_{k+1}}\frac{1}{(T-t)^{\frac12-\kappa}}dt,
\end{align*}
using Lemma~\ref{lem:S_1}.

\subsubsection*{Treatment of $a_{k}^{1,2}$}

Similarly, thanks to~\eqref{eq:norm_AB}, the boundedness of the mappings $F_1$ and $F_2$ from $L^q$ to $L^q$, thanks to Property~\ref{ass:F}, using Theorem~\ref{theo:D1} with $\beta=1-\kappa$,
\begin{align*}
|a_k^{1,2}|
&\le C_\kappa(1+|x|_{L^{\max(p,2q)}})^{K+1}\Delta t\int_{t_k}^{t_{k+1}}\frac{1}{(T-t)^{1-\kappa}}\Delta t\sum_{\ell=0}^{k-1}\big| (-A)^{\frac12+3\kappa}S_{\Delta t}(-A)^{1-\kappa}S_{\Delta t}^{k-\ell}\big|_{\mathcal{L}(L^{2q})}dt\\
& \le C_\kappa(1+|x|_{L^{\max(p,2q)}})^{K+1}|x|_{L^q}\Delta t^{\frac12-3\kappa}\int_{t_k}^{t_{k+1}}\frac{1}{(T-t)^{1-\kappa}}dt \Delta t\sum_{\ell=0}^{k-1}\frac{1}{t_{k-\ell}^{1-\kappa}}\\
&\le C_\kappa(1+|x|_{L^{\max(p,2q)}})^{K+1}|x|_{L^q}\Delta t^{\frac12-3\kappa}\int_{t_k}^{t_{k+1}}\frac{1}{(T-t)^{1-\kappa}}dt.
\end{align*}

\subsubsection*{Treatment of $a_{k}^{1,3}$}

Let $k\in\left\{1,\ldots,N-1\right\}$. For technical reasons, we decompose $a_{k}^{1,3}=a_{k}^{1,3,1}+a_{k}^{1,3,2}$ where
\begin{align*}
a_{k}^{1,3,1}&=-\Delta t\E\int_{t_{k}}^{t_{k+1}}\langle \sum_{\ell=0}^{k-2}A^2S_{\Delta t}^{k-\ell+1}\sigma(X_\ell)\Delta W_\ell,Du(T-t,\tilde{X}(t))\rangle dt\\
&=-\Delta t\E\int_{t_{k}}^{t_{k+1}}\langle \int_{0}^{t_{k-1}}A^2 S_{\Delta t}^{k-\ell_s+1}\sigma(X_{\ell_s})dW(s),Du(T-t,\tilde{X}(t))\rangle dt\\
a_{k}^{1,3,2}&=-\Delta t\E\int_{t_{k}}^{t_{k+1}}\langle A^2S_{\Delta t}^{2}\sigma(X_{k-1})\Delta W_{k-1},Du(T-t,\tilde{X}(t))\rangle dt,
\end{align*}
with the convention that $a_{1}^{1,3,1}=0$.

We first treat $a_{k}^{1,3,1}$. Using {the Malliavin calculus duality formula,}
\begin{align*}
a_{k}^{1,3,1}&=-\Delta t\E\int_{t_{k}}^{t_{k+1}}\langle \int_{0}^{t_{k-1}}A^2 S_{\Delta t}^{k-\ell_s+1}\sigma(X_{\ell_s})dW(s),Du(T-t,\tilde{X}(t))\rangle dt\\
&=-\Delta t\E\int_{t_k}^{t_{k+1}}\int_{0}^{t_{k-1}}{\rm Tr}\Bigl(\sigma(X_{\ell_s})^\star A^2S_{\Delta t}^{k-\ell_s+1}D^2u(T-t,\tilde{X}(t))\D_s\tilde{X}(t)\Bigr)dsdt\\
&=-\Delta t\E\int_{t_k}^{t_{k+1}}\int_{0}^{t_{k-1}}{\rm Tr}\Bigl(\sigma(X_{\ell_s})^\star A^2S_{\Delta t}^{k-\ell_s+1}D^2u(T-t,\tilde{X}(t))U(t,s)S_{\Delta t}\sigma(X_{\ell_s})\Bigr)dsdt\\
&=-\Delta t\E\int_{t_k}^{t_{k+1}}\int_{0}^{t_{k-1}}\sum_{i} D^2u(T-t,\tilde{X}(t)).\Bigl(A^2S_{\Delta t}^{k-\ell_s+1}\sigma(X_{\ell_s})^2 e_i,U(t,s)S_{\Delta t}e_i\Bigr)dsdt,
\end{align*}
where we use that $\sigma(x)^\star=\sigma(x)$, and we have introduced the linear operator $U(t,s)$ such that $\D_s\tilde{X}(t)=U(t,s)S_{\Delta t}\sigma(X_{\ell_s})$. We then apply Theorem~\ref{theo:D2}.

On the one hand, using Property~\ref{ass:sigma},
\begin{align*}
\bigl(\E\big|A^{-\frac12+\kappa+2}S_{\Delta t}^{k-\ell_s+1}&\sigma(X_{\ell_s})^2e_i\big|_{L^{4q}}^2\bigr)^{\frac12}\\
&\le \big|(-A)^{1-\kappa}S_{\Delta t}^{k-\ell_s}\big|_{\mathcal{L}(L^{4q})}\big|(-A)^{\frac12+2\kappa}S_{\Delta t}\big|_{\mathcal{L}(L^{4q})}\bigl(\E\big|\sigma(X_{\ell_s})^2e_i\big|_{L^{4q}}^2\bigr)^{\frac12}\\
&\le \frac{C\Delta t^{-\frac12-2\kappa}}{t_{k-\ell_s}^{1-\kappa}}|e_i|_{L^{4q}},
\end{align*}
thanks to Lemma~\ref{lem:S_1}, under the condition that $\ell_s<k-1$.

On the other hand, we use Lemma~\ref{lem:U_num}, see Section~\ref{sec:aux_num}. Thanks to Theorem~\ref{theo:D2}, we thus have
\begin{align*}
|a_{k}^{1,3,1}|&\le C\Delta t(1+|x|_{L^{\max(p,2q)}})^{K+1}\int_{t_k}^{t_{k+1}}\int_{0}^{t_{k-1}}\frac{C\Delta t^{-\frac12-2\kappa}}{(T-t)^{1-\kappa}t_{k-\ell_s}^{1-\kappa}}\sum_i \bigl(\E|(-A)^{-\frac12+\kappa}U(t,s)S_{\Delta t}e_i|_{L^{4q}}^{4}\bigr)^{\frac14}dsdt\\
&\le C\Delta t^{\frac12-3\kappa}(1+|x|_{L^{\max(p,2q)}})^{K+1}\int_{t_k}^{t_{k+1}}\frac{C}{(T-t)^{1-\kappa}}dt \bigl(\Delta t\sum_{\ell=0}^{k-2}\frac{1}{t_{k-\ell}^{1-\kappa}}\bigr)\sum_i\bigl(|(-A)^{-\frac12+2\kappa}S_{\Delta t}e_i|_{L^{4q}}\\
&\hspace{12cm}+C\Delta t^{\frac12}|S_{\Delta t}e_i|_{L^{4q}}\bigr)dt\\
&\le C\Delta t^{\frac12-5\kappa}(1+|x|_{L^{\max(p,2q)}})^{K+1}\int_{t_k}^{t_{k+1}}\frac{C}{(T-t)^{1-\kappa}}dt
\end{align*}

Indeed, $\sum_i\bigl(|(-A)^{-\frac12+2\kappa}S_{\Delta t}e_i|_{L^{4q}}+C\Delta t^{\frac12}|S_{\Delta t}e_i|_{L^{4q}}\bigr)\le C\Delta t^{-3\kappa}$.

It remains to treat $a_{k}^{1,3,2}$. This is done with much simpler arguments: using Theorem~\ref{theo:D1},
\begin{align*}
|a_{k}^{1,3,2}|&\le C\Delta t(1+|x|_{L^{\max(p,2q)}})^{K+1}\int_{t_k}^{t_{k+1}}\frac{1}{(T-t)^{\frac12-\kappa}} \bigl(\Delta t\E|(-A)^{-\frac12+\kappa}A^2S_{\Delta t}^2\sigma(X_{k-1})|_{R(L^2,L^{2q})}^2\bigr)^{\frac12}dt\\
&\le C\Delta t(1+|x|_{L^{\max(p,2q)}})^{K+1}\int_{t_k}^{t_{k+1}}\frac{1}{(T-t)^{\frac12-\kappa}}dt  \Delta t^{\frac12-1+\kappa},
\end{align*}
using $|(-A)^{-\frac12+3\kappa}|_{R(L^2,L^2)}<\infty$ and $|(-A)^{1-\kappa}S_{\Delta t}|_{\mathcal{L}(L^2)}\le C\Delta t^{-1+\kappa}$.

\subsubsection*{Conclusion}

Gathering the estimates on $a_{k}^{1,1}$, $a_{k}^{1,2}$ and $a_{k}^{1,3}$, and summing for $k\in\left\{1,\ldots,N-1\right\}$, we obtain
\begin{equation}\label{eq:num_a^1}
\sum_{k=1}^{N}\big|a_{k}^1\big|\le C\Delta^{\frac12-5\kappa}(1+|x|_{L^{\max(p,2q)}})^{K+1} \int_{0}^{T}\frac{1}{(T-t)^{1-\kappa}}\bigl(1+\frac{1}{t^{1-\kappa}}\bigr)dt.
\end{equation}

\subsubsection{Treatment of $a_k^2$}

\subsubsection*{Treatment of $a_{k}^{2,1}$}

Since $AS_{\Delta t}=-\frac{1}{\Delta t}(I-S_{\Delta t})$, we rewrite
\[a_{k}^{2,1}=-\E\int_{t_{k}}^{t_{k+1}}\frac{(t-t_k)}{\Delta t}\langle A(I-S_{\Delta t}) X_k,Du(T-t,\tilde{X}(t))\rangle dt\]
and observe that the right-hand side has the same structure as $a_k^1$. Using the straightforward inequality $t-t_k\le \Delta t$ when $t_k\le t\le t_{k+1}$, we thus directly obtain that $\sum_{k=1}^{N-1}|a_k^{2,1}|$ is bounded from above by the right-hand side of~\eqref{eq:num_a^1}.

\subsubsection*{Treatment of $a_{k}^{2,2}$}

We again use Theorem~\ref{theo:D1} (with $\beta=1-\kappa$), inequality~\eqref{eq:norm_AB} with Proposition~\ref{ass:F}, and obtain
\begin{align*}
|a_{k}^{2,2}|&=\big|\E\int_{t_{k}}^{t_{k+1}}(t-t_k)\langle AS_{\Delta t}G(X_k),Du(T-t,\tilde{X}(t))\rangle dt\big|\\
&\le C\Delta t(1+|x|_{L^{\max(p,2q)}})^{K+1}\int_{t_{k}}^{t_{k+1}}\frac{1}{(T-t)^{1-\kappa}}\big|(-A)^{\frac12+2\kappa}S_{\Delta t}(-A)^{-\frac12-\kappa}G(X_k)\big|_{L^{2q}} dt\\
&\le C\Delta t^{\frac12-2\kappa}(1+|x|_{L^{\max(p,2q)}})^{K+1}\int_{t_{k}}^{t_{k+1}}\frac{1}{(T-t)^{1-\kappa}}dt,
\end{align*}
thanks to Lemma~\ref{lem:S_1}.

\subsubsection*{Treatment of $a_{k}^{2,3}$}

To treat this term, we again use {the Malliavin calculus duality formula}. Writing the Wiener increment as a stochastic integral, we obtain
\begin{align*}
a_{k}^{2,3}&=\E\int_{t_{k}}^{t_{k+1}}\langle \int_{t_k}^{t}AS_{\Delta t}\sigma(X_k)dW(s),Du(T-t,\tilde{X}(t))\rangle dt\\
&=\E\int_{t_k}^{t_{k+1}}\E\int_{t_k}^{t}{\rm Tr}\Bigl(\sigma(X_k)^\star S_{\Delta t}A D^2u(T-t,\tilde{X}(t))\D_s\tilde{X}(t)\Bigr)dsdt\\
&=\E\int_{t_k}^{t_{k+1}}\E\int_{t_k}^{t}{\rm Tr}\Bigl(\sigma(X_k)^\star S_{\Delta t}A D^2u(T-t,\tilde{X}(t))S_{\Delta t}\sigma(X_k)\Bigr)dsdt\\
&=\E\int_{t_k}^{t_{k+1}}(t-t_k)\sum_i D^2u(T-t,\tilde{X}(t)).\bigl(S_{\Delta t}e_i,AS_{\Delta t}\sigma(X_k)^2 e_i\bigr)dt,
\end{align*}
where we have used $\sigma(x)^\star=\sigma(x)$, and the equality $\D_s\tilde{X}(t)=S_{\Delta t}\sigma(X_k)$ for $t_k\le s<t\le t_{k+1}$, obtained from~\eqref{eq:tilde_num_1}. We then use Theorem~\ref{theo:D2}
and  obtain
\begin{align*}
|a_{k}^{2,3}|&\le C\Delta t(1+|x|_{L^{\max(p,2q)}})^{K+1} \int_{t_k}^{t_{k+1}}\frac{1}{(T-t)^{1-\kappa}}dt\sum_i\big|(-A)^{-\frac12+\kappa}S_{\Delta t}e_i\big|_{L^{4q}}\big|(-A)^{\frac12+\kappa}S_{\Delta t}\sigma(X_k)^2e_i\big|_{L^{4q}}\\
&\le C\Delta t(1+|x|_{L^{\max(p,2q)}})^{K+1} \int_{t_k}^{t_{k+1}}\frac{1}{(T-t)^{1-\kappa}}dt \bigl(\Delta t^{-2\kappa}\sum_i\frac{1}{\lambda_i^{\frac12+\kappa}}\bigr)\Delta t^{-\frac12-\kappa}.
\end{align*}

\subsubsection*{Conclusion}

Gathering the estimates on $a_{k}^{2,1}$, $a_{k}^{2,2}$ and $a_{k}^{2,3}$, and summing for $k\in\left\{1,\ldots,N-1\right\}$, we obtain
\begin{equation}\label{eq:num_a^2}
\sum_{k=1}^{N}\big|a_{k}^2\big|\le C\Delta^{\frac12-2\kappa}(1+|x|_{L^{\max(p,2q)}})^{K+1} \int_{0}^{T}\frac{1}{(T-t)^{1-\kappa}}\bigl(1+\frac{1}{t^{1-\kappa}}\bigr)dt.
\end{equation}

\subsection{Control of $b_k$}

\subsubsection{Decompositions}

For each $k\in\left\{1,\ldots,N-1\right\}$, $b_k$ is decomposed into the following terms:
\begin{equation}\label{e77}
b_k=b_k^1+b_k^2=b_k^1+\bigl(b_{k}^{2,1}+b_{k}^{2,2}+b_{k}^{2,3}+b_k^{2,4}\bigr),
\end{equation}
where
\begin{align*}
b_k^1&=\int_{t_{k}}^{t_{k+1}}\E\langle (I-S_{\Delta t})G(X_k),D u\bigl(T-t,\tilde{X}(t)\bigr)\rangle dt,\\
b_k^2&=\int_{t_{k}}^{t_{k+1}}\E\langle G(\tilde{X}(t))-G(X_k),D u\bigl(T-t,\tilde{X}(t)\bigr)\rangle dt\\
&=\int_{t_{k}}^{t_{k+1}}\E\sum_i \bigl[G_i(\tilde{X}(t))-G_i(X_k)\bigr] \partial_iu\bigl(T-t,\tilde{X}(t)\bigr)dt,
\end{align*}
where $G_i(\cdot)=\langle G(\cdot),e_i\rangle$ and $\partial_i u(\cdot,\cdot)=\langle Du(\cdot,\cdot),e_i\rangle$.

In addition, $b_k^2$ is further decomposed with
\begin{align*}
b_k^{2,1}&=\E\int_{t_{k}}^{t_{k+1}}\int_{t_k}^{t}\sum_i \partial_iu\bigl(T-t,\tilde{X}(t)\bigr){\rm Tr}\Bigl(S_{\Delta t}\sigma(X_k)^2 S_{\Delta t}D^2G_i(\tilde{X}(s))\Bigr)dsdt\\
b_k^{2,2}&=\E\int_{t_{k}}^{t_{k+1}}\int_{t_k}^{t}\sum_i \partial_iu\bigl(T-t,\tilde{X}(t)\bigr)\langle S_{\Delta t}AX_k,DG_i(\tilde{X}(s))\rangle dsdt\\
b_k^{2,3}&=\E\int_{t_{k}}^{t_{k+1}}\int_{t_k}^{t}\sum_i \partial_iu\bigl(T-t,\tilde{X}(t)\bigr)\langle S_{\Delta t}F(X_k),DG_i(\tilde{X}(s))\rangle dsdt\\
b_k^{2,4}&=\E\int_{t_{k}}^{t_{k+1}}\sum_i \partial_iu\bigl(T-t,\tilde{X}(t)\bigr)\int_{t_k}^{t}\langle DG_i(\tilde{X}(s)),S_{\Delta t}\sigma(X_k)dW(s)\rangle dt,
\end{align*}
thanks to It\^o formula, and using $\sigma(\cdot)^\star=\sigma(\cdot)$.

\subsubsection{Treatment of $b_k^1$}

We directly apply Theorem~\ref{theo:D1}, with $\beta=1-\kappa$, and thanks to~\eqref{eq:error_S_Deltat}, Property~\ref{ass:F}, and inequality~\eqref{eq:norm_AB}, we get
\begin{align*}
|b_k^1|&\le C(1+|x|_{L^{\max(p,2q)}})^{K+1} \int_{t_{k}}^{t_{k+1}}\frac{1}{(T-t)^{1-\kappa}}\big|(-A)^{\frac12+2\kappa}S_{\Delta t}(-A)^{-\frac12-\kappa}G(X_k)  \big|_{\mathcal{L}(L^q)}dt\\
&\le C\Delta t^{\frac12-2\kappa}(1+|x|_{L^{\max(p,2q)}})^{K+1} \int_{t_{k}}^{t_{k+1}}\frac{1}{(T-t)^{1-\kappa}}dt,
\end{align*}
As a consequence,
\begin{equation}\label{e_alpha}
\sum_{k=1}^{N-1}|b_k^1|\le C\Delta t^{\frac12-2\kappa}(1+|x|_{L^{\max(p,2q)}})^{K+1}.
\end{equation}

\subsubsection{Treatment of $b_k^2$}

\subsubsection*{Control of $b_{k}^{2,1}$}

To treat the term $b_k^{2,1}$, we expand the trace, using the orthonormal system $\bigl(e_i\bigr)_{i\in\N^\star}$, and with straightforward calculations we write
\begin{align*}
b_k^{2,1}&=\E\int_{t_{k}}^{t_{k+1}}\int_{t_k}^{t}\sum_i \partial_iu\bigl(T-t,\tilde{X}(t)\bigr){\rm Tr}\Bigl(S_{\Delta t}\sigma(X_k)^2 S_{\Delta t}D^2G_i(\tilde{X}(s))\Bigr)dsdt\\
&=\E\int_{t_{k}}^{t_{k+1}}\int_{t_k}^{t}\sum_{i,n} \partial_iu\bigl(T-t,\tilde{X}(t)\bigr)D^2G_i(\tilde{X}(s)).\bigl(S_{\Delta t}e_n,S_{\Delta t}\sigma(X_k)^2e_n\bigr)dsdt\\
&=\E\int_{t_{k}}^{t_{k+1}}\int_{t_k}^{t}\sum_{n} \langle Du\bigl(T-t,\tilde{X}(t)\bigr),D^2G(\tilde{X}(s)).\bigl(S_{\Delta t}e_n,S_{\Delta t}\sigma(X_k)^2e_n\bigr)\rangle dsdt
\end{align*}

Using Theorem~\ref{theo:D1}, with $\beta=\frac12+\kappa$, combined with inequality~\eqref{eq:norm_AB}, and Properties~\ref{ass:A} and~\ref{ass:F}, we get
\begin{align*}
|b_{k}^{2,1}|&\le C(1+|x|_{L^{\max(p,2q)}})^{K+1}\E\int_{t_{k}}^{t_{k+1}}\frac{1}{(T-t)^{\frac12+\kappa}}\\
&\hspace{5cm}\int_{t_k}^{t}
\sum_{j\in\left\{1,2\right\},n\in\N^\star}\bigl(\E|D^2F_j(\tilde{X}(s)).\bigl(S_{\Delta t}e_n,S_{\Delta t}\sigma(X_k)^2e_n\bigr)|_{L^{q}}^2\bigr)^{\frac12+\kappa}dsdt\\
&\le C\Delta t(1+|x|_{L^{\max(p,2q)}})^{K+1}\sum_n |S_{\Delta t}e_n|_{L^{2q}}\int_{t_k}^{t_{k+1}}\frac{1}{(T-t)^{\frac12+\kappa}}dt\\
&\le C\Delta t^{\frac12-\kappa}(1+|x|_{L^{\max(p,2q)}})^{K+1} \sum_{n}\frac{1}{\lambda_n^{\frac12+\kappa}}\int_{t_k}^{t_{k+1}}\frac{1}{(T-t)^{\frac12+\kappa}}dt.
\end{align*}

\subsubsection*{Control of $b_{k}^{2,2}$}

As for the term $a_{k}^{1,3}$, we need to further decompose 
\begin{align*}
b_k^{2,2}&=\E\int_{t_{k}}^{t_{k+1}}\int_{t_k}^{t}\langle Du\bigl(T-t,\tilde{X}(t)\bigr),DG(\tilde{X}(s)).\bigl(S_{\Delta t}AX_k\bigr)\rangle dsdt\\
&=b_k^{2,2,1}+b_k^{2,2,2}+b_k^{2,2,3},
\end{align*}
where, using~\eqref{eq:mild_sol_num},
\begin{align*}
b_k^{2,2,1}&=\E\int_{t_{k}}^{t_{k+1}}\int_{t_k}^{t}\langle Du\bigl(T-t,\tilde{X}(t)\bigr),DG(\tilde{X}(s)).\bigl(S_{\Delta t}^{k+1}Ax\bigr)\rangle dsdt\\
b_k^{2,2,2}&=\E\int_{t_{k}}^{t_{k+1}}\int_{t_k}^{t}\langle Du\bigl(T-t,\tilde{X}(t)\bigr),DG(\tilde{X}(s)).\bigl(\Delta t\sum_{\ell=0}^{k-1}AS_{\Delta t}^{k-\ell+1}G(X_{\ell})\bigr)\rangle dsdt\\
b_k^{2,2,3}&=\E\int_{t_{k}}^{t_{k+1}}\int_{t_k}^{t}\langle Du\bigl(T-t,\tilde{X}(t)\bigr),DG(\tilde{X}(s)).\bigl(\sum_{\ell=0}^{k-1}AS_{\Delta t}^{k-\ell+1}\sigma(X_\ell)\Delta W_\ell\bigr)\rangle dsdt.
\end{align*}

The terms $b_k^{2,2,1}$ and $b_{k}^{2,2,2}$ are estimated using Theorem~\ref{theo:D1} in a straightforward way..

On the one hand, thanks to \eqref{eq:norm_AB},
\begin{align*}
|b_k^{2,2,1}|&\le C\Delta t\int_{t_k}^{t_{k+1}}\frac{1}{(T-t)^{\frac12+\kappa}}dt (1+|x|_{L^{\max(p,2q)}})^{K+1} \big|S_{\Delta t}^{k+1}Ax|_{L^q}dt\\
&\le C\frac{\Delta t^{1-\kappa}}{t_k^{1-\kappa}}(1+|x|_{L^{\max(p,2q)}})^{K+1}\int_{t_k}^{t_{k+1}}\frac{1}{(T-t)^{\frac12+\kappa}}dt.
\end{align*}

On the other hand,
\begin{align*}
|b_k^{2,2,2}|&\le C\Delta t \int_{t_k}^{t_{k+1}}\frac{1}{(T-t)^{\frac12+\kappa}}dt(1+|x|_{L^{\max(p,2q)}})^{K+1} \Delta t\sum_{\ell=0}^{k-1}\bigl(\big|AS_{\Delta t}^{k-\ell+1}\big|_{\mathcal{L}(L^q)}+
\big|AS_{\Delta t}^{k-\ell+1}B\big|_{\mathcal{L}(L^q)}\bigr)\\
&\le C\Delta t^{\frac12-2\kappa}\int_{t_k}^{t_{k+1}}\frac{1}{(T-t)^{\frac12+\kappa}}dt (1+|x|_{L^{\max(p,2q)}})^{K+1} \bigl(\Delta t\sum_{\ell=0}^{k-1}\frac{1}{t_{k-\ell}^{1-\kappa}}\bigr).
\end{align*}

It remains to treat $b_k^{2,2,3}$. Writing
\[\sum_{\ell=0}^{k-1}AS_{\Delta t}^{k-\ell+1}\sigma(X_\ell)\Delta W_\ell=\int_{0}^{t_k}AS_{\Delta t}^{k-\ell_r+1}\sigma(X_{\ell_r})dW(r)
\]
as a stochastic integral, and subdividing the interval $[0,t_k]=[0,t_{k-1}]\cup [t_{k-1},t_k]$, we have the decomposition
\begin{align*}
b_{k}^{2,2,3}&=\E\int_{t_{k}}^{t_{k+1}}\int_{t_k}^{t}\langle Du\bigl(T-t,\tilde{X}(t)\bigr),DG(\tilde{X}(s)).\bigl(\int_{0}^{t_k}AS_{\Delta t}^{k-\ell_r+1}\sigma(X_{\ell_r})dW(r)\bigr)\rangle dsdt\\
&=\E\int_{t_{k}}^{t_{k+1}}\int_{t_k}^{t}\langle Du\bigl(T-t,\tilde{X}(t)\bigr),DG(\tilde{X}(s)).\bigl(\int_{0}^{t_{k-1}}AS_{\Delta t}^{k-\ell_r+1}\sigma(X_{\ell_r})dW(r)\bigr)\rangle dsdt\\
&~+\E\int_{t_{k}}^{t_{k+1}}\int_{t_k}^{t}\langle Du\bigl(T-t,\tilde{X}(t)\bigr),DG(\tilde{X}(s)).\bigl(\int_{t_{k-1}}^{t_k}AS_{\Delta t}^{k-\ell_r+1}\sigma(X_{\ell_r})dW(r)\bigr)\rangle dsdt\\
&=b_{k}^{2,2,3,1}+b_{k}^{2,2,3,2}.
\end{align*}
Using {the Malliavin calculus duality formula,}
\begin{align*}
&b_{k}^{2,2,3,1}\\
&=\E\int_{t_{k}}^{t_{k+1}}\int_{t_k}^{t}\langle Du\bigl(T-t,\tilde{X}(t)\bigr),DG(\tilde{X}(s)).\bigl(\int_{0}^{t_{k-1}}AS_{\Delta t}^{k-\ell_r+1}\sigma(X_{\ell_r})dW(r)\bigr)\rangle dsdt\\
&=\E\int_{t_{k}}^{t_{k+1}}\int_{t_k}^{t}\sum_n\langle Du\bigl(T-t,\tilde{X}(t)\bigr),DG(\tilde{X}(s)).\bigl(\int_{0}^{t_{k-1}}AS_{\Delta t}^{k-\ell_r+1}\sigma(X_{\ell_r})e_n d\beta_n(r)\bigr)\rangle dsdt\\
&=\E\int_{t_{k}}^{t_{k+1}}\int_{t_k}^{t}\sum_{n,m}\langle Du\bigl(T-t,\tilde{X}(t)\bigr),DG(\tilde{X}(s)).e_m\rangle \int_{0}^{t_{k-1}}\langle AS_{\Delta t}^{k-\ell_r+1}\sigma(X_{\ell_r})e_n,e_m\rangle d\beta_n(r) dsdt\\
&=\E\int_{t_{k}}^{t_{k+1}}\int_{t_k}^{t}\int_{0}^{t_{k-1}}
\sum_{n,m}\langle AS_{\Delta t}^{k-\ell_r+1}\sigma(X_{\ell_r})e_n,e_m\rangle D^2u(T-t,\tilde{X}(t)).\bigl(\D_r\tilde{X}(t)e_n,DG(\tilde{X}(s)).e_m\bigr)drdsdt\\
&~+\E\int_{t_{k}}^{t_{k+1}}\int_{t_k}^{t}\int_{0}^{t_{k-1}}
\sum_{n,m}\langle AS_{\Delta t}^{k-\ell_r+1}\sigma(X_{\ell_r})e_n,e_m\rangle \langle Du(T-t,\tilde{X}(t)),D^2G(\tilde{X}(s)).\bigl(e_m,\D_r\tilde{X}(s)e_n\bigr)\rangle drdsdt\\
&=\E\int_{t_{k}}^{t_{k+1}}\int_{t_k}^{t}\int_{0}^{t_{k-1}}\sum_nD^2u(T-t,\tilde{X}(t)).\bigl(U(t,r)S_{\Delta t}e_n,DG(\tilde{X}(s))AS_{\Delta t}^{k-\ell_r+1}\sigma(X_{\ell_r})^2e_n\bigr)drdtds\\
&~+\E\int_{t_{k}}^{t_{k+1}}\int_{t_k}^{t}\int_{0}^{t_{k-1}}
\sum_n \langle Du(T-t,\tilde{X}(t)),D^2G(\tilde{X}(s)).\bigl(AS_{\Delta t}^{k-\ell_r+1}\sigma(X_{\ell_r})^2e_n,U(s,r)S_{\Delta t}e_n\bigr)\rangle drdsdt,
\end{align*}
where we have used the identities $\D_r\tilde{X}(t)=U(t,r)S_{\Delta t}\sigma(X_{\ell_r})$ and $\D_r\tilde{X}(s)=U(s,r)S_{\Delta t}\sigma(X_{\ell_r})$ for $r<t_{k-1}\le s\le t\le t_k$. 

To estimate $b_{k}^{2,2,3,1}$ we first write
\begin{align*}
\E|(-A)^{-\frac12+\kappa}DG(\tilde{X}(s))AS_{\Delta t}^{k-\ell_r+1}\sigma(X_{\ell_r})^2e_n|_{L^{4q}}^2&\le c\E|F'_1(\tilde{X}(s))AS_{\Delta t}^{k-\ell_r+1}\sigma(X_{\ell_r})^2e_n|_{L^{4q}}^2\\
&+ c\E |(-A)^{2\kappa}F'_2(\tilde{X}(s))AS_{\Delta t}^{k-\ell_r+1}\sigma(X_{\ell_r})^2e_n|_{L^{4q}}^2.
\end{align*}
The treatment of the first term is straightforward, with upper bound given by $c\big|AS_{\Delta t}^{k-\ell_r+1}\big|_{\mathcal{L}(L^{4q})}$. For the second term, we use \eqref{eq:Sobolev-Lipschitz} and \eqref{eq:product_2}:
\begin{align*}
\E|(-A)^{2\kappa}F'_2(\tilde{X}(s))AS_{\Delta t}^{k-\ell_r+1}\sigma(X_{\ell_r})^2e_n|_{L^{4q}}^2&\le c \E\left(1+|(-A)^{3\kappa}\tilde X(s)|_{L^{8q}})|(-A)^{1+3\kappa}S_{\Delta t}^{k-\ell_r+1}\sigma(X_{\ell_r})^2e_n|_{L^{8q}}\right)^2\\
& \le c \E\left((1+|(-A)^{3\kappa}\tilde X(s)|_{L^{8q}})|(-A)^{1+3\kappa}S_{\Delta t}^{k-\ell_r+1}|_{\mathcal L(L^{8q)}}\right)^2.
\end{align*}
Therefore, using \eqref{eq:tilde_num_alpha}, we obtain:
$$
\left(\E|(-A)^{-\frac12+\kappa}DG(\tilde{X}(s))AS_{\Delta t}^{k-\ell_r+1}\sigma(X_{\ell_r})^2e_n|_{L^{4q}}^2\right)^{\frac12}
\le  c(1+\frac{1}{t_k^{3\kappa}}|x|_{L^{8q}}) \Delta t^{-4\kappa}\frac1{ t_{k-{\ell_r}}^{1-\kappa}}.
$$
Moreover by Lemma~\ref{lem:U_num}:
$$
\bigl(\E|(-A)^{-\frac12+\kappa}U(t,r)S_{\Delta t}e_n|_{L^{4q}}^2\bigr)^{\frac12}\le c \bigl(|(-A)^{-\frac12+\kappa}S_{\Delta t}e_n|_{L^{4q}}+\Delta t^{\frac12-\kappa}|S_{\Delta t}e_n|_{L^{4q}}\bigr),
$$
and using  Theorem \ref{theo:D2} we get 
\begin{align*}
&\Big|\E D^2u(T-t,\tilde{X}(t)).\bigl(U(t,r)S_{\Delta t}e_n,DG(\tilde{X}(s))AS_{\Delta t}^{k-\ell_r+1}\sigma(X_{\ell_r})^2e_n\bigr)\Big|\\
&\le  c (1+|x|_{L^{\max(p,2q)}})^{K+1} \frac{1}{(T-t)^{1-2\kappa}} (1+t_k^{-3\kappa}|x|_{L^{8q}}) \Delta t^{-4\kappa} \frac1{t_{k-{\ell_r}}^{1-\kappa}}\\
&\times \bigl(|(-A)^{-\frac12+\kappa}S_{\Delta t}e_n|_{L^{4q}}+\Delta t^{\frac12-\kappa}|S_{\Delta t}e_n|_{L^{4q}}\bigr)\\
&\le c (1+|x|_{L^{\max(p,2q)}})^{K+1} \frac{1}{(T-t)^{1-2\kappa}} (1+t_k^{-3\kappa}|x|_{L^{8q}}) \Delta t^{-6\kappa} \frac1{t_{k-{\ell_r}}^{1-\kappa}} \frac1{\lambda_n^{\frac12+\kappa}}.
\end{align*}
The second term of $b_{k}^{2,2,3,1}$ is treated similarly thanks to Theorem~\ref{theo:D1} (with $\beta=\frac12+\kappa$):
\begin{align*}
& \Big|\E\langle Du(T-t,\tilde{X}(t)),D^2G(\tilde{X}(s)).\bigl(AS_{\Delta t}^{k-{\ell_r}_r+1}\sigma(X_{\ell_r})^2e_n,U(s,r)S_{\Delta t}e_n\bigr)\rangle\Big|\\
 &\le c (1+|x|_{L^{\max(p,2q)}})^{K+1}\frac1{(T-t)^{\frac12+\kappa}} \E\bigl|(AS_{\Delta t}^{k-\ell_r}\sigma(X_{\ell_r})^2e_n)(U(s,r)S_{\Delta t}e_n)\bigr|_{L^q}\\
  &\le c(1+|x|_{L^{\max(p,2q)}})^{K+1} \frac1{(T-t)^{\frac12+\kappa}} \bigl|AS_{\Delta t}^{k-\ell_r}\bigr|_{\mathcal L(L^{2q})}\E\bigl|U(s,r)S_{\Delta t}e_n\bigr|_{L^{2q}}\\
  &\le c (1+|x|_{L^{\max(p,2q)}})^{K+1}\frac1{(T-t)^{\frac12+\kappa}}\Delta t^{-\kappa}\frac1{t_{k-{\ell_r}}^{1-\kappa}}\bigl|S_{\Delta t}e_n\bigr|_{L^{2q}}\\
  &\le c (1+|x|_{L^{\max(p,2q)}})^{K+1}\frac1{(T-t)^{\frac12+\kappa}}\Delta t^{-\frac12-2\kappa}\frac1{t_{k-{\ell_r}}^{1-\kappa}}\frac1{\lambda_n^{\frac12+\kappa}}.
\end{align*}
We deduce
\begin{align*}
|b_{k}^{2,2,3,1}|&\le C(1+|x|_{L^{\max(p,8q)}})^{K+2}\int_{t_{k}}^{t_{k+1}}\int_{t_k}^{t}\int_{0}^{t_{k-1}}\sum_n \frac{1}{(T-t)^{1-2\kappa}} (1+t_k^{-3\kappa}|x|_{L^{8q}}) \Delta t^{-\frac12-\kappa} \frac1{t_{k-{\ell_r}}^{1-\kappa}} \frac1{\lambda_n^{\frac12+\kappa}}drdsdt\\
&\le C\Delta t^{\frac12-4\kappa}(1+|x|_{L^{\max(p,8q)}})^{K+3}\int_{t_{k}}^{t_{k+1}}\frac1{t^{3\kappa}}
\frac{1}{(T-t)^{1-\kappa}}\bigl(\Delta t\sum_{{\ell}=0}^{k-1}\frac{1}{t_{k-\ell}^{1-\kappa}}\bigr)\bigl(\sum_n\frac{1}{\lambda_{n}^{\frac12+\kappa}}\bigr)dt\\
&\le C\Delta t^{\frac12-4\kappa}(1+|x|_{L^{\max(p,8q)}})^{K+3}\int_{t_{k}}^{t_{k+1}}
\frac1{t^{3\kappa}}\frac{1}{(T-t)^{1-\kappa}}dt,
\end{align*}
using similar arguments to the control of $a_{k}^{1,3}$.

To treat the remaining term $b_{k}^{2,2,3,2}$, we again use {the Malliavin calculus duality formula}. With the same arguments as for $b_k^{2,2,3,1}$, we get the identity
\begin{align*}
b_k^{2,2,3,2}&=\E\int_{t_{k}}^{t_{k+1}}\int_{t_k}^{t}\langle Du\bigl(T-t,\tilde{X}(t),DG(\tilde{X}(s)).\bigl(\int_{t_{k-1}}^{t_k}AS_{\Delta t}^{2}\sigma(X_{k-1})dW(r)\bigr)\rangle dsdt\\
&=\E\int_{t_{k}}^{t_{k+1}}\int_{t_k}^{t}\int_{t_{k-1}}^{t_{k}}\sum_nD^2u(T-t,\tilde{X}(t)).\bigl(U(t,r)S_{\Delta t}e_n,DG(\tilde{X}(s))AS_{\Delta t}^{2}\sigma(X_{k-1})^2e_n\bigr)drdtds\\
&~+\E\int_{t_{k}}^{t_{k+1}}\int_{t_k}^{t}\int_{t_{k-1}}^{t_k}
\sum_n \langle Du(T-t,\tilde{X}(t)),D^2G(\tilde{X}(s)).\bigl(AS_{\Delta t}^{2}\sigma(X_{k-1})^2e_n,U(s,r)S_{\Delta t}e_n\bigr)\rangle drdsdt,
\end{align*}
and we get, using Theorems~\ref{theo:D1} and~\ref{theo:D2}, and Lemma~\ref{lem:U_num},
\begin{align*}
|b_{k}^{2,2,3,2}|&\le C\Delta t^2(1+|x|_{L^{\max(p,2q)}})^{K+1}\int_{t_k}^{t_{k+1}}\frac{1}{(T-t)^{1-\kappa}}dt\\
&\hspace{4cm}\sum_n \left(|(-A)^{-\frac12+\kappa}S_{\Delta t}e_n|_{L^{4q}} +\Delta t^{\frac12-\kappa}|S_{\Delta t}e_n|_{L^{4q}} \right)|(-A)^{ 1+\kappa}S_{\Delta t}^2|_{\mathcal{L}(L^{4q})}\\
&\le C\Delta t^{\frac12-3\kappa}(1+|x|_{L^{\max(p,2q)}})^{K+1}\int_{t_k}^{t_{k+1}}\frac{1}{(T-t)^{1-\kappa}}dt \sum_{n}\frac{1}{\lambda_n^{\frac12+\kappa}}.
\end{align*}

\subsubsection*{Control of $b_{k}^{2,3}$}

The treatment of this term is straightforward, using Theorem~\ref{theo:D1}, with $\beta=\frac12+\kappa$, and Property~\ref{ass:F}. Indeed,
\begin{align*}
|b_k^{2,3}|&=\Big|\E\int_{t_{k}}^{t_{k+1}}\int_{t_k}^{t}\sum_i \partial_iu\bigl(T-t,\tilde{X}(t)\bigr)\langle S_{\Delta t}G(X_k),DG_i(\tilde{X}(s))\rangle dsdt\Big|\\
&=\Big|\E\int_{t_{k}}^{t_{k+1}}\int_{t_k}^{t}\langle Du\bigl(T-t,\tilde{X}(t)\bigr),DG(\tilde{X}(s)).\bigl(S_{\Delta t}G(X_k)\bigr)\rangle dsdt\Big|\\
&\le C(1+|x|_{L^{\max(p,2q)}})^{K+1}\int_{t_{k}}^{t_{k+1}}\frac{1}{(T-t)^{\frac12+\kappa}}\int_{t_k}^{t}\bigl(\E\big|(-A)^{\frac12+\kappa}S_{\Delta t}(-A)^{-\frac12-\kappa}G(X_k)\big|_{L^q}^2\bigr)^{\frac12}dsdt\\
&\le C\Delta t^{\frac12-\kappa}(1+|x|_{L^{\max(p,2q)}})^{K+1}\int_{t_{k}}^{t_{k+1}}\frac{1}{(T-t)^{\frac12+\kappa}}dt,
\end{align*}
using $\big|S_{\Delta t}B\big|_{\mathcal{L}(L^q)}\le C\Delta t^{-\frac12-\kappa}$ thanks to~\eqref{eq:norm_AB} and Lemma~\ref{lem:S_1}.

\subsubsection*{Control of $b_{k}^{2,4}$}

The term $b_k^{2,4}$ involves a stochastic integral. Similarly to the treatment of the term $a_k^{3,2}$, we use {the Malliavin calculus duality formula}, and the identity $\D_s\tilde{X}(t)=S_{\Delta t}\sigma(X_k)$ for $t_k\le s<t\le t_{k+1}$. We obtain
\begin{align*}
|b_k^{2,4}|&=\Big|\E\int_{t_{k}}^{t_{k+1}}\sum_i \partial_iu\bigl(T-t,\tilde{X}(t)\bigr)\langle \int_{t_k}^{t}\langle DG_i(\tilde{X}(s)),S_{\Delta t}\sigma(X_k)dW(s)\rangle dt\Big|\\
&=\Big|\E\int_{t_{k}}^{t_{k+1}}\int_{t_k}^{t}{\rm Tr}\Bigl(\bigl(\D_s\tilde{X}(t)\bigr)^\star D^2u\bigl(T-t,\tilde{X}(t)\bigr)DG(\tilde{X}(s))S_{\Delta t}\sigma(X_k)\Bigr)dsdt\Big|\\
&=\Big|\E\int_{t_{k}}^{t_{k+1}}\int_{t_k}^{t}\sum_n D^2u\bigl(T-t,\tilde{X}(t)\bigr).\bigl(S_{\Delta t}e_n,DG(\tilde{X}(s))S_{\Delta t}\sigma(X_k)^2e_n\bigr)dsdt\Big|\\
&\le C(1+|x|_{L^{\max(p,2q)}})^{K+1} \int_{t_{k}}^{t_{k+1}}\int_{t_k}^{t}\sum_n\frac1{(T-t)^{1-2\kappa}} |(-A)^{-\frac12+\kappa}S_{\Delta t}e_n|_{L^{4q}}\\
&\hspace{5cm}\bigl(\E\big|(-A)^{-\frac12+\kappa}DG(\tilde{X}(s))S_{\Delta t}\sigma(X_k)^2e_n\big|_{L^{4q}}^2\bigr)^{\frac12}dsdt\\
&\le C\Delta t^{\frac12-4\kappa}\sum_n\frac{1}{\lambda_n^{\frac12+\kappa}}  (1+|x|_{L^{\max(p,2q)}})^{K+1} (1+\frac{1}{t_k^{3\kappa}}|x|_{L^{8q}})\int_{t_k}^{t_{k+1}}\frac1{(T-t)^{1-2\kappa}}dt\\
&\le C\Delta t^{\frac12-4\kappa}\sum_n\frac{1}{\lambda_n^{\frac12+\kappa}}  (1+|x|_{L^{\max(p,8q)}})^{K+2}\int_{t_k}^{t_{k+1}}
\bigl(1+\frac1{t^{3\kappa}}\bigr)\frac1{(T-t)^{1-2\kappa}}dt,
\end{align*}
thanks to similar arguments as for the treatment of $b_k^{2,2,3}$.

\subsubsection*{Conclusion}
Gathering the estimates on $b_{k}^{2,1}$, $b_{k}^{2,2}$, $b_k^{2,3}$ and $b_{k}^{2,4}$, and summing for $k\in\left\{1,\ldots,N-1\right\}$, we obtain
\begin{equation}\label{eq:num_b^2}
\sum_{k=1}^{N}\big|b_{k}^2\big|\le C\Delta^{\frac12-\kappa}(1+|x|_{L^{\max(p,8q)}})^{K+2}\int_{0}^{T}
\bigl(1+\frac1{t^{3\kappa}}\bigr)\bigl(1+\frac{1}{(T-t)^{1-\kappa}}\bigr)dt.
\end{equation}

\subsection{Control of $c_k$}

\subsubsection{Decompositions}

For each $k\in\left\{1,\ldots,N-1\right\}$, $c_k$ is decomposed into the following terms:
\begin{equation}\label{e80}
c_k=c_k^1+c_k^2+c_k^3=c_k^1+c_k^2+\bigl(c_{k}^{3,\calA}+c_{k}^{3,\calB}+c_{k}^{3,\calC}+c_k^{3,\calD}\bigr),
\end{equation}
where (using the symmetry of $D^2u$)
\begin{align*}
c_k^1&=\frac{1}{2}\E\int_{t_k}^{t_{k+1}}{\rm Tr}\Bigl((I-S_{\Delta t})\sigma\bigl(\tilde{X}(t)\bigr)^2 (I-S_{\Delta t}) D^2u\bigl(T-t,\tilde{X}(t)\bigr)\Bigr)dt,\\
c_k^2&=\E\int_{t_k}^{t_{k+1}}{\rm Tr}\Bigl(S_{\Delta t} \sigma\bigl(\tilde{X}(t)\bigr)^2 (I-S_{\Delta t}) D^2u\bigl(T-t,\tilde{X}(t)\bigr)\Bigr)dt,\\
c_k^3&=\frac{1}{2}\E\int_{t_k}^{t_{k+1}}{\rm Tr}\Bigl( S_{\Delta t}\bigl[\sigma\bigl(\tilde{X}(t)\bigr)^2-\sigma(X_k)^2\bigr]S_{\Delta t} D^2u\bigl(T-t,\tilde{X}(t)\bigr)\Bigr)dt.
\end{align*}
In addition, $c_k^3$ is further decomposed as follows: for $\Lambda\in\left\{\calA,\calB,\calC,\calD\right\}$,
\begin{equation}\label{eq:decomp_ck3}
\begin{aligned}
c_k^{3,\Lambda}&=\frac{1}{2}\E\int_{t_k}^{t_{k+1}}{\rm Tr}\Bigl( S_{\Delta t} \Lambda S_{\Delta t} D^2u\bigl(T-t,\tilde{X}(t)\bigr)\Bigr)dt\\
&=\frac{1}{2}\E\int_{t_k}^{t_{k+1}}\sum_n D^2u(T-t,\tilde{X}(t)).\bigl(S_{\Delta t}\Lambda e_n,S_{\Delta t}e_n\bigr) dt,
\end{aligned}
\end{equation}
with the linear operators $\calA,\calB,\calC,\calD$ obtained by applying It\^o formula:
\begin{equation*}
\langle \bigl[\sigma\bigl(\tilde{X}(t)\bigr)^2-\sigma(X_k)^2\bigr]h_1,h_2 \rangle=\sum_{\Lambda\in\left\{\calA,\calB,\calC,\calD\right\}}\langle \Lambda h_1,h_2\rangle,
\end{equation*}
with
\begin{align*}
\langle \calA h_1,h_2\rangle&=\frac{1}{2}\int_{t_k}^{t}{\rm Tr}\Bigl(S_{\Delta t}\sigma(X_k)^2 S_{\Delta t} D^2\sigma_{h_1,h_2}^2(\tilde{X}(s))\Bigr)ds\\
\langle \calB h_1,h_2\rangle&=\int_{t_k}^{t}\langle S_{\Delta t}AX_k,D\sigma_{h_1,h_2}^2(\tilde{X}(s))\rangle ds\\
\langle \calC h_1,h_2\rangle&=\int_{t_k}^{t}\langle S_{\Delta t}G(X_k),D\sigma_{h_1,h_2}^2(\tilde{X}(s))\rangle   ds\\
\langle \calD h_1,h_2\rangle&=\int_{t_k}^{t} \langle D\sigma_{h_1,h_2}^2(\tilde{X}(s)),S_{\Delta t}\sigma(X_k)dW(s)\rangle ds,
\end{align*}
using the notation $\sigma_{h_1,h_2}^2=\langle \sigma(\cdot)^2 h_1,h_2\rangle$.

For future reference, note the following identity:
\begin{equation}\label{eq:dsigma}
\langle D\sigma_{e_n,e_m}^2(x),h\rangle=\int_{(0,1)}(\sigma^2)'\bigl(x(\xi)\bigr)e_n(\xi)e_m(\xi)h(\xi)d\xi=\langle D\sigma_{e_n,h}^2(x),e_m\rangle.
\end{equation}

\subsubsection{Treatment of $c_k^1$}

Since Theorem~\ref{theo:D2} is restricted to $\beta,\gamma<\frac12$, this term needs some work:
\begin{align*}
|c_k^1|&=\frac{1}{2}\left|\E\int_{t_k}^{t_{k+1}}{\rm Tr}\Bigl( (I-S_{\Delta t}) D^2u\bigl(T-t,\tilde{X}(t)\bigr)(I-S_{\Delta t})\sigma\bigl(\tilde{X}(t)\bigr)^2\Bigr)dt\right|\\
&\le \frac{1}{2}\E\int_{t_k}^{t_{k+1}}\sum_{n}\left| D^2u(T-t,\tilde{X}(t)).\bigl((I-S_{\Delta t})e_n,(I-S_{\Delta t}) \sigma(\tilde{X}(t))^2e_n\bigr)\right|dt\\
&\le C\bigl(1+|x|_{L^{\max(p,2q)}}\bigr)^{K+1}\sum_n \int_{t_k}^{t_{k+1}}\frac{1}{(T-t)^{1-2\kappa}}\big|(-A)^{-\frac12\kappa}(I-S_{\Delta t})e_n\big|_{L^{4q}}\\
&\hspace{7cm} \bigl(\E\big|(-A)^{-\frac12+\kappa}(I-S_{\Delta t})\sigma\bigl(\tilde{X}(t)\bigr)^2e_n\big|_{L^{4q}}^2\bigr)^{\frac12}dt\\
&\le C\bigl(1+|x|_{L^{\max(p,2q)}}\bigr)^{K+1}\Delta t^{\frac12-4\kappa}\sum_n \lambda_{n}^{-\frac12+\kappa} \int_{t_k}^{t_{k+1}}\frac{1}{(T-t)^{1-2\kappa}}\bigl(\E\big|(-A)^{-3\kappa}\sigma\bigl(\tilde{X}(t)\bigr)^2e_n\big|_{L^{4q}}^2\bigr)^{\frac12}dt.
\end{align*}
Using~\eqref{eq:Sobolev-Lipschitz}, \eqref{eq:product_3}, we get
\[
\big|(-A)^{-2\kappa}\sigma\bigl(\tilde{X}(t)\bigr)^2e_n\big|_{L^{4q}}\le C|(-A)^{-2\kappa}e_n|_{L^{8q}}\big|(-A)^{4\kappa}\sigma\bigl(\tilde{X}(t)\bigr)^2\big|_{L^{8q}}\le C\lambda_{n}^{-2\kappa}\bigl(1+\big|(-A)^{5\kappa}\tilde{X}(t)\big|_{L^{8q}}\bigr).
\]
As a consequence, using \eqref{eq:tilde_num_alpha},
\begin{equation}\label{e_beta}
|c_k^1|\le C\bigl(1+|x|_{L^{\max(p,8q)}}\bigr)^{K+2}\Delta t^{\frac12-5\kappa}\int_{t_k}^{t_{k+1}}\bigl(1+\frac1{t^{5\kappa}}\bigr)\frac{1}{(T-t)^{1-\kappa}}dt.
\end{equation}

\subsubsection{Treatment of $c_k^2$}

The treatment of $c_k^2$ is straightforward, using Theorem~\ref{theo:D2}:
\begin{equation}\label{e_gamma}
\begin{aligned}
|c_k^2|&\le \Big|\E\int_{t_k}^{t_{k+1}}\sum_n D^2u\bigl(T-t,\tilde{X}(t)\bigr).\bigl(S_{\Delta t}e_n,(I-S_{\Delta t})\sigma(\tilde{X}(t))\bigr)dt\Big|\\
&\le C(1+|x|_{L^{\max(p,2q)}})^{K+1} \int_{t_k}^{t_{k+1}}\frac{1}{(T-t)^{1-2\kappa}}\sum_n |(-A)^{-\frac12+\kappa}S_{\Delta t}e_n|_{L^{4q}} \big|(-A)^{-\frac12+\kappa}(I-S_{\Delta t})\big|_{\mathcal{L}(L^{4q})}dt\\
&\le C(1+|x|_{L^{\max(p,2q)}})^{K+1} \Delta t^{\frac12-3\kappa} \int_{t_k}^{t_{k+1}}\frac{1}{(T-t)^{1-2\kappa}}dt\sum_n\frac{1}{\lambda_{n}^{\frac12+\kappa}}.
\end{aligned}
\end{equation}

\subsubsection{Treatment of $c_k^3$}

\subsubsection*{Control of $c_k^{3,\calA}$}

We proceed similarly, and applying Theorem~\ref{theo:D2} we obtain:
\begin{align*}
|c_k^{3,\calA}|&\le C(1+|x|_{L^{\max(p,2q)}})^{K+1} \int_{t_k}^{t_{k+1}}\frac{1}{(T-t)^{1-2\kappa}}\sum_n |(-A)^{-\frac12+\kappa}S_{\Delta t}e_n|_{L^{4q}}\bigl(\E|(-A)^{-\frac12+\kappa}S_{\Delta t}\calA e_n|_{L^{4q}}^2\bigr)^{\frac12}dt\\
&\le C\Delta t^{-2\kappa} C(1+|x|_{L^{\max(p,2q)}})^{K+1} \int_{t_k}^{t_{k+1}}\frac{1}{(T-t)^{1-2\kappa}}dt \sum_n\frac{1}{\lambda_n^{\frac12+\kappa}}\bigl(\E|\calA e_n|_{L^{4q}}^2\bigr)^{\frac12}.
\end{align*}

To have a control on $\bigl(\E|\calA e_n|_{L^{4q}}^2\bigr)^{\frac12}$, let any $h_1\in L^{4q}$ and $h_2\in L^{r}$, with $\frac{1}{4q}+\frac{1}{r}=1$; then
\begin{align*}
\langle \calA h_1,h_2\rangle&=\frac{1}{2}\int_{t_k}^{t}{\rm Tr}\Bigl(S_{\Delta t}\sigma(X_k)^2 S_{\Delta t} D^2\sigma_{h_1,h_2}^2(\tilde{X}(s))\Bigr)ds\\
&=\frac{1}{2}\int_{t_k}^{t_{k+1}}\sum_n  D^2\sigma_{h_1,h_2}^2(\tilde{X}(s)).\bigl(S_{\Delta t}\sigma(X_k)^2e_n,S_{\Delta t}e_n\bigr)ds\\
&\le C\Delta t|h_1|_{L^{4q}}|h_2|_{L^r}\sum_n|S_{\Delta t}\sigma(X_k)^2e_n|_{L^{\infty}}|S_{\Delta t}e_n|_{L^\infty}\\
&\le C\Delta t^{\frac12-\kappa}|h_1|_{L^{4q}}|h_2|_{L^r}\sum_n\frac{1}{\lambda_n^{\frac12+\kappa}}.
\end{align*}
Thus $\bigl(\E|\calA e_n|_{L^{4q}}^2\bigr)^{\frac12}\le C\Delta t^{\frac12-\kappa}$, and we obtain
\begin{equation*}
|c_k^{3,\calA}|\le C\Delta t^{\frac12-3\kappa} C(1+|x|_{L^{\max(p,2q)}})^{K+1} \int_{t_k}^{t_{k+1}}\frac{1}{(T-t)^{1-2\kappa}}dt \sum_n\frac{1}{\lambda_n^{\frac12+\kappa}}.
\end{equation*}

\subsubsection*{Control of $c_k^{3,\calB}$}

Like $a_{k}^{1,3}$ and $b_{k}^{2,2}$, the term $c_k^{3,\calB}$ contains a bad term and require a careful analysis. We introduce the decomposition $\calB=\calB_1+\calB_2+\calB_3$, and the associated terms $c_k^{3,\calB_1}$, $c_k^{3,\calB_2}$ and $c_k^{3,\calB_3}$, with
\begin{align*}
\langle \calB_1 h_1,h_2\rangle&=\int_{t_k}^{t}\langle S_{\Delta t}^{k+1}Ax,D\sigma_{h_1,h_2}^2(\tilde{X}(s))\rangle ds\\
\langle \calB_2 h_1,h_2\rangle&=\int_{t_k}^{t}\langle \int_{0}^{t_k}S_{\Delta t}^{k-\ell_r+1}AG(X_{\ell_r})dr,D\sigma_{h_1,h_2}^2(\tilde{X}(s))\rangle ds\\
\langle \calB_3 h_1,h_2\rangle&=\int_{t_k}^{t}\langle \int_{0}^{t_k}AS_{\Delta t}^{k-\ell_r+1}\sigma(X_{\ell_r})dW(r),D\sigma_{h_1,h_2}^2(\tilde{X}(s))\rangle ds.
\end{align*}

The terms $c_k^{3,\calB_1}$ and $c_k^{3,\calB_2}$ do not present difficulties, 
using~\eqref{eq:dsigma} and standard arguments. Indeed, for any $h_1\in L^{4q}$ and $h_2\in L^{r}$, with $\frac{1}{4q}+\frac{1}{r}=1$,
\[\langle \calB_1 h_1,h_2\rangle \le C\Delta t|h_1|_{L^{4q}}|h_2|_{L^r}|AS_{\Delta t}^{k+1}x|_{L^\infty}\le C|x|_{L^p}\frac{\Delta t^{\frac34-3\kappa}}{t_{k}^{1-\kappa}}|h_1|_{L^{4q}}|h_2|_{L^r},\]
using $|(-A)^{\kappa}S_{\Delta t}x|_{L^\infty}\le |(-A)^{\kappa}S_{\Delta t}x|_{W^{\frac12+\frac{\kappa}{2}}}\le C|(-A)^{\frac14+2\kappa}S_{\Delta t}x|_{L^2}\le \frac{C}{\Delta t^{\frac14+2\kappa}}|x|_{L^{p}}$.

Similarly, we have for $\calB_2$
\[
\langle \calB_2 h_1,h_2\rangle\le C\Delta t|h_1|_{L^{4q}}|h_2|_{L^r}\int_{0}^{t_k}\big|S_{\Delta t}^{k-\ell_r+1}AG(X_{\ell_r})\big|_{L^{\infty}}dr.
\]
Moreover, using Property~\ref{ass:F} and $G=F_1+BF_2$,
\begin{align*}
\big|S_{\Delta t}^{k-\ell_r+1}AG(X_{\ell_r})\big|_{L^{\infty}}&\le \big|(-A)^{\frac12+2\kappa}S_{\Delta t}|_{\mathcal{L}(L^{\frac1\kappa},L^\infty)}|(-A)^{1-\kappa}S_{\Delta t}^{k-\ell_r}\big|_{\mathcal{L}(L^{\frac1\kappa})}
\left(1+\big|(-A)^{-\frac12-\kappa}B\big|_{\mathcal{L}(L^{\frac1\kappa})}\right)\\
&\le C\big|(-A)^{\frac12+4\kappa}S_{\Delta t}|_{\mathcal{L}(L^{\frac1\kappa})}\\
&\le \Delta t^{-\frac12-4\kappa}t_{k-\ell_r}^{-1+\kappa},
\end{align*}
using~\eqref{eq:norm_AB}, as well as the following inequalities, which are consequences of Sobolev inequalities and of~\eqref{eq:Sobolev-domain}: for any $x\in L^{\frac1\kappa}$,
\[
\big|(-A)^{\frac12+\kappa}S_{\Delta t}x\big|_{L^\infty}\le C\big|(-A)^{\frac12+\kappa}S_{\Delta t}x\big|_{W^{2\kappa,\frac1\kappa}}\le C\big|(-A)^{\frac12+4\kappa}S_{\Delta t}x\big|_{L^{\frac1\kappa}}.
\]

We thus obtain
\[
|c_{k}^{3,\calB_1}|+|c_k^{3,\calB_2}|\le C(1+|x|_{L^{\max(p,2q)}})^{K+1} \Delta t^{\frac12-5\kappa}\int_{t_k}^{t_{k+1}}\bigl(1+\frac{1}{t^{1-\kappa}}\bigr)\frac{1}{(T-t)^{1-2\kappa}}dt \sum_n\frac{1}{\lambda_n^{\frac12+\kappa}}.
\]

Finally,  $c_k^{\calB_3}$ requires a Malliavin integration. First, we write
\begin{align*}
c_k^{3,\calB_3}&=\frac{1}{2}\E\int_{t_k}^{t_{k+1}}\sum_n D^2u(T-t,\tilde{X}(t)).\bigl(S_{\Delta t}\calB_3 e_n,S_{\Delta t}e_n\bigr) dt\\
&=\frac{1}{2}\E\int_{t_k}^{t_{k+1}}\sum_{n,m} \langle \calB_3 e_n,e_m\rangle D^2u(T-t,\tilde{X}(t)).\bigl(S_{\Delta t}e_m,S_{\Delta t}e_n\bigr) dt\\
&=\frac{1}{2}\E\iiint
\sum_{n,m,j} \langle AS_{\Delta t}^{k-\ell_r+1}\sigma(X_{\ell_r})e_j,D\sigma_{e_n,e_m}^2(\tilde{X}(s))\rangle D^2u(T-t,\tilde{X}(t)).\bigl(S_{\Delta t}e_m,S_{\Delta t}e_n\bigr) d\beta_j(r) dsdt,
\end{align*}
where for simplicity we use the notation $\iiint\bigl(\ldots\bigr)d\beta_n(r)dsdt
=\int_{t_{k}}^{t_{k+1}}\int_{t_{k}}^{t}\int_{0}^{t_k}\bigl(\ldots\bigr)d\beta_n(r)dsdt$.

Using {the Malliavin calculus duality formula}, for $t\in[t_k,t_{k+1}]$ and $s\in[t_k,t]$, then
\begin{align*}
\E\Bigl[&\sum_{m,j} \int_{0}^{t_k}\langle AS_{\Delta t}^{k-\ell_r+1}\sigma(X_{\ell_r})e_j,D\sigma_{e_n,e_m}^2(\tilde{X}(s))\rangle D^2u(T-t,\tilde{X}(t)).\bigl(S_{\Delta t}e_m,S_{\Delta t}e_n\bigr) d\beta_j(r) \Bigr]\\
&=\E\Bigl[\sum_{m,j}\int_{0}^{t_k} D^2\sigma_{e_n,e_m}(\tilde{X}(s)).\bigl(AS_{\Delta t}^{k-\ell_r+1}\sigma(X_{\ell_r})e_j,\D_r\tilde{X}(s)e_j\bigr)  D^2u(T-t,\tilde{X}(t)).\bigl(S_{\Delta t}e_m,S_{\Delta t}e_n\bigr)dr\Bigr]\\
&~+\E\Bigl[\sum_{m,j}\int_{0}^{t_k}\langle AS_{\Delta t}^{k-\ell_r+1}\sigma(X_{\ell_r})e_j,D\sigma_{e_n,e_m}^2(\tilde{X}(s))\rangle D^3u(T-t,\tilde{X}(t)).\bigl(S_{\Delta t}e_m,S_{\Delta t}e_n,\D_r\tilde{X}(t)e_j\bigr) dr\Bigr]\\
&=\E\Bigl[\sum_{m,j}\int_{0}^{t_k} D^2\sigma_{e_n,e_m}(\tilde{X}(s)).\bigl(AS_{\Delta t}^{k-\ell_r+1}e_j,\D_r\tilde{X}(s)\sigma(X_{\ell_r})e_j\bigr)  D^2u(T-t,\tilde{X}(t)).\bigl(S_{\Delta t}e_m,S_{\Delta t}e_n\bigr)dr\Bigr]\\
&~+\E\Bigl[\sum_j\int_{0}^{t_k} D^3u(T-t,\tilde{X}(t)).\bigl(S_{\Delta t}e_n,S_{\Delta t}D\sigma_{e_n,AS_{\Delta t}^{k-\ell_r+1}e_j}^2(\tilde{X}(s)),\D_r\tilde{X}(t)\sigma(X_{\ell_r})e_j\bigr) dr\Bigr],
\end{align*}
using the identity $\sigma(\cdot)^\star=\sigma(\cdot)$ for both lines, and~\eqref{eq:dsigma} for the second line.  Moreover, 
\begin{gather*}
\Big|D^2\sigma_{e_n,e_m}(\tilde{X}(s)).\bigl(AS_{\Delta t}^{k-\ell_r+1}e_j,\D_r\tilde{X}(s)\sigma(X_{\ell_r})e_j\bigr)\Big|\le C|e_n|_{L^\infty}|e_m|_{L^\infty}|AS_{\Delta t}^{k-\ell_r}e_j|_{L^2}|\D_r\tilde{X}(s)\sigma(X_{\ell_r})e_j|_{L^2}\\
\Big|D\sigma_{e_n,AS_{\Delta t}^{k-\ell_r+1}e_j}^2(\tilde{X}(s))\Big|_{L^{4q}}\le C|e_n|_{L^\infty}|AS_{\Delta t}^{k-\ell_r+1}e_j|_{L^{\infty}},
\end{gather*}
and using Lemma~\ref{lem:U_num} we get
\[E|\D_r\tilde{X}(s)\sigma(X_{\ell_r})e_j|_{L^2}^2\le C\E|\sigma(X_{\ell_r})e_j|_{L^2}\le C.\]

Then, using Theorem~\ref{theo:D2} and Proposition~\ref{theo:D3}, we obtain
\begin{align*}
|c_k^{3,\calB_3}|&\le C(1+|x|_{L^{\max(p,2q)}})^{K+1} \int_{t_{k}}^{t_{k+1}}\int_{t_{k}}^{t}\int_{0}^{t_k}
\frac{\Delta t^{-\frac12-6\kappa}}{t_{k-\ell_r}^{1-\kappa}}\frac{1}{(T-t)^{1-2\kappa}}drdsdt\\
&~+C(1+|x|_{L^{\max(p,2q)}})^{K+1}\int_{t_{k}}^{t_{k+1}}\int_{t_{k}}^{t}\int_{0}^{t_k}
\frac{\Delta t^{-\frac12-4\kappa}}{t_{k-\ell_r}^{1-\kappa}}\frac{1}{(T-t)^{\frac12-\kappa}}drdsdt\\
&\le C\Delta t^{\frac12-6\kappa}(1+|x|_{L^{\max(p,2q)}})^{K+1} \int_{t_k}^{t_{k+1}}\frac{1}{(T-t)^{1-2\kappa}}dt\\
&+C\Delta t^{\frac12-6\kappa}(1+|x|_{L^{\max(p,2q)}})^{K+1} \int_{t_k}^{t_{k+1}}\frac{1}{(T-t)^{\frac12-\kappa}}dt.
\end{align*}

\subsubsection*{Control of $c_k^{3,\calC}$}

Using~\eqref{eq:decomp_ck3}, similarly to $c_{k}^{3,\calA}$, we get
\begin{equation*}
|c_{k}^{3,\calC}|\le C\Delta t^{-2\kappa}\int_{t_k}^{t_{k+1}}\frac{1}{(T-t)^{1-2\kappa}}dt \sum_n\frac{1}{\lambda_n^{\frac12+\kappa}}\bigl(\E|\calC e_n|_{L^{4q}}^2\bigr)^{\frac12}.
\end{equation*}

For any $h_1\in L^{4q}$ and $h_2\in L^{r}$, with $\frac{1}{4q}+\frac{1}{r}=1$, we get
\begin{equation*}
|\langle \calC h_1,h_2\rangle|\le C\Delta t|h_1|_{L^{4q}}|h_2|_{L^r}|S_{\Delta t}G(X_k)|_{L^\infty}\le C\Delta t^{-1/2-\kappa}|h_1|_{L^{4q}}|h_2|_{L^r}.
\end{equation*}

We thus obtain
\begin{equation*}
|c_{k}^{3,\calC}|\le C\Delta t^{\frac12-2\kappa}\int_{t_k}^{t_{k+1}}\frac{1}{(T-t)^{1-2\kappa}}dt \sum_n\frac{1}{\lambda_n^{\frac12+\kappa}}.
\end{equation*}

\subsubsection*{Control of $c_k^{3,\calD}$}

Using~\eqref{eq:decomp_ck3}, the definition of $\calD$, and {the Malliavin calculus duality formula},
\begin{align*}
c_k^{3,\calD}&=\frac{1}{2}\E\int_{t_k}^{t_{k+1}}\sum_{n,m} \langle \calD e_n,e_m\rangle D^2u(T-t,\tilde{X}(t)).\bigl(S_{\Delta t}e_m,S_{\Delta t}e_n\bigr) dt\\
&=\frac{1}{2}\E\int_{t_k}^{t_{k+1}}\int_{t_k}^{t}\sum_{j,n,m}\langle D\sigma_{e_n,e_m}^2(\tilde{X}(s)),S_{\Delta t}\sigma(X_k)e_j\rangle d\beta_j(s) D^2u(T-t,\tilde{X}(t)).\bigl(S_{\Delta t}e_m,S_{\Delta t}e_n\bigr) dt\\
&=\frac{1}{2}\E\int_{t_k}^{t_{k+1}}\int_{t_k}^{t}\sum_{j,n,m} D^3u(T-t,\tilde{X}(t)).\bigl(\D_s\tilde{X}(t)e_j,S_{\Delta t}e_m,S_{\Delta t}e_n\bigr)\langle D\sigma_{e_n,e_m}^2(\tilde{X}(s)),S_{\Delta t}\sigma(X_k)e_j\rangle ds dt\\
&=\frac{1}{2}\E\int_{t_k}^{t_{k+1}}\int_{t_k}^{t}\sum_{j,n,m} D^3u(T-t,\tilde{X}(t)).\bigl(S_{\Delta t}\sigma(X_k)e_j,S_{\Delta t}e_m,S_{\Delta t}e_n\bigr)\langle D\sigma_{e_n,e_m}^2(\tilde{X}(s)),S_{\Delta t}\sigma(X_k)e_j\rangle ds dt\\
&=\frac{1}{2}\E\int_{t_k}^{t_{k+1}}\int_{t_k}^{t}\sum_{j,n}D^3u(T-t,\tilde{X}(t)).\bigl(S_{\Delta t}\sigma(X_k)e_j,S_{\Delta t}D\sigma_{e_n,S_{\Delta t}\sigma(X_k)e_j}^2(\tilde{X}(s)),S_{\Delta t}e_n\bigr)dsdt\\
&=\frac{1}{2}\E\int_{t_k}^{t_{k+1}}\int_{t_k}^{t}\sum_{j,n}D^3u(T-t,\tilde{X}(t)).\bigl(S_{\Delta t}e_j,S_{\Delta t}D\sigma_{e_n,S_{\Delta t}\sigma(X_k)^2e_j}^2(\tilde{X}(s)),S_{\Delta t}e_n\bigr)dsdt,
\end{align*}
where we haved used~\eqref{eq:dsigma}, and then~\eqref{eq:sigma_star}.

We now use Proposition~\ref{theo:D3}. Note that
\[
|\langle D\sigma_{e_n,S_{\Delta t}\sigma(X_k)^2e_j}^2(\tilde{X}(s)),h\rangle| \le C|h|_{L^{r}}|e_n|_{L^{4q}}|e_j|_{L^\infty} \le C|h|_{L^{r}},
\]
for any $h\in L^{r}$, with $\frac{1}{4q}+\frac{1}{r}=1$; thus $\bigl(\E|D\sigma_{e_n,S_{\Delta t}\sigma(X_k)^2e_j}^2(\tilde{X}(s))|_{L^{4q}}^2\bigr)^{\frac12}\le C$.

We obtain
\begin{align*}
|c_k^{3,\calD}|&\le C(1+|x|_{L^{p}})^{K} \int_{t_k}^{t_{k+1}}\int_{t_k}^{t}\frac{1}{(T-t)^{\frac12-\kappa}}\sum_{j,n}|-A)^{-\frac12+\kappa}S_{\Delta t}e_j|_{L^{4q}}|S_{\Delta t}e_n|_{L^{4q}} dsdt\\
&\le C\Delta t^{\frac12-3\kappa}(1+|x|_{L^{p}})^{K}\int_{t_k}^{t_{k+1}}\frac{1}{(T-t)^{\frac12-\kappa}}dt \sum_{n}\frac{1}{\lambda_n^{\frac12+\kappa}}\sum_j\frac{1}{\lambda_{j}^{\frac12+\kappa}}.
\end{align*}

\subsubsection*{Conclusion}
Gathering the estimates on $c_{k}^{3,\calA}$, $c_{k}^{3,\calB}$, $c_k^{3,\calC}$ and $c_{k}^{3,\calD}$, and summing for $k\in\left\{1,\ldots,N-1\right\}$, we obtain
\begin{equation}\label{eq:num_c^3}
\sum_{k=1}^{N}\big|c_{k}^3\big|\le C\Delta t^{\frac12-\kappa}(1+|x|_{L^{\max(p,2q)}})^{K+1}\int_{0}^{T}\bigl(1+\frac{1}{(T-t)^{1-\kappa}}\bigr)dt.
\end{equation}

\subsection{Conclusion of the proof of Theorem \ref{theo:num}} 

Gathering \eqref{e73}, \eqref{eq:num_a^1} and \eqref{eq:num_a^2} allows to control the sum of  $a_k$. Similarly, $b_k$ is estimated thanks to 
\eqref{e77}, \eqref{e_alpha} and \eqref{eq:num_b^2} whereas $c_k$ is estimated by \eqref{e80}, \eqref{e_beta}, \eqref{e_gamma} and \eqref{eq:num_c^3}.

Then, \eqref{eq:theo_num} follows from 
 \eqref{e72} and  \eqref{e0}.
   
\subsection{An auxiliary result}\label{sec:aux_num}

We used the estimate below for the treatment of several terms, for instance $a_{k}^{1,3}$ and $b_{k}^{2,2,3,1}$. Recall that $\D_s\tilde{X}(t)=U(t,s)S_{\Delta t}\sigma(X_\ell)$ for $t\in[t_{k},t_{k+1})$, $s\in[t_{\ell},t_{\ell+1})$, and $\ell\le k-1$.
\begin{lemma}\label{lem:U_num}
For every $q\in[2,\infty)$, $T\in(0,\infty)$ and $\kappa>0$ sufficiently small, there exists  $C_{q,\kappa}(T)$ such that for every $h\in L^q$, $t\in[t_{k},t_{k+1})$, $s\in[t_{\ell},t_{\ell+1})$, with $1\le k\le N$,
\begin{equation}
\bigl(\E|(-A)^{-\frac12+\kappa}U(t,s)h|_{L^q}^{2K}\bigr)^{\frac1{2K}}\le C_{q,\kappa}(T)\left(|(-A)^{-\frac12+\kappa}h|_{L^q}+\Delta t^{\frac12-\kappa}|h|_{L^{q}}\right) \quad \text{if}~k>\ell+1.
\end{equation}
\end{lemma}

\begin{proof}[Proof of Lemma~\ref{lem:U_num}]
Let $s$ be fixed. It can be seen that $U_t=U(t,s)h$ satisfies:
\begin{align*}
U_t&=U_{t_k}+\int_{t_k}^{t}\bigl(AS_{\Delta t}U_{t_k}+S_{\Delta t}G'(X_k).U_{t_k}\bigr)dr+\int_{t_k}^{t}S_{\Delta t}\bigl(\sigma'(X_k).U_{t_k}\bigr)dW(r),\\
U_{t_{k+1}}&=S_{\Delta t}U_{t_k}+\Delta tS_{\Delta t}G'(X_k).U_{t_k}+S_{\Delta t}\bigl(\sigma'(X_k).U_{t_k}\bigr)\Delta W_k,\\
U_{t_{\ell+1}}&=h.
\end{align*}

First, for every $t\in[t_{k},t_{k+1})$,
\begin{align*}
\E|(-A)^{-\frac12+\kappa}U_t|_{L^q}^{2}&\le C\E|(-A)^{-\frac12+\kappa}U_{t_k}|_{L^q}^{2}+C\Delta t^2\big|AS_{\Delta t}\big|_{\mathcal{L}(L^q)}^2\E|(-A)^{-\frac12+\kappa}U_{t_k}|_{L^q}^{2}\\
&+C\Delta t^{1-2\kappa} \E|U_{t_k}|_{L^q}^{2}+C\Delta t\E\big|(-A)^{-\frac12+\kappa}S_{\Delta t}\bigl(\sigma'(X_k).U_{t_k}\bigr)\big|_{R(L^2,L^q)}^2\\
&\le C\E|(-A)^{-\frac12+\kappa}U_{t_k}|_{L^q}^{2}+C\Delta t^{1-2\kappa}\E|U_{t_k}|_{L^q}^{2}+C\Delta t|(-A)^{\frac{1}{2q}-\frac12+\kappa}|_{R(L^2,L^q)}^2\E|U_{t_k}|_{L^q}^{2}.
\end{align*}
Note that $|(-A)^{\frac{1}{2q}-\frac12+\kappa}|_{R(L^2,L^q)}^2<\infty$ when $\frac{1}{2q}-\frac{1}{2}+\kappa<-\frac{1}{4}$; this condition is satisfied when $q>2$ and $\kappa>0$ is chosen sufficiently small.

The result is clear for  $k=\ell+1$. For $k>\ell+1$,  since $U_{t_k}=\Pi_{k-1:\ell+1}h$, we get, by Lemma~\ref{lem:Pi},
\[
\Delta t^{1-2\kappa} \E|U_{t_k}|_{L^q}^{2}\le \frac{C\Delta t^{1-2\kappa}}{(k-\ell-1)^{1-2\kappa}\Delta t^{1-2\kappa}}|(-A)^{-\frac12+\kappa}h|_{L^q}^{2}\le C |(-A)^{-\frac12+\kappa}h|_{L^q}^{2}.
\]

Now,
\[
U_{t_k}=S_{\Delta t}^{k-\ell-1}h+\Delta t \sum_{m=\ell+1}^{k-1}S_{\Delta t}^{k-m}BF'(X_m).U_{t_m}+\sum_{m=\ell+1}^{k-1}S_{\Delta t}^{k-m}\bigl(\sigma'(X_m).U_{t_m}\bigr)\Delta W_m,
\]
and thus, with the condition $\frac{1}{2q}-\frac{1}{2}+\kappa<-\frac{1}{4}$ fulfilled for $\kappa>0$ sufficiently small,
\begin{align*}
\E\big|(-A)^{-\frac12+\kappa}U_{t_k}|_{L^{2q}}^{2}&\le |(-A)^{-\frac12+\kappa}h|_{L^q}^{2}+C\left(\Delta t\sum_{m=\ell+1}^{k-1}\frac1{t_{k-m}^{2\kappa}}\E|U_{t_m}|_{L^q}\right)^{2}+C\Delta t\sum_{m=\ell+1}^{k-1}\E|U_{t_m}|_{L^q}^{2}\\
&\le |(-A)^{-\frac12+\kappa}h|_{L^q}^{2}+C\Delta t^2 \frac1{t_{k-\ell-1}^{4\kappa}}|h|_{L^q}^{2}+C\Delta t|h|_{L^q}^{2}\\
&+C\left(\left(\Delta t\sum_{m=\ell+2}^{k-1}\frac1{t_{k-m}^{2\kappa}}\frac{1}{t_{m-\ell-1}^{\frac12-\kappa}}\right)^2+\Delta t\sum_{m=\ell+2}^{k-1}
\frac{1}{t_{m-\ell-1}^{1-2\kappa}}\right)|(-A)^{-\frac12+2\kappa}h|_{L^{q}}^{2}\\
&\le C|(-A)^{-\frac12+2\kappa}h|_{L^q}^{2}+C\Delta t|h|_{L^q}^{2}.
\end{align*}
This concludes the proof of Lemma~\ref{lem:U_num}.

\end{proof}

\def\cprime{$'$}


\end{document}